%% file: SKTeps.tex
\documentclass[10pt,a4paper]{article}

\usepackage[utf8]{inputenc} 
\usepackage[T1]{fontenc} 

\usepackage[english]{babel}

\usepackage{amsmath}
\usepackage{amsfonts}
\usepackage{amssymb}
\usepackage{dsfont} 
\usepackage{amsthm} 

\usepackage{csquotes} 
\usepackage{graphicx} 
\usepackage{xcolor} 
\usepackage{enumitem} 
\usepackage{hyperref} 
\usepackage{verbatim} 

\usepackage[capitalise]{cleveref}
\usepackage[backend=bibtex, url = false, citestyle=alphabetic, bibstyle=alphabetic,isbn=false,maxcitenames=10, maxnames=10]{biblatex}

\DeclareSourcemap{
  \maps[datatype=bibtex]{
    \map{
      \step[fieldsource=doi,final]
      \step[fieldset=url,null]
      \step[fieldset=urldate,null]
    }
  }
}

\input{BaseTex.tex}

\author{Hector Bouton
  \thanks{Email: \href{bouton@imj-prg.fr}
    {\texttt{bouton@imj-prg.fr}}\\
  Université Paris Cité and Sorbonne Université, CNRS, IMJ-PRG, F-75013 Paris, France}}
\title{Improved convergence rates in the fast-reaction approximation of the triangular Shigesada-Kawasaki-Teramoto system}
\date{}

\begin{document}
\maketitle

\begin{abstract}
  We consider the fast-reaction approximation to the triangular Shigesada-Kawasaki-Teramoto model on a bounded domain in the physical dimension $d\ioe 3$. We provide explicit convergence rates on the whole domain in $\fsL^\infty\fsL^2\cap\fsL^2\fsH^1$ and in the interior we prove convergence with an explicit rate in any $\fsL^\infty\fsH^l$ for all $l>0$.
\end{abstract}

\section{Introduction}
This article deals with the Shigesada-Kawasaki-Teramoto (SKT) system introduced in \cite{shigesada1979} to model two species with competitive interactions. One of its main characteristics is the inclusion of cross-diffusion effects, to take into account the interaction between the two species at the level of diffusion. More precisely, we consider a triangular form of this system that writes down as
\begin{equation}\label{eq:SKT}
  \left\{
  \begin{aligned}
    \partial_t u - \Delta[ (d_u +\sigma v) u ] = f_u (u,v)u, \\
    \partial_t v - d_v\,\Delta v = f_v (u,v)v,               \\
  \end{aligned}
  \right.
\end{equation}
on a smooth bounded domain $\Omega$ of $\R^d$, $d\soe 1$. The quantities $u(t, x), v(t, x)\soe 0$ represent number densities. The parameters $d_u, d_v, \sigma > 0$ are diffusion and cross-diffusion rates. The reaction terms are given by Lotka-Volterra type terms, modelling competition (between the two species and inside each species), they write as
\begin{align}
  f_u(u,v) & = r_u - d_{11}\, u - d_{12}\, v, \\ \label{SKTge2}
  f_v(u,v) & = r_v - d_{21}\, u - d_{22}\, v,
\end{align}
where $r_u, r_v \soe 0$ are reproduction rates and $d_{ij} \soe 0$ for $i, j \in \{1, 2\}$ are parameters measuring the intensity of the competition.\medskip

We complete the system with Neumann boundary conditions
\begin{equation}\label{eq:SKTbc}
  \nabla v \cdot n(x) = 0 \quad \text{ and } \nabla [(d_1+\sigma v) u ] \cdot  n(x) = 0  \quad  \quad \text{ for } (t,x) \in \R_+ \times \partial\Omega,
\end{equation}
and initial  data:
\begin{equation}\label{eq:SKTic}
  u(0, \cdot) = u_\init , \qquad \qquad v(0, \cdot) = v_\init.
\end{equation}

The SKT system has attracted a lot of attention from the mathematical community. It is known to exhibit Turing patterns, see for instance \cite{iidamimura2006}. The presence of cross-diffusion leads to several mathematical challenges.

It is known that this system admits a global solution at least in dimension $d\ioe 4$. More precisely we recall
\begin{prop}[Theorem 1 of \cite{boutondesvillettes2025}]\label{thm:globalSKT}
  Let $d \ioe 4$ and $\Omega \subset \R^d$ be a smooth bounded domain, and let $d_u,d_v, \sigma >0$, $r_u,r_v\ge 0$, $d_{ij} > 0$ for $i,j=1,2$.  We suppose that $u_\init, v_\init \ge 0$ are initial data which lie in $\Ccal^2(\overline{\Omega})$ and are compatible with the homogeneous Neumann boundary conditions.

  Then, there exists a global (defined on $\R_+ \times \Omega$) strong (that is, all terms appearing in the equation are defined a.e.) solution to system \eqref{eq:SKT}--\eqref{eq:SKTbc}.\medskip

  More precisely, for any $T>0$ we have $u\in\fsL^\infty([0, T], \fsL^q(\Omega))\cap\fsL^2([0, T], \fsH^1(\Omega))$ and $v\in\fsL^\infty([0, T], \fsW^{1, \infty}(\Omega))\cap\fsL^q([0, T], \fsW^{2, q}(\Omega))\cap\fsW^{1, q}([0, T], \fsL^{q}(\Omega))$ for every $1 < q <\infty$.
\end{prop}

Since the seminal work of \textcite{iidamimura2006}, it is well-known that the system \eqref{eq:SKT} can be derived at the formal level as the singular limit $\eps\to 0$ of a microscopic semi-linear parabolic system on $\Omega\times[0, T]$, namely the system
\begin{equation} \label{eq:SKT-eps}
  \left\{\begin{aligned}
     & \partial_{t} u_{A}^\epsilon - d_A \,\Delta_x u_{A}^\epsilon = f_u(u_A^\eps+u_B^\eps, v^\eps)\,u_{A}^\epsilon + \frac{1}{\epsilon}\,[k(v^\epsilon)\,u_{B}^\epsilon-h(v^\epsilon)\,u_{A}^\epsilon],     \\
     & \partial_{t} u_{B}^\epsilon - (d_A+d_B)\, \Delta_x u_{B}^\epsilon = f_u(u_A^\eps+u_B^\eps, v^\eps)u_{B}^\epsilon - \frac{1}{\epsilon}\,[k(v^\epsilon)\,u_{B}^\epsilon-h(v^\epsilon)\,u_{A}^\epsilon], \\
     & \partial_{t} v^\epsilon - d_v \,\Delta_x v^\epsilon = f_v(u_A^\eps+u_B^\eps, v^\eps)v^\eps,\\
     & \nabla_x u_A^\epsilon \cdot n = \nabla_x u_B^\epsilon \cdot n = \nabla_x v^\epsilon \cdot n = 0  \quad \text{on } [0, T]\times\partial\Omega,
  \end{aligned}
  \right.
\end{equation}
where
\begin{align*}
  f_u(u_A^\eps+u_B^\eps, v^\eps) & = r_u-d_{11}(u_{A}^\epsilon+u_{B}^\epsilon)-d_{12}\,v^\epsilon,     \\
  f_v(u_A^\eps+u_B^\eps, v^\eps) & = r_v -d_{21}(u_{A}^\epsilon +u_{B}^\epsilon) - d_{22}\,v^\epsilon,
\end{align*}

the quantities $u_A^\eps, u_B^\eps, v^\eps$ are nonnegative. The coefficients $d_A, d_B > 0$ and the smooth functions $h, k \soe 0$ are such that for some constants $S, h_0 > 0$:
\begin{equation}\label{eq:hk-relation}
  d_A + d_B\frac{h(v)}{h(v)+k(v)} = d_u + \sigma v, \quad h, k > h_0, \quad\text{and} \quad h(v) + k(v) = S.
\end{equation}
The existence of $d_A, d_B, h$ and $k$ such that \eqref{eq:hk-relation} holds is provided by \cite{desvillettestrescases2015}. We quickly recall their argument. An application of the maximum principle shows that $\|v^\eps\|_{\fsL^\infty(\Omega_T)} \ioe \max(\|v_\init \|_{\fsL^\infty(\Omega_T)}, \frac{r_v}{2d_{22}}) =: A$. Then, one takes $d_A := \frac{d_u}{2}, d_B := d_u + \sigma A + 1$, $\phi$ a smooth function such that $\phi(x) = \sigma x$ for $0 \ioe x \ioe A$ and $-\frac{d_u}{4} \ioe \phi (x) \ioe \sigma A + 1$ otherwise. Then one defines $h:= \frac{d_u}{2} + \phi$ and $k := \frac{d_u}{2} + \sigma A + 1 - \phi$. One can check that the hypothesis are satisfied. We notice that on $[0, A]$, $h$ and $k$ can be taken affine, however we will only assume this hypothesis sometimes for simplicity.\medskip

Then, for every $\eps > 0$, \eqref{eq:SKT-eps} admits a unique global solution in $\Ccal^2(\R_+\times\overline\Omega)$ (see for instance \cite[Proposition 3.2]{desvillettes2007} and standard manipulations), and at the formal level, we have $u_A^\eps + u_B^\eps \to u$ and $v^\eps\to v$, where $u, v$ are given by \cref{thm:globalSKT}.\medskip

This model is also biologically meaningful. Indeed one can consider $u_A$ as the number density of the individuals of the species $u$ that are not stressed and diffuse slowly, whereas the number density $u_B$ represents the individuals, who are stressed by the presence of the individuals of the species $v$, and diffuse faster. From a mathematical perspective, understanding the emergence of cross-diffusion from semi-linear equations is also challenging, see for instance \cite{dausdesvillettes2020}.

\subsection{Main results and key ideas}
In this article, we mainly focus on quantifying the convergence of the solution to \eqref{eq:SKT-eps} toward the solution to \eqref{eq:SKT} in dimension $d\ioe 3$. More precisely, we define
\begin{equation}\label{eq:def}
  Q^\eps := k(v^\epsilon)\,u_{B}^\epsilon-h(v^\epsilon)\,u_{A}^\epsilon, \quad u^\eps := u_A^\eps + u_B^\eps, \quad U^\eps : = u^\eps - u, \quad V^\eps := v^\eps - v,
\end{equation}
and we show the
\begin{theo}\label{thm:conv}
  Let  $d\ioe 4$, $T > 0$, let $\Omega$ be a smooth  bounded domain of $\R^d$, $d_1,d_2,\sigma >0$, $r_u,r_v\ge 0$, $d_{ij} > 0$ for $1\ioe i, j \ioe 2$ and $d_A, d_B, h, k$ satisfying \eqref{eq:hk-relation}. Let $u_{A, \init}, u_{B, \init}, v_\init \soe 0$ belong to $\Ccal^{2, \alpha}(\overline\Omega)$ for some $\alpha > 0$ and compatible with the homogeneous Neumann boundary conditions. \medskip
  
  Let $u^\eps_A, u^\eps_B, v^\eps$ be the unique global strong solution (in $\Ccal^2(\R_+\times\overline\Omega)$) to \eqref{eq:SKT-eps} (with initial data $u_{A, \init}, u_{B, \init}, v_\init$), and $u, v$ be a solution to \eqref{eq:SKT} given by \cref{thm:globalSKT} (with initial data $u_\init := u_{A, \init} + u_{B, \init}, v_\init$). We recall that $U^\eps, V^\eps, Q^\eps$ are defined by \eqref{eq:def}.\medskip

  Then there exists a constant $C_0$ that depends only on the parameters of \eqref{eq:SKT-eps}, $d, T, \Omega$ and the norm of $u_{A, \init}, u_{B, \init}, v_\init$ in $\Ccal^{2, \alpha}(\overline\Omega)$, but not on $\eps$, such that
  \begin{equation*}
    \|U^\eps\|_{\fsL^\infty([0, T], \fsL^2(\Omega))} + \|U^\eps\|_{\fsL^2([0, T], \fsH^1(\Omega))} \ioe C_0\eps^\frac12 ,
  \end{equation*}
  and 
  \begin{equation*}
    \|V^\eps\|_{\fsL^\infty([0, T], \fsL^2(\Omega))} + \|V^\eps\|_{\fsL^2([0, T], \fsH^1(\Omega))} \ioe C_0\eps^\frac12.
  \end{equation*}

  Furthermore we also have an improved rate of convergence
    \begin{equation*}
    \|U^\eps\|_{\fsL^2([0, T], \fsL^2(\Omega))}^2 + \|V^\eps\|_{\fsL^2([0, T], \fsL^2(\Omega))}^2 \ioe C_0\left(\int_\Omega |Q^\eps(0)|^2\eps + \eps^2\right).
  \end{equation*}

  Let $l \soe 0$, we furthermore assume that $d\ioe 3$ and that $u_{A, \init}, u_{B, \init}, v_\init$ are in $\Ccal^\infty(\overline{\Omega})$ and compatible with the homogeneous Neumann boundary conditions. We also assume that $h, k$ are affine on $[0, \sup_{\eps > 0} \|v^\eps\|_{\fsL^\infty(\Omega_T)}]$. Let $\tilde\Omega \Subset \Omega$, and $\eps_{\init, l}^2 := \sup_{k\ioe l + 1} \int_\Omega |\nabla^{\otimes k} Q^\eps(0)|^2$. Then there exists a constant $C_l$ that depends only on the parameters of \eqref{eq:SKT-eps}, $d, T, l, \tilde\Omega$ and the norm of the initial data $u_{A, \init}, u_{B, \init}, v_\init$ in a Sobolev space $\fsH^k(\Omega)$ for $k$ large enough, but not on $\eps$, such that
  \begin{equation*}
    \|U^\eps\|_{\fsL^\infty([0, T], \fsH^l(\tilde{\Omega}))} + \|V^\eps\|_{\fsL^\infty([0, T], \fsH^l(\tilde{\Omega}))}
    \ioe C_l(\eps_{\init, l}\eps^\frac12 + \eps).
  \end{equation*}
\end{theo}

\begin{rema}
  We remark that the bound $\|U^\eps\|_{\fsL^2([0, T], \fsL^2(\Omega))}^2 \lesssim \int_\Omega |Q^\eps(0)|^2\eps + \eps^2$ can be written as $\|U^\eps\|_{\fsL^2([0, T], \fsL^2(\Omega))}^2 \lesssim \eps_{\init, -1}^2\eps + \eps^2$, where the notation $\eps_{\init, -1}$ is consistent with the definition of $\eps_{\init, l}$.  
\end{rema}

\begin{rema}
  The same result also holds if we only assume that $v_\init := v_\init^\eps$ and $u_\init := u_\init^\eps := u_{A,\init}^\eps + u_{B,\init}^\eps$ are independent of $\eps$, and that $u_{A, \init}^\eps, u_{B, \init}^\eps, v_\init^\eps$ are uniformly bounded in $\eps$ in $\Ccal^{2, \alpha}(\overline\Omega)$ (respectively in $\fsH^{k}(\Omega)$ for $k$ large enough).
\end{rema}
 
\begin{rema}
If we work with a fixed $l$, one can assume less regularity on the initial data $u_{A, \init}, u_{B, \init}, v_\init$. Moreover in the first part of the Proposition, the assumption $u_{A, \init}, u_{B, \init}, v_\init\in \Ccal^{2, \alpha}(\overline\Omega)$ is only made to have a simpler proof of a sufficient regularity on the limit problem. Finer arguments should allow to reduce it to a bound in $\Ccal^2(\overline\Omega).$
\end{rema}

\begin{rema}
  Since global solutions of \eqref{eq:SKT-eps} are unique in $\Ccal^2(\R_+\times\overline\Omega)$, we obtain as a by-product of \cref{thm:conv} a uniqueness result for \eqref{eq:SKT} amongst the class of solutions satisfying the regularity estimates in the conclusion of \cref{thm:globalSKT}.
\end{rema}

\begin{rema}
  The assumption that $h, k$ should be affine on $[0, \sup_{\eps > 0} \|v^\eps\|_{\fsL^\infty(\Omega_T)}]$ is only made for simplicity and can be removed thanks to the use of the Faa Di Bruno formula to control the derivatives of $h(v^\eps)$ and $k(v^\eps)$.
\end{rema}

We highlight that we neither assume any smallness condition on the cross-diffusion $d_B$ nor any extra a priori bound on $U^\eps$ or $V^\eps$. To prove this result, we need to show that $u^\eps$ and $v^\eps$ enjoy some uniform in $\eps$ space regularity, hence we first prove the

\begin{prop}\label{thm:SKTeps-regu}
  Let $T>0$, $d\ioe 4$, $\Omega$ be a smooth  bounded domain of $\R^d$,
  $d_1,d_2,\sigma >0$, $r_u,r_v\ge 0$, $d_{ij} > 0$ and $d_A, d_B, h, k$ satisfying \eqref{eq:hk-relation}. Let $u_{A, \init}, u_{B, \init}, v_\init \soe 0$ belong to $\Ccal^2(\overline\Omega)$ and compatible with the homogeneous Neumann boundary conditions. Let $u^\eps_A, u^\eps_B, v^\eps$ be the unique global strong solution (in $\Ccal^2(\R_+\times\overline\Omega)$) to \eqref{eq:SKT-eps} (with initial data $u_{A, \init}, u_{B, \init}, v_\init$).

  Then there exists a constant $C_0$ that depends only on the parameters of \eqref{eq:SKT-eps}, $d, T, \Omega$ and the norm of $u_{A, \init}, u_{B, \init}, v_\init$ in $\Ccal^2(\overline\Omega)$, but not on $\eps$, such that
  \begin{equation*}
    \|\nabla u_A^\eps \|_{\fsL^{\infty}([0, T], \fsL^{2}(\Omega))} + \|\nabla u_B^\eps \|_{\fsL^{\infty}([0, T], \fsL^{2}(\Omega))}\ioe C_0,
  \end{equation*}
  \begin{equation*}
    \|u_A^\eps\|_{\fsL^2([0, T], \fsH^2(\Omega))} + \|u_B^\eps\|_{\fsL^2([0, T], \fsH^2(\Omega))} \ioe C_0,
  \end{equation*}
  and
  \begin{equation}\label{eq:estimateQ}
    \int_0^T\int_\Omega \frac{1}{\eps}|\nabla Q^\eps|^2 \ioe C_0.
  \end{equation}\\

  Let $\tilde\Omega\Subset\Omega$ and $l > 0$, we furthermore assume that $d\ioe 3$, that $u_{A, \init}, u_{B, \init}$ belong to $\fsH^{l}(\Omega)$, $v_\init$ belongs to $\fsW^{l+2, \infty}(\Omega)$ and that $h, k$ are affine on $[0, \sup_{\eps > 0} \|v^\eps\|_{\fsL^\infty(\Omega_T)}]$. Then there exists a constant $C_l$ that depends only on the parameters of \eqref{eq:SKT-eps}, $d, T, l, \tilde\Omega$ and the norm of $u_{A, \init}, u_{B, \init}$ in $\Ccal^2(\overline\Omega)\cap \fsH^{l}(\Omega)$, the norm of $v_\init$ in $\fsW^{l+2, \infty}(\Omega)$, but not on $\eps$, such that
   \begin{align*}
    \|\nabla^{\otimes l}u_A^\eps\|_{\fsL^\infty([0, T], \fsL^2({\tilde\Omega}))} + \|\nabla^{\otimes (l+1)}u_A^\eps\|_{\fsL^2({\tilde\Omega\times[0, T]})} \ioe C_l,\\
    \|\nabla^{\otimes l}u_B^\eps\|_{\fsL^\infty([0, T], \fsL^2({\tilde\Omega}))} + \|\nabla^{\otimes (l+1)}u_B^\eps\|_{\fsL^2({\tilde\Omega\times[0, T]})} \ioe C_l,\\
    \|\nabla^{\otimes (l - 1)}v^\eps\|_{\fsL^{\infty}(\tilde\Omega\times[0, T])} + \|\nabla^{\otimes l}v^\eps\|_{\fsL^{10}(\tilde\Omega\times[0, T])}\ioe C_l,\\
    \|\partial_t\nabla^{\otimes (l-1)}v^\eps\|_{\fsL^{\frac{10}{3}}(\tilde\Omega\times[0, T])} + \|\nabla^{\otimes (l+1)}v^\eps\|_{\fsL^\frac{10}{3}(\tilde\Omega\times[0, T])}\ioe C_l,
  \end{align*}
  and
  \begin{equation*}
    \frac{1}{\eps}\int_0^t\int_{\tilde\Omega} |\nabla^{\otimes l} Q|^2 \ioe C_l.
  \end{equation*}
\end{prop}

\begin{rema}
The dimension $d=4$ is the critical dimension in our proof, for instance the regularity estimate $\|\nabla^{\otimes (l - 1)}v^\eps\|_{\fsL^{\infty}(\tilde\Omega\times[0, T])} \ioe C$ does not hold anymore in dimension $d=4$ and we have to deal with the weaker $\|\nabla^{\otimes (l - 1)}v^\eps\|_{\fsL^{q}(\tilde\Omega\times[0, T])} \ioe C_q$ for all $q<\infty$. However, it is probably possible to consider the dimension $d=4$ thanks to a finer analysis.
\end{rema}

In the proof of \cref{thm:SKTeps-regu}, our main idea consists in introducing new energy functionals
\begin{align*}
  \Ecal^\eps_A(t):= \int_\Omega h(v^\eps)|\nabla u^\eps_A - \frac{(u^\eps_A+u^\eps_B)}{S}\nabla (k(v^\eps))|^2,\\
  \Ecal^\eps_B(t):= \int_\Omega k(v^\eps)|\nabla u^\eps_B - \frac{(u^\eps_A+u^\eps_B)}{S}\nabla (h(v^\eps))|^2.
\end{align*}
We show that this energy and higher-order variants of itself can be controlled.\\

Those regularity estimates allow to understand the precise behavior of the solution to \eqref{eq:SKT-eps}. As an example, one can consider the apparition of an initial layer, that becomes small after a time $O(\eps|\ln\eps|)$. More precisely, one shows the

\begin{coro}\label{thm:initiallayer}
  Let  $d\ioe 4$, $T > 0$, $\Omega$ be a smooth  bounded domain of $\R^d$, $d_1,d_2,\sigma >0$, $r_u,r_v\ge 0$, $d_{ij} > 0$ for $1\ioe i, j \ioe 2$ and $d_A, d_B, h, k$ satisfying \eqref{eq:hk-relation}. Let $u_{A, \init}, u_{B, \init}, v_\init \soe 0$ belong to $\Ccal^2(\overline\Omega)$ and compatible with the homogeneous Neumann boundary conditions. \medskip
  
  Let $u^\eps_A, u^\eps_B, v^\eps$ be the unique global strong solution (in $\Ccal^2(\R_+\times\overline\Omega)$) to \eqref{eq:SKT-eps} (with initial data $u_{A, \init}, u_{B, \init}, v_\init$).\medskip 

  Then there exists a constant $C_0$ that depends only on the parameters of \eqref{eq:SKT-eps}, $d, T, \Omega$ and the norm of $u_{A, \init}, u_{B, \init}, v_\init$ in $\Ccal^2(\overline\Omega)$, but not on $\eps$, such that for $t_\eps :=\frac{\eps\lvert\ln\eps\rvert}{2S}$
  \begin{equation*}
    \|Q^\eps(t_\eps)\|_{\fsL^2(\Omega)}\ioe C_0\eps^{\frac 12}.
  \end{equation*}

  Let $\tilde\Omega\Subset\Omega$, we furthermore assume that $d\ioe 3$, that $v_\init$ belongs to $\fsW^{4, \infty}(\Omega)$ and that $h, k$ are affine on $[0, \sup_{\eps > 0} \|v^\eps\|_{\fsL^\infty(\Omega_T)}]$. Then there exists a constant $C_1$ that depends only on the parameters of \eqref{eq:SKT-eps}, $d, T, \tilde\Omega$ and the norm of $u_{A, \init}, u_{B, \init}$ in $\Ccal^2(\overline\Omega)$, the norm of $v_\init$ in $\fsW^{4, \infty}(\Omega)$, but not on $\eps$, such that for $t'_\eps := \frac{\eps\lvert\ln\eps^2\rvert}{2S}$, we have
  \begin{equation*}
    \|Q^\eps(t'_\eps)\|_{\fsL^2(\tilde\Omega)} \ioe C_1\eps \lvert\ln\eps\rvert^\frac12.
  \end{equation*}
\end{coro}

\subsection{Review of the literature}
The derivation of various cross-diffusion systems from microscopic systems has attracted a lot of attention, we refer to \cite{soresina2023,tang2024,desvillettes2025a} for recent results in chemical or prey-predator models that study quantitative convergence rates.\\
Concerning the triangular SKT system, \textcite{iidamimura2006} show convergence of $u_A^\eps + u_B^\eps$ towards $u$ at rate $O(\eps)$ in $\fsL^2(\Omega_T)$ provided the initial data are well prepared (in the sense that $Q^\eps(t = 0) = 0$), under the assumption that the cross-diffusion is small and assuming some bounds on $u^\eps$. \textcite{confortodesvillettes2014} and later \textcite{desvillettestrescases2015} also consider this derivation but they do not quantify the rate of convergence and they assume some restrictive hypothesis either on the dimension $d=1$ or on $f_u, f_v$.
\textcite{morgansoresina2026} consider a predator-prey system that shares a lot of mathematical similarity with \eqref{eq:SKT}. In dimension $d=2$ and under some minor assumptions on the coefficients of their model, they derive a convergence result in $\fsL^2([0, T], \fsL^2(\Omega))$ at rate $O(\eps^\frac 12)$. \textcite{brocchieri2025} consider a variant of \eqref{eq:SKT-eps}, where $h, k$ are replaced by some functions $\psi(u_A^\eps, u_B^\eps, v^\eps), \phi(u_A^\eps, u_B^\eps, v^\eps)$ of prescribed form. Using the explicit dependency of $\psi, \phi$ with respect to $u_A^\eps$ and $u_B^\eps$, they derive a convergence result in $\fsL^\infty([0, T], \fsL^2(\Omega))$ and in $\fsL^2([0, T], \fsH^1(\Omega))$ under the hypothesis that $u^\eps$ is uniformly bounded in $\fsL^\infty([0, T], \fsL^\infty(\Omega))$ and under a smallness assumption on $d_B$. They provide a rate of convergence $O(\eps^\frac 12\eps_{\init, 0} + \eps^2)$ in the whole domain.\\
Variants of the SKT system have also been considered, for instance \cite{confortodesvillettes2018} and \cite{desvillettesfiorentino2025} consider different logistic terms, but no explicit convergence rate is provided. We also refer to \cite{murakawa2012} and \cite{desvillettessoresina2019} where fast-reaction approximations of non-triangular predator-prey or competitive systems are considered.

\subsection{Notations}
We introduce here some notations that will be used throughout the whole article. For a domain $\Omega$ and some $T > 0$, we denote $\Omega_T := \Omega \times[0, T]$. We also denote $\N^* := \N\setminus\{0\}$. For some bounded domain $Q$ and $1\ioe p < + \infty$, we denote $\fsL^{p+}(Q) := \cup_{q > p} \fsL^q(Q)$, similarly for $1\ioe p \ioe + \infty$, we define $\fsL^{p-}(Q) := \cap_{q < p} \fsL^q(Q)$.\\

We will write $u^\eps_C$ for either $u^\eps_A$ or $u^\eps_B$. We introduce $u_A := \frac{k(v)}{S}u$ and $u_B := \frac{h(v)}{S}u$ (hence $u_A + u_B = u$). We also introduce $\tilde h = h(v)$, $\tilde k = k(v)$, and $\tilde h^\eps = h(v^\eps)$, $\tilde k^\eps = k(v^\eps)$.

\subsection{Organization of the paper}
In \cref{sec:regularity} we prove \cref{thm:SKTeps-regu}. We divide the proof in two parts: we first consider the rate of convergence $O(\eps^\frac 12)$ for $l=0$ and then we deal with the improved rate of convergence $\eps_{\init, l}\eps^\frac12 + \eps$. Those two proofs are very similar, but they are presented independently since some arguments are different. In \cref{sec:convergence}, we prove \cref{thm:conv}, we have similarly separated the proof in two cases. Finally, we prove \cref{thm:initiallayer} in \cref{sec:initial}. \medskip

\section{Regularity estimates : proof of \texorpdfstring{\cref{thm:SKTeps-regu}}{Theorem 2}}\label{sec:regularity}
\subsection{Initial regularity}
We first start by showing that the estimates on $u^\eps$ and $v^\eps$ obtained in \cite{boutondesvillettes2025} hold uniformly in $\eps$. More precisely we show the

\begin{prop}\label{thm:DGNM-bound}
  Let $T>0$, $d\ioe 4$, $\Omega$ be a smooth bounded domain of $\R^d$, $d_1, d_2, \sigma >0$, $r_u,r_v\ge 0$, $d_{ij} > 0$ and $d_A, d_B, h, k$ satisfying \eqref{eq:hk-relation}. Let $u_{A, \init}, u_{B, \init}, v_\init \soe 0$ belong to $W^{2, \infty}(\Omega)$ and compatible with the homogeneous Neumann boundary conditions.

  Then, there exists a strong solution $u_A^\eps, u_B^\eps$ and $v^\eps$ to \eqref{eq:SKT-eps} on $\Omega_T$ such that
  \begin{align}
     & \forall 1 \ioe q < +\infty, \exists C_q >0, \|u^\eps_A\|_{\fsL^{\infty}([0, T], \fsL^q(\Omega))}, \|u^\eps_B\|_{\fsL^{\infty}([0, T], \fsL^q(\Omega))} \ioe C_q,\label{eq:initial-estimates-uLq} \\
     & \forall 1 \ioe p < + \infty, \exists C_p >0, \|\nabla((u^\eps_A)^p)\|_{\fsL^2(\Omega_T)}, \|\nabla((u^\eps_B)^p)\|_{\fsL^2(\Omega_T)} \ioe C_p, \label{eq:initial-estimates-nablau}                  \\
     & \forall 1 \ioe q < +\infty, \exists C_q>0, \|\partial_t v^\eps\|_{\fsL^q(\Omega_T)}, \|\nabla^2 v^\eps\|_{\fsL^q(\Omega_T)} \ioe C_q,\label{eq:initial-estimates-v}                              \\
     & \exists C>0, \|v^\eps\|_{\fsL^\infty(\Omega_T)}, \|\nabla v^\eps\|_{\fsL^\infty(\Omega_T)} \ioe C.\label{eq:initial-estimates-vinf}
  \end{align}
  All the constants above are independent of $\eps$.
\end{prop}

To prove those estimates, we rely on a crucial entropy estimate

\begin{lemm}[Direct consequence of the proof of Lemma 3 from \cite{desvillettestrescases2015}]\label{thm:entropy-estimate}
  Let $\Omega$ be a $\Ccal^2$ bounded domain of $\R^d$, $d_1,d_2,\sigma >0$, $r_u,r_v\ge 0$, $d_{ij} > 0$, for $i,j=1,2$. Let $p>1$, let $u_{A, \init}, u_{B, \init} \soe 0$ in $\fsL^{p}(\Omega)$, let $v_\init\soe 0$ in $\fsL^\infty(\Omega)$. Then there exists a constant $C_{T, p}$ that may depend on $p, T, \Omega, \|v_\init\|_{\fsL^\infty}$, but that is independent of $\eps$, such that
  \begin{align*}
     & \|u_A^\epsilon\|_{\fsL^\infty([0, T], \fsL^p(\Omega))}+ \|u_B^\epsilon\|_{\fsL^\infty([0, T], \fsL^p(\Omega))}                                                      \\
     & + \|\nabla((u_A^\eps)^\frac p2)\|_{\fsL^2(\Omega_T)}^2+ \|\nabla((u_B^\eps)^\frac p2)\|_{\fsL^2(\Omega_T)}^2                                                        \\
     & +\frac{1}{\epsilon}\| (h(v^\epsilon)u_A^\epsilon)^{p/2} - (k(v^\epsilon)u_B^\epsilon)^{p/2} \|_{\fsL^2(\Omega_T)}^2                                                 \\
     & \ioe C_{T, p}(1 + \|u_{A, \init}^\eps\|_{\fsL^p(\Omega)} + \|u_{B, \init}^\eps\|_{\fsL^p(\Omega)}+ \|u_A^\eps + u_B^\eps\|_ {\fsL^{p + 1}(\Omega_T)}^{p + 1}).
  \end{align*}
\end{lemm}

Then, the proof uses a De Giorgi-Nash-Moser like argument. We consider the equation (with homogeneous Neumann boundary conditions and initial data $w_\init$)
\begin{equation}
  \label{eq:rough-scalar-pde}
  a(t,x) \partial_t w(t,x) - \Delta w(t,x) = f(t,x)
  \qquad \text{assuming moreover that }  \quad \partial_t w \ge 0,
\end{equation}
where \(a\) is a ``rough'' coefficient, i.e., we only suppose that
\begin{equation}
  \label{eq:assumption-a}
  0 < a_0 \le a \le c_0 a_0 < \infty ,
\end{equation}
for some constants \(a_0\) and \(c_0 \ge 1\).  
\par 
\begin{prop}[Theorem 4 of \cite{boutondesvillettes2025}]\label{thm:DGNM}
  We consider a bounded, $\Ccal^2$ domain \(\Omega \subset \R^d\).  Set \(T>0\), and
  \(p, q \in [1, \infty]\) such that $\gamma := 2 - \frac2p - \frac{d}q > 0$.  Then there
  exists a constant \(\alpha >0\) only depending on \(\gamma, d,c_0\), and a constant
  \(C_*\) depending on \(p,q,d, \Omega, T, a_0,c_0\) such that for any Lipschitz
  initial data $w_\init$, forcing data \(f \in \fsL^{p}((0,T]; \fsL^q(\Omega))\) and a
  coefficient \(a := a(t,x)\) satisfying the bound \eqref{eq:assumption-a}, a
  solution \(w \ge 0\) of \eqref{eq:rough-scalar-pde} over \((0,T] \times \Omega\) with
  homogeneous Neumann boundary data ($\vec n \cdot \nabla_x w = 0$ on $ (0,T] \times \partial\Omega$),
  lies in \(\Ccal^{0, \alpha}({\overline{\Omega}} \times[0, T]) \).  Moreover, the following
  estimate holds:
  \begin{equation*}
    \| w \|_{\Ccal^{0, \alpha}({\overline{\Omega}}\times [0,T] ) }
    \le C_* \left(
    \| f \|_{\fsL^p((0,T]; \fsL^{q}(\Omega))}
    + \| w_\init \|_{\mathrm{Lip}}
    \right).
  \end{equation*}
\end{prop}

In order to use the information given by the two previous propositions, we will need the following interpolation estimate.

\begin{prop}[Proposition 9 of \cite{boutondesvillettes2025}]\label{thm:interpolation}
  Let $\Omega \subset \R^d$ be a bounded \(\Ccal^2\) domain, $\alpha \in (0,1)$, and assume  $p,q \in [1, \infty)$ satisfying $q \ge \frac{3 - \alpha}{2 - \alpha}\, p > \frac{d}{2 - \alpha}$. Then for any $u,w : \Omega \to \R$   such that $0 \le u \le\Delta w$, it holds that
  \begin{equation}\label{eq:interpolation}
    \|u \|_{\fsL^q(\Omega)}
    \le  C_{d,p,q, \alpha}\, \| {w} \|_{\Ccal^{0, \alpha}(\overline{\Omega}) }^{1-\frac{2-\alpha - d/q}{3- \alpha - d/p}}\,
    \| \nabla u\|_{\fsL^p(\Omega)}^{\frac{2- \alpha - d/q}{3- \alpha - d/p}}
    +  C_{d,p,q, \alpha}\, \| {w} \|_{\Ccal^{0, \alpha}(\overline{\Omega})},
  \end{equation}
  where $C_{d,p,q, \alpha}>0$ is a constant depending only on $d,p,q, \alpha$.
\end{prop}

\begin{proof}[Proof of \cref{thm:DGNM-bound}]
  The existence of strong solutions $u_A^\eps, u_B^\eps, v^\eps$ is provided by standard theorems for reaction-diffusion equations (cf. \cite{desvillettestrescases2015}). We are interested in the uniform bounds and we show that \cite{boutondesvillettes2025} provides them.\\
  We start by showing that $u_A^\eps$ and $u_B^\eps$ are bounded in $\fsL^{3+}(\Omega_T)$. \medskip

  We recall that $u^\eps := u_A^\eps + u_B^\eps$. We introduce the auxiliary quantity $m^\eps:= m^\eps(t,x)$ defined as the solution to
  \begin{equation*}
    \left\{
    \begin{aligned}
       & \partial_t m^\eps - \Delta m^\eps = u^\eps\,(d_{11}u^\eps + d_{12}v^\eps), \\
       & \nabla m^\eps\cdot \vec n(x) = 0  \quad \text{on } [0, T]\times\partial\Omega, \\
       & m^\eps(0,\cdot) = 0.
    \end{aligned}
    \right.
  \end{equation*}
  By the minimum principle, it is clear that $m^\eps\ge 0$.\\

  Defining $\nu^\eps := \frac{d_A u^\eps_A + (d_A + d_B) u_B^\eps +m^\eps}{u^\eps_A + u_B^\eps + m^\eps} $, we observe that $\min(1,d_A) \le \nu \le \max(1, d_A + d_B)$, and $u^\eps+ m^\eps$ satisfies the equation
  \begin{equation}\label{eq:initial-um}
    \left\{
    \begin{aligned}
       & \partial_t(u^\eps+m^\eps) - \Delta ( \nu^\eps \, (u^\eps+m^\eps) ) = r_u\,u^\eps \ioe r_u(u^\eps + m^\eps),          \\
       & \nabla (u^\eps + m^\eps) \cdot \vec n(x) = 0  \quad \text{on } [0, T]\times\partial\Omega, \\
       & (u^\eps+m^\eps)(0,x) \equiv u^\eps_{\init}(x).
    \end{aligned}
    \right.
  \end{equation}

  An immediate adaptation of the improved duality lemma \cite[Proposition 1.1] {canizo2014}shows that for some $q > 2$ that depends on $d_A$ and $d_B$
  \begin{equation*}
    \|u^\eps + m^\eps\|_{\fsL^{q}(\Omega\times[0, T])} \ioe C,
  \end{equation*}
  uniformly in $\eps$. By the non-negativity of $u^\eps$ and $m^\eps$, we obtain
  \begin{equation} \label{eq:dualityum}
    u^\eps, m^\eps\in \fsL^{2+}(\Omega\times[0, T]).
  \end{equation}

  We can also integrate \eqref{eq:initial-um}, to deduce that
  \begin{equation*}
    \int_\Omega u^\eps(t)+ m^\eps(t) \ioe r_u \int_0^t\int_\Omega u^\eps + m^\eps.
  \end{equation*}
  Hence, by Gronwall's lemma we obtain that $u^\eps, m^\eps \in \fsL^\infty([0, T], \fsL^1(\Omega))$.\medskip 

  We now consider the quantity $w^\eps := \int_0^t (d_A u^\eps_A + (d_A + d_B) u_B^\eps +m^\eps)$, and observe that $w^\eps\ge 0$ and $\partial_t w^\eps \ge 0$. Moreover $\Delta w^\eps = \int_0^t \Delta (d_A u^\eps_A + (d_A + d_B) u_B^\eps +m^\eps) = \int_0^t [\partial_t (u^\eps + m^\eps) - r_u\,u^\eps] = u^\eps + m^\eps - u^\eps_{in} - r_u \,\int_0^t u^\eps$.

  It satisfies therefore the parabolic equation
  \begin{equation*}
    (\nu^\eps)^{-1}\, \partial_t w^\eps - \Delta w^\eps =  u^\eps_{\init} +    r_u \,\int_0^t u^\eps,
  \end{equation*}
  together with the homogeneous Neumann boundary conditions, and the initial condition $w^\eps(0, \cdot) = 0$.\\

  Observing that thanks to \eqref{eq:dualityum}, $\int_0^t u^\eps$ is bounded in $\fsL^{\infty}([0,T]; \fsL^{2+}(\Omega))$, we see that we can use \cref{thm:DGNM} with $p =\infty$ and $q= 2+$ (recalling that $d \le 4$), and deduce from it that $\|w^\eps\|_{\Ccal^{0,\alpha}(\overline{\Omega}\times[0, T])} \le C$ (for some $\alpha, C>0$, independent of $\eps$).\\

  Then, we use the estimate
  \begin{align*}
  0 &\le u^\eps\\
  &\le u^\eps+m^\eps\\
  &= \Delta w^\eps  +  u^\eps_{\init} +    r_u \,\int_0^t u^\eps \le \Delta w^\eps  +  \|u^\eps_{\init}\|_{\infty} +    \frac{r_u}{|\Omega|}
    \,\int_0^t \int_{\Omega} u^\eps  + r_u\, \int_0^t \bigg[ u^\eps -  |\Omega|^{-1}  \int_{\Omega} u \bigg]\\
  &\le  \Delta \tilde{w^\eps},
    \end{align*}
  where
  \begin{equation*}
    \tilde{w^\eps} := w^\eps +  \frac{|x|^2}{2d}\, \bigg(\|u_{\init}^\eps\|_{\infty} +  \frac{r_u\, T}{|\Omega|} \,\|u^\eps\|_{L^{\infty}([0,T] ; L^1(\Omega))} \bigg) + r_u \, \Delta^{-1} \int_0^t \bigg[ u^\eps -  |\Omega|^{-1}  \int_{\Omega} u \bigg],
  \end{equation*}
  and  $\Delta^{-1}$ is defined as the operator going from the subset of $\fsL^2(\Omega)$ consisting of functions with $0$-mean value towards itself, which to a function associates the (unique) solution of the Poisson equation with (homogeneous) Neumann boundary conditions.
  \medskip

  Observing that $ \int_0^t \bigg[ u^\eps - |\Omega|^{-1} \int_{\Omega} u^\eps \bigg]$ is bounded in $\fsL^{\infty}([0,T]; L^{2+}(\Omega))$, we see that $\Delta^{-1} \int_0^t \bigg[ u^\eps - |\Omega|^{-1} \int_{\Omega} u^\eps \bigg]$ is bounded in $\fsL^{\infty}([0,T]; \fsW^{2,2+\delta}(\Omega))$ for some $\delta>0$ and therefore, in dimension $d \le 4$, thanks to a Sobolev embedding, in $\fsL^{\infty}([0,T]; \Ccal^{0,\alpha}(\overline{\Omega}))$, for some $\alpha >0$. The same holds for the function $(t,x) \mapsto \frac{|x|2}{2d}$.\\

  We now use  the one-sided interpolation \eqref{eq:interpolation} of \cref{thm:interpolation}, with $d=4$, $p=2$ and $q= 2\,\frac{3 - \alpha}{2-\alpha}$, that is
  \begin{equation*}
    \|u^\eps \|_{\fsL^{2\,\frac{3 - \alpha}{2-\alpha}} (\Omega)}^3
    \le   C\, \bigg( \,  \| \tilde{w}^\eps \|_{\Ccal^{0, \alpha}(\overline{\Omega})}^{\frac3{3-\alpha}}\,
    \| \nabla u^\eps\|_{\fsL^2(\Omega)}^{3\frac{2 -\alpha}{3-\alpha}} +
    \| \tilde{w}^\eps \|_{\Ccal^{0, \alpha}(\overline{\Omega})}^{3} \,\bigg)   ,
  \end{equation*}
  which holds for any $u^\eps, \tilde{w}^\eps : \Omega \subset \R^d \to \R$, such that
  $0 \le u^\eps \le\Delta \tilde{w}^\eps$, and $\alpha \in (0,1)$. Here we use it for a given time $t\in [0,T]$.

  Recalling \cref{thm:entropy-estimate} for $p=1$, we get that
  \begin{equation*}
    \begin{aligned}
      \|u^\eps\|_{\fsL^{2\,\frac{3 - \alpha}{2-\alpha}} (\Omega\times[0, T])}^{{2\,\frac{3 - \alpha}{2-\alpha}} }
      &\le C \, \int_0^T \bigg(\,  \| \tilde{w}^\eps \|_{\Ccal^{0, \alpha} (\overline{ \Omega})}^{\frac{2}{2 - \alpha}} \, \, \| \nabla u^\eps\|_{\fsL^2(\Omega)}^2 +    \| \tilde{w}^\eps \|_{\Ccal^{0, \alpha} ( \overline{\Omega})}^{\frac{3-\alpha}{1 - \alpha/2}} \bigg) \\
      & \le C (\| \nabla u^\eps\|_{\fsL^2(\Omega\times[0, T])}^{2} + 1)\\
      &\le C   +   C \, \|u^\eps\|_{\fsL^3(\Omega\times[0, T])}^3.
    \end{aligned}
  \end{equation*}
  Since ${2\,\frac{3 - \alpha}{2-\alpha}} > 3$, we see that $u^\eps$ is bounded in $\fsL^{3+}(\Omega\times[0, T])$.\medskip

  We can then upgrade this estimate to \eqref{eq:initial-estimates-uLq}. Let $p>0$, using \cref{thm:entropy-estimate}, we see that if $u^\eps$ is bounded in $\fsL^{p+1}$ then $u_C^\eps$ is bounded in $\fsL^\infty([0, T], \fsL^p(\Omega))$ and $\nabla((u_C^\eps)^\frac p2)$ is bounded in $\fsL^2([0, T], L^2(\Omega))$. By a Sobolev injection, we deduce that $u_C^\eps$ is bounded in $\fsL^\infty([0, T], \fsL^p(\Omega))\cap\fsL^2([0, T], \fsL^{\frac{pd}{d-2}}(\Omega))$ (if $d\ioe 2$, we can replace $\frac{2d}{d-2}$ by any $q < +\infty$). Thus by interpolation, we deduce that $u_C^\eps$ is bounded in $\fsL^{p\frac{d + 2}{d}}(\Omega_T)$. Hence, we can improve the regularity, provided $p\frac{d + 2}{d} > p + 1$ which is equivalent to $p > \frac{d}{2}$. Since we know that $u^\eps\in\fsL^{p+1}(\Omega_T)$ with $p=2+$, we can iterate this procedure and deduce \eqref{eq:initial-estimates-uLq}.\medskip 

  Using this estimate and \cref{thm:entropy-estimate}, we also deduce \eqref{eq:initial-estimates-nablau}.\medskip

  For the bounds on $v^\eps$, we first notice that the maximum principle implies that $v^\eps$ is bounded in $\fsL^\infty(\Omega_T)$. Using \eqref{eq:initial-estimates-uLq} for $q$ large enough and a maximal regularity argument, we deduce \eqref{eq:initial-estimates-v}. The bound $\nabla v^\eps\in\fsL^\infty(\Omega_T)$ in \eqref{eq:initial-estimates-vinf} comes from classical properties of the heat equation. For instance, we deduce form $\partial_t v^\eps - d_v \Delta v^\eps = f_v^\eps v^\eps \in\fsL^{\infty-}(\Omega)$, that by maximal regularity, we have for any $q<+\infty$, $\nabla v^\eps \in \fsW^{0, q}([0, T], \fsW^{1, q}(\Omega)) \cap \fsW^{1, q}([0, T], \fsW^{-1, q}(\Omega))$. By an interpolation argument, we deduce that for any $0 < s < 1$, we have $\nabla v^\eps \in \fsW^{(1-s), q}([0, T], \fsW^{2s - 1, q}(\Omega))$. We take any $\frac 12 < s < 1$ and $q < + \infty$ large enough to deduce by a Sobolev embedding the sought conclusion.
\end{proof}

\subsection{Part I: estimate in \texorpdfstring{$\fsL^\infty\fsH^1$ and $\fsL^2\fsH^2$}{LinftyH1 and L2H2}}\label{sec:firstconvergence}
In this section, we prove the first part of \cref{thm:SKTeps-regu}, namely the bounds in $\fsL^\infty([0, T], \fsH^1(\Omega))$ and in $\fsL^2([0, T], \fsH^2(\Omega))$. However, we expect $u_A^\eps$ to relax to $\frac{\tilde{k}}{S}u$. Thus even if $u^\eps_A+u^\eps_B$ is constant at the beginning, we formally expect a gradient $\nabla u_A = \nabla(\frac{\tilde{k}}{S}(u_A + u_B))$ to be created almost instantaneously, if $v$ (and thus $\tilde{k}$) is not homogeneous at the beginning. That is the reason why we will rather control the following quantities:
\begin{equation*}
  \Ecal^\eps_A(t):= \int_\Omega \tilde{h}^\eps|\nabla u^\eps_A - \frac{(u^\eps_A+u^\eps_B)}{S}\nabla \tilde{k}^\eps|^2, \qquad \Ecal^\eps_B(t):= \int_\Omega \tilde{k}^\eps|\nabla u^\eps_B - \frac{(u^\eps_A+u^\eps_B)}{S}\nabla \tilde{h}^\eps|^2.
\end{equation*}

We start with the

\begin{lemm}\label{thm:regu1}
  Let $T>0$, $d\ioe 4$, $\Omega$ be a smooth  bounded domain of $\R^d$,
  $d_1,d_2,\sigma >0$, $r_u,r_v\ge 0$, $d_{ij} > 0$ and $d_A, d_B, h, k$ satisfying \eqref{eq:hk-relation}. Let $u_{A, \init}, u_{B, \init}, v_\init \soe 0$ belong to $\Ccal^2(\overline\Omega)$ and compatible with the homogeneous Neumann boundary conditions. Let $u^\eps_A, u^\eps_B, v^\eps$ be the unique global strong solution (in $\Ccal^2(\R_+\times\overline\Omega)$) to \eqref{eq:SKT-eps} (with initial data $u_{A, \init}, u_{B, \init}, v_\init$).

  Let $\delta > 0$, then there exists a constant $C_{\delta}$ that depends only on the parameters of \eqref{eq:SKT-eps}, $\delta, d, T, \Omega$ and the norm of $u_{A, \init}, u_{B, \init}, v_\init$ in $\Ccal^2(\overline\Omega)$, but not on $\eps$, such that for all $0 < t <T$, one has
    \begin{align*}
     & \Ecal_A(t) + \Ecal_B(t) + \int_0^t\int_\Omega |\Delta u_A|^2 + \int_0^t\int_\Omega |\Delta u_B|^2 + \int_0^t\int_\Omega \frac{1}{\eps}|\nabla Q|^2                            \\
     & \ioe  C_\delta(\Ecal_A(0) + \Ecal_B(0)+ 1) + \delta (\|\nabla u_A \|_{\fsL^{\infty}([0, t], \fsL^{2}(\Omega))}^2 + \|\nabla u_B \|^2_{\fsL^{\infty}([0, t], \fsL^{2}(\Omega))}).
  \end{align*}

  Furthermore, there exists a constant $C$ (with the same dependence as $C_\delta$, except $\delta$) such that
  \begin{equation*}
    \|\nabla u_A^\eps \|_{\fsL^{\infty}([0, T], \fsL^{2}(\Omega))} + \|\nabla u_B^\eps \|_{\fsL^{\infty}([0, T], \fsL^{2}(\Omega))}\ioe C,
  \end{equation*}
  \begin{equation*}
    \|u_A^\eps\|_{\fsL^2([0, T], \fsH^2(\Omega))} + \|u_B^\eps\|_{\fsL^2([0, T], \fsH^2(\Omega))} \ioe C,
  \end{equation*}
  and
  \begin{equation*}
    \int_0^T\int_\Omega \frac{1}{\eps}|\nabla Q^\eps|^2 \ioe C.
  \end{equation*}
\end{lemm}

\begin{proof}
  For readability we will drop the super-script $\eps$ in the rest of the proof. Let $0<t<T$, we first compute
  \begin{align*}
    \frac{\dd \Ecal_A}{\dd t}(t)
     & = \int_\Omega \partial_t \tilde{h}|\nabla u_A -  \frac{u_A+u_B}{S}\nabla \tilde{k}|^2                                                                                                                                      \\
     & +2\int_\Omega \tilde{h}(\nabla u_A - \frac{u_A+u_B}{S}\nabla \tilde{k})\cdot\left(\partial_t \nabla u_A - (\partial_t\nabla \tilde{k}) \frac{u_A+u_B}{S} - \nabla \tilde{k}\partial_t\left(\frac{u_A+u_B}{S}\right)\right) \\
     & = \int_\Omega \partial_t \tilde{h} |\nabla u_A - \frac{u_A+u_B}{S}\nabla \tilde{k}|^2                                                                                                                                      \\
     & +2\int_\Omega \tilde{h}(\nabla u_A - \frac{u_A+u_B}{S}\nabla \tilde{k})\cdot\Big(d_A\nabla\Delta u_A + \nabla(f_u u_A) + \frac 1\eps \nabla Q                                                                              \\
     & - (\partial_t\nabla \tilde{k}) \frac{u_A+u_B}{S} - \frac1S \nabla \tilde{k}(d_A\Delta u_A +(d_A + d_B)\Delta u_B + f_u(u_A+u_B))\Big)                                                                                      \\
     & = \int_\Omega \partial_t \tilde{h} |\nabla u_A - \frac{u_A+u_B}{S}\nabla \tilde{k}|^2                                                                                                                                      \\
     & - 2\int_\Omega \tilde{h}d_A|\Delta u_A|^2 - 2\int_\Omega d_A\nabla \tilde{h}\cdot\nabla u_A\Delta u_A - 2\int_\Omega d_A\tilde{h}\nabla \tilde{k}\frac{u_A+u_B}{S}\cdot\nabla\Delta u_A                                    \\
     & + 2\int_\Omega \tilde{h}\frac{1}{\eps}\nabla Q\cdot(\nabla u_A - \frac{u_A+u_B}{S}\nabla \tilde{k})                                                                                                                        \\
     & + 2\int_\Omega \tilde{h}(\nabla u_A - \frac{u_A+u_B}{S}\nabla \tilde{k})\cdot\Big(\nabla(f_u u_A) - (\partial_t\nabla \tilde{k}) \frac{u_A+u_B}{S}                                                                         \\
     & -\frac{1}{S}\nabla \tilde{k}(d_A\Delta u_A +(d_A + d_B)\Delta u_B+ f_u(u_A+u_B))\Big)                                                                                                                                      \\
     & =:  - 2\int_\Omega \tilde{h}d_A|\Delta u_A|^2 + 2\int_\Omega \tilde{h}\frac{1}{\eps}\nabla Q\cdot(\nabla u_A - \frac{u_A+u_B}{S}\nabla \tilde{k}) + I_A.
  \end{align*}

  This expression has two important features: the term in $\frac{1}{\eps}$ will turn out to be non-positive (after we sum it with the corresponding term in the computation of $\frac{\dd\Ecal_B}{\dd t}$) and the higher-order term (in space regularity) $- \int_\Omega \tilde{h}d_A|\Delta u_A|^2$ is also non-positive. Thus $I_A$ only acts as a lower-order term (in $\eps$ and in terms of derivatives). We can similarly define the quantity $I_B$.\medskip

  We integrate the previous expression over $[0,t]$ and sum it with its counter-part $\Ecal_B$. We also use the assumption $h, k > h_0$:
  \begin{align*}
     & \Ecal_A(t) + \Ecal_B(t) + 2d_Ah_0\int_0^t\int_\Omega |\Delta u_A|^2 + 2(d_A+d_B)h_0\int_0^t\int_\Omega |\Delta u_B|^2                                                                               \\
     & \ioe  \Ecal_A(0) + \Ecal_B(0)+ \int_0^t (I_A + I_B)                                                                                                                                                 \\
     & +\frac{2}{\eps}\int_0^t\int_\Omega \nabla Q\cdot\left[\tilde{h}\nabla u_A - \tilde{h}\frac{u_A+u_B}{S}\nabla \tilde{k} - (\tilde{k}\nabla u_B - \tilde{k}\frac{u_A+u_B}{S}\nabla \tilde{h})\right].
  \end{align*}

  We first start with the term in $\frac{1}{\eps}$. We recall that $\tilde{h}+\tilde{k} = S$ so that $\nabla \tilde{h} + \nabla \tilde{k} = 0$
  \begin{align*}
     & \frac{1}{\eps}\int_0^t\int_\Omega \nabla Q\cdot\left[\tilde{h}\nabla u_A - \tilde{h}\frac{u_A+u_B}{S}\nabla \tilde{k} - (\tilde{k}\nabla u_B - \tilde{k}\frac{u_A+u_B}{S}\nabla \tilde{h})\right]              \\
     & =\frac{1}{\eps}\int_0^t\int_\Omega \nabla Q\cdot\left[\tilde{h}\nabla u_A + \frac{\tilde{k}+\tilde{h}}{S}\nabla \tilde{h} u_A - \tilde{k}\nabla u_B - \frac{\tilde{k}+\tilde{h}}{S}\nabla \tilde{k} u_B\right] \\
     & = - \int_0^t\int_\Omega \frac{1}{\eps}|\nabla Q|^2 \ioe 0.
  \end{align*}

  In the last line we used $\nabla Q = \tilde{k}\nabla u_B - \tilde{h}\nabla u_A + \nabla \tilde{k} u_B - \nabla \tilde{h} u_A$.\medskip

  We will denote by $C$ a constant with the same dependency as in the statement of the Proposition, this constant may change from line to line. We also introduce a small constant $\eta > 0$ to be chosen later (and that may depend on the same parameters as $C$).

  We can then bound $\int_0^t (I_A + I_B)$. Thanks to the regularity of $k, h$, we deduce that $\tilde{k}, \tilde{h}$ respect the same estimates as $v$ in \cref{thm:DGNM-bound}.\medskip

  In the following computations, we will need to bound $\|\nabla u \|_{\fsL^{2+}(\Omega_t)}$, but we only control $\|\nabla u \|_{\fsL^{2}(\Omega_t)}$ by \cref{thm:DGNM-bound}. To be able to control this quantity, we will interpolate with $\|\Delta u \|_{\fsL^{2}(\Omega_t)}$ in space and with $\|\nabla u \|_{\fsL^{\infty}([0, T], \fsL^2(\Omega))}$ in time.

  Let $\kappa >0$ be a small parameter. Then by interpolation we first obtain
  $$\|\nabla u \|^{2+\kappa}_{\fsL^{2+\kappa}(\Omega)} \ioe \|\nabla u \|_{\fsL^{2}(\Omega)}^{(1 - \theta)(2+\kappa)}\|\nabla u \|_{\fsL^{2*}(\Omega)}^{\theta(2+\kappa)},$$
  with $2^* := \frac{2d}{d- 2}$ (and any $q<+\infty$ large enough if $d\ioe 2$) the Sobolev exponent and $\theta = \frac{\frac 12 - \frac{1}{2 +\kappa}}{\frac 12 - \frac{1}{2^*}}$. We notice that $\theta (2 + \kappa) = \frac{\kappa}{2}(\frac 12 - \frac{1}{2^*})^{-1} < 2$ (for $\kappa$ small enough). Thus, thanks to Young's inequality, we obtain the estimate $\|\nabla u \|^{2+\kappa}_{\fsL^{2+\kappa}(\Omega)} \ioe \eta \|\nabla u \|^{2}_{\fsL^{2^*}(\Omega)} + C_\eta \|\nabla u \|^{s}_{\fsL^{2}(\Omega)}$, where $s = 2 \frac{(1 - \theta)}{(1 - \theta) - \kappa} + \frac{\kappa(1 - \theta)}{(1 - \theta) - \kappa} > 2$.

  Thus, thanks to a Sobolev injection (and elliptic regularity) we get
  \begin{align*}
    \int_0^t\|\nabla u \|^{2+}_{\fsL^{2+}(\Omega)}
     & \ioe \int_0^t \left(\eta \|\Delta u(t) \|^{2}_{\fsL^{2}(\Omega)} + C_\eta \|\nabla u \|^{2+}_{\fsL^{2}(\Omega)}\right)\nonumber                                             \\
     & \ioe  \eta \|\Delta u\|_{\fsL^{2}(\Omega_t)}^2 + C_\eta \|\nabla u \|_{\fsL^{\infty}([0, t], \fsL^{2}(\Omega))}^{0+}\int_0^t \|\nabla u \|^{2}_{\fsL^{2}(\Omega)}\nonumber.
  \end{align*}

  Finally we can use Young's inequality again. Let $\delta > 0$, then there exists a constant $C_\delta >0$ such that for any $x >0$, we have $x^{0+} \ioe \delta x^{2} + C_{\delta, 0+}$. We obtain
  \begin{align}\label{eq:initial-bound2plus}
    \|\nabla u \|^{2}_{\fsL^{2+}(\Omega_t)}
     & \ioe \|\nabla u \|^{2+}_{\fsL^{2+}(\Omega_t)} + 1\nonumber \\
     & \ioe \eta\|\Delta u\|_{\fsL^{2}(\Omega_t)}^2 + \delta \|\nabla u \|_{\fsL^{\infty}([0, t], \fsL^{2}(\Omega))}^2\|\nabla u \|_{\fsL^{2}(\Omega_t)}^{2}\\
     &+ C_{\eta, \delta, T, 2+}\|\nabla u \|_{\fsL^{2}(\Omega_t)}^{2} + 1\nonumber.
  \end{align}
  The same result holds for $u_C$.\medskip

  We recall the estimates obtained in \cref{thm:DGNM-bound}. We can start to bound all the terms appearing in $I_A$:
  \begin{align*}
     & \int_0^t\int_\Omega \partial_t \tilde{h} \left|\nabla u_A - \frac{u_A+u_B}{S}\nabla \tilde{k}\right|^2                                                                                                   \\
     & \ioe 2 \int_0^t\int_\Omega \partial_t \tilde{h} (|\nabla u_A|^2 + |\frac{u_A+u_B}{S}\nabla \tilde{k}|^2)                                                                                      \\
     & \lesssim \|\partial_t \tilde{h}\|_{\fsL^{\infty-}(\Omega_t)}(\|\nabla u_A\|_{\fsL^{2+}(\Omega_t)}^2+  \|u_A + u_B\|_{\fsL^{2+}(\Omega_t)}^2 \|\nabla \tilde{k}\|_{\fsL^{\infty}(\Omega_t)}^2) \\
     & \ioe \eta \|\Delta u_A\|_{\fsL^{2}(\Omega_t)}^2 + \delta \|\nabla u_A \|^2_{\fsL^{\infty}([0, t], \fsL^{2}(\Omega))} + C_{\eta, \delta},
  \end{align*}
  where we used \eqref{eq:initial-bound2plus} to obtain the last line (up to changing $\eta, \delta$).

  For the next term, we simply use Young's inequality,
  \begin{align*}
    \int_0^t\int_\Omega \nabla \tilde{h}\cdot\nabla u_A\Delta u_A
     & \ioe \eta \|\Delta u_A\|_{\fsL^2(\Omega_t)}^2 + C_\eta \|\nabla u_A\|_{\fsL^{2}(\Omega_t)}^2\|\nabla \tilde{h}\|_{\fsL^{\infty}(\Omega_t)}^2 \\
     & \ioe \eta \|\Delta u_A\|_{\fsL^2(\Omega_t)}^2 + C_\eta.
  \end{align*}

  For the next term, we can also use Young's inequality:
  \begin{align*}
    \int_0^t\int_\Omega (u_A+u_B)\tilde{h}\nabla \tilde{k}\cdot \nabla\Delta u_A
     & = -\int_0^t\int_\Omega \Delta u_A \nabla\cdot((u_A+u_B)\tilde{h}\nabla \tilde{k})                                                                                                           \\
     & \ioe \eta \|\Delta u_A\|_{\fsL^2(\Omega_t)}^2 + C_\eta \|\nabla\cdot((u_A+u_B)\tilde{h}\nabla \tilde{k})\|_{\fsL^2(\Omega_t)}^2                                                             \\
     & \ioe \eta \|\Delta u_A\|_{\fsL^2(\Omega_t)}^2 + C_\eta (\|\tilde{h}\|^2_{\fsL^\infty(\Omega_t)}\|\nabla \tilde{k}\|_{\fsL^{\infty}(\Omega_t)}^2\|\nabla(u_A + u_B)\|_{\fsL^{2}(\Omega_t)}^2 \\
     & + \|\nabla \tilde{h}\|^2_{\fsL^\infty(\Omega_t)}\|\nabla \tilde{k}\|_{\fsL^{\infty}(\Omega_t)}^2\|(u_A + u_B)\|_{\fsL^{2}(\Omega_t)}^2                                                      \\
     & + \|\tilde{h}\|^2_{\fsL^\infty(\Omega_t)}\|\nabla \nabla \tilde{k}\|_{\fsL^{4}(\Omega_t)}^2\|(u_A + u_B)\|_{\fsL^{4}(\Omega_t)}^2)                                                          \\
     & \ioe \eta \|\Delta u_A\|_{\fsL^2(\Omega_t)}^2 + C_\eta.
  \end{align*}

  Let us start with the easier part of the last term:
  \begin{align*}
     & \int_\Omega \tilde{h}\frac{u_A+u_B}{S}\nabla \tilde{k}\cdot \left[ - \frac{1}{S}\nabla \tilde{k}(d_A\Delta u_A +(d_A + d_B)\Delta u_B + f_u(u_A + u_B))\right] \\
     & \lesssim \|\tilde{h}\|_{\fsL^\infty(\Omega_t)}\|u_A + u_B\|_{\fsL^{2}(\Omega_t)}\|\nabla \tilde{k}\|_{\fsL^{\infty}(\Omega_t)}                                 \\
     & \times\left(\|\nabla \tilde{k}(d_A\Delta u_A +(d_A + d_B)\Delta u_B)\|_{\fsL^{2}(\Omega_t)} + \|f_u(u_A + u_B)\|_{\fsL^{2}(\Omega_t)}\right).
  \end{align*}

  Moreover, we can bound
  \begin{align*}
    &\|\nabla \tilde{k}(d_A\Delta u_A +(d_A + d_B)\Delta u_B)\|_{\fsL^{2}(\Omega_t)}\\
     & \ioe \|\nabla \tilde{k}\|_{\fsL^{\infty}(\Omega_t)}\|d_A\Delta u_A +(d_A + d_B)\Delta u_B\|_{\fsL^{2}(\Omega_t)}\\
     & \ioe \eta \|\Delta u_A\|_{\fsL^{2}(\Omega_t)}^2 + \eta \|\Delta u_B\|_{\fsL^{2}(\Omega_t)}^2 + C_\eta \|\nabla \tilde{k}\|_{\fsL^{\infty}(\Omega_t)}^2 \\
     & \ioe \eta \|\Delta u_A\|_{\fsL^{2}(\Omega_t)}^2 + \eta \|\Delta u_B\|_{\fsL^{2}(\Omega_t)}^2 + C_\eta.
  \end{align*}

  Hence, we obtain
  \begin{align*}
     & \int_\Omega \tilde{h}\frac{u_A+u_B}{S}\nabla \tilde{k}\cdot \left[ - \frac{1}{S}\nabla \tilde{k}(d_A\Delta u_A +(d_A + d_B)\Delta u_B + f_u(u_A + u_B))\right] \\
     & \ioe \eta \|\Delta u_A\|_{\fsL^{2}(\Omega_t)}^2 + \eta \|\Delta u_B\|_{\fsL^{2}(\Omega_t)}^2 + C_\eta.
  \end{align*}

  We can go to the next term
  \begin{align*}
    \int_0^t\int_\Omega \tilde{h}\frac{u_A+u_B}{S}\nabla \tilde{k} \cdot\nabla(f_u u_A)
     & = - \int_0^t\int_\Omega \nabla\cdot\left(\tilde{h}\frac{u_A+u_B}{S}\nabla \tilde{k}\right) f_u u_A                                             \\
     & \ioe \|f_u u_A\|_{\fsL^{2}(\Omega_t)}                                                                                                          \\
     & \times (\|\nabla \tilde{h}\|_{\fsL^{\infty}(\Omega_t)}\|\nabla \tilde{k}\|_{\fsL^{\infty}(\Omega_t)}\|\frac{u_A+u_B}{S}\|_{\fsL^{2}(\Omega_t)} \\
     & + \|\tilde{h}\|_{\fsL^{\infty}(\Omega_t)}\|\nabla\nabla \tilde{k}\|_{\fsL^{4}(\Omega_t)}\|\frac{u_A+u_B}{S}\|_{\fsL^{4}(\Omega_t)}             \\
     & + \frac{2}{S}\| \tilde{h}\|_{\fsL^{\infty}(\Omega_t)}\|\nabla \tilde{k}\|_{\fsL^{\infty}(\Omega_t)}\|\nabla(u_A+u_B)\|_{\fsL^{2}(\Omega_t)})   \\
     & \ioe C.
  \end{align*}

  In the same way, we have
  \begin{align*}
    \int_0^t\int_\Omega \tilde{h}\left(\frac{u_A+u_B}{S}\right)^2\nabla \tilde{k} \cdot \partial_t\nabla \tilde{k}
     & = - \int_0^t\int_\Omega \nabla\cdot\left(\tilde{h}\left(\frac{u_A+u_B}{S}\right)^2\nabla \tilde{k}\right) \partial_t \tilde{k}\\
     & \ioe \|\partial_t \tilde{k}\|_{\fsL^{2+}(\Omega_t)} \\
     & \times (\|\nabla \tilde{h}\|_{\fsL^{\infty}(\Omega_t)}\|\nabla \tilde{k}\|_{\fsL^{\infty}(\Omega_t)}\|\left(\frac{u_A+u_B}{S}\right)^2\|_{\fsL^{2+}(\Omega_t)}                                    \\
     & + \|\tilde{h}\|_{\fsL^{\infty}(\Omega_t)}\|\nabla\nabla \tilde{k}\|_{\fsL^{4}(\Omega_t)}\|\left(\frac{u_A+u_B}{S}\right)^2\|_{\fsL^{4+}(\Omega_t)}                                                \\
     & + \frac{2}{S^2}\| \tilde{h}\|_{\fsL^{\infty}(\Omega_t)}\|\nabla \tilde{k}\|_{\fsL^{\infty}(\Omega_t)}\|\nabla(u_A+u_B)\|_{\fsL^{2}(\Omega_t)}\|u_A+u_B\|_{\fsL^{\infty - }(\Omega_t)}) \\
     & \ioe C.
  \end{align*}

  Similarly, we have
  \begin{align*}
    \int_0^t\int_\Omega \tilde{h}\nabla u_A\cdot \nabla(f_u u_A)
     & = - \int_0^t\int_\Omega \nabla\cdot(\tilde{h}\nabla u_A) f_u u_A                                                                                                                      \\
     & \ioe \|f_u u_A\|_{\fsL^{2}}(\|\tilde{h}\|_{\fsL^\infty(\Omega_t)}\|\Delta u_A\|_{\fsL^{2}(\Omega_t)}+\|\nabla \tilde{h}\|_{\fsL^\infty(\Omega_t)}\|\nabla u_A\|_{\fsL^{2}(\Omega_t)}) \\
     & \ioe \eta \|\Delta u_A\|^2_{\fsL^{2}(\Omega_t)} + C_\eta.
  \end{align*}

  For the next term, we can proceed in a similar way:
  \begin{align*}
    \int_0^t\int_\Omega \tilde{h}\nabla u_A\cdot (\partial_t\nabla \tilde{k}) (u_A + u_B)
     & = - \int_0^t\int_\Omega \nabla\cdot\left[\tilde{h}\nabla u_A(u_A + u_B)\right] (\partial_t \tilde{k})                                  \\
     & \ioe \|\partial_t \tilde{k}\|_{\fsL^{\infty-}(\Omega_t)}                                                                               \\
     & \times (\|\nabla \tilde{h}\|_{\fsL^{\infty}(\Omega_t)}\|\nabla u_A\|_{\fsL^{2}(\Omega_t)}\|u_A + u_B\|_{\fsL^{2+}(\Omega_t)}           \\
     & + \|\tilde{h}\|_{\fsL^{\infty}(\Omega_t)}\|\Delta u_A\|_{\fsL^{2}(\Omega_t)}\|u_A + u_B\|_{\fsL^{2+}(\Omega_t)}                        \\
     & + \| \tilde{h}\|_{\fsL^{\infty}(\Omega_t)}\|\nabla u_A\|_{\fsL^{2+}(\Omega_t)}\|\nabla(u_A + u_B)\|_{\fsL^{2}(\Omega_t)})              \\
     & \ioe \eta \|\Delta u_A\|^2_{\fsL^{2}(\Omega_t)} + \delta \|\nabla u_A \|_{\fsL^{\infty}([0, t], \fsL^{2}(\Omega))}^2 + C_{\eta, \delta}.
  \end{align*}
  To obtain the last line we used \eqref{eq:initial-bound2plus} (to bound $\|\nabla u_A\|_{\fsL^{2+}(\Omega_t)}$) as well as Young's inequality (to bound the term with $\|\Delta u_A\|_{\fsL^{2}(\Omega_t)}$).\medskip

  Finally we have a last term to control:
  \begin{align*}
     & \int_0^t\int_\Omega \tilde{h}\nabla u_A\cdot\nabla \tilde{k}(d_A\Delta u_A +(d_A + d_B)\Delta u_B + f_u(u_A + u_B))                                                                                                                                 \\
     & \lesssim \|\tilde{h}\|_{\fsL^\infty(\Omega_t)}\|\nabla \tilde{k}\|_{\fsL^{\infty}(\Omega_t)}\|\nabla u_A\|_{\fsL^{2}(\Omega_t)}\left(\|\Delta u_A\|_{\fsL^{2}(\Omega)} + \|\Delta u_B\|_{\fsL^{2}(\Omega_t)} + \|f_u (u_A + u_B)\|_{\fsL^{2}(\Omega_t)}\right) \\
     & \ioe C_\eta + \eta (\|\Delta u_A\|^2_{\fsL^{2}(\Omega_t)} + \|\Delta u_B\|^2_{\fsL^{2}(\Omega_t)}).
  \end{align*}

  The same estimates hold true for $\Ecal_B$, thus if we recall all our estimates and choose $\eta$ small enough, we get:
  \begin{align}\label{eq:part1inter}
     & \Ecal_A(t) + \Ecal_B(t) + \int_0^t\int_\Omega |\Delta u_A|^2 + \int_0^t\int_\Omega |\Delta u_B|^2 + \int_0^t\int_\Omega \frac{1}{\eps}|\nabla Q|^2 \\
     & \lesssim  C(\Ecal_A(0) + \Ecal_B(0)+ 1) + \delta (\|\nabla u_A \|_{\fsL^{\infty}([0, t], \fsL^{2}(\Omega))}^2 + \|\nabla u_B \|^2_{\fsL^{\infty}([0, t], \fsL^{2}(\Omega))}), \nonumber
  \end{align}
  which yields the desired conclusion.\medskip

  For the second part of the Proposition, we use the elementary inequality $a^2 \ioe 2(a-b)^2 + 2b^2$ for $a, b \soe 0$. Thus
  \begin{align*}
    \int_\Omega h_0|\nabla u_A(t)|^2
     & \ioe 2\int_\Omega \tilde{h}\left|\nabla u_A - \frac{(u_A+u_B)}{S}\nabla \tilde{k}\right|^2 + 2\int_\Omega \tilde{h}\left|\frac{(u_A+u_B)}{S}\nabla \tilde{k}\right|^2                    \\
     & \ioe 2 \Ecal_A(t) + \frac 2S\|\tilde{h}\|_{\fsL^\infty(\Omega_T)}\|\nabla \tilde{k}\|^2_{\fsL^\infty(\Omega_T)}\|u_A+u_B\|^2_{\fsL^\infty([0, T], \fsL^2(\Omega))} \\
     & \ioe C \left(1 + \delta\|\nabla u_A \|_{\fsL^{\infty}([0, t], \fsL^{2}(\Omega))}^2 + \delta\|\nabla u_B \|_{\fsL^{\infty}([0, t], \fsL^{2}(\Omega))}^2\right).
  \end{align*}
  A similar estimate holds for $\nabla u_B$. Thus, if we fix $\delta$ small enough and take the supremum over $0 < t < T$, we obtain
  \begin{equation*}
    \|\nabla u_A \|_{\fsL^{\infty}([0, T], \fsL^{2}(\Omega))}^2 + \|\nabla u_B \|_{\fsL^{\infty}([0, T], \fsL^{2}(\Omega))}^2 \ioe C.
  \end{equation*}

  Using this estimate in \eqref{eq:part1inter}, we conclude the proof of the lemma.

\end{proof}

\subsection{Part 2: Higher regularity estimates}
In this section, we prove the second part of \cref{thm:SKTeps-regu}, namely the estimates in the interior of the domain for $l > 0$. The proof relies on similar ideas as in the case $l=0$, namely the introduction of a functional $\Ecal$ that remains bounded along time. More precisely, we prove the

\begin{lemm}\label{thm:regu2}
  Let $T>0$, $d\ioe 3$, $\Omega$ be a smooth  bounded domain of $\R^d$, $d_1,d_2,\sigma >0$, $r_u,r_v\ge 0$, $d_{ij} > 0$ and $d_A, d_B, h, k$ satisfying \eqref{eq:hk-relation}. We assume that $h, k$ are affine on $[0, \sup_{\eps > 0} \|v^\eps\|_{\fsL^\infty(\Omega_T)}]$. Let $\tilde\Omega\Subset\Omega$, let $l> 0$ and $u_{A, \init}, u_{B, \init} \soe 0$ belong to $\Ccal^2(\overline\Omega)\cap\fsH^l(\Omega)$, $v_\init\soe 0$ belongs to $\fsW^{l+ 2, \infty}(\Omega)$ and compatible with the homogeneous Neumann boundary conditions. Let $u^\eps_A, u^\eps_B, v^\eps$ be the unique global strong solution (in $\Ccal^2(\R_+\times\overline\Omega)$) to \eqref{eq:SKT-eps} (with initial data $u_{A, \init}, u_{B, \init}, v_\init$).

  Then there exists a constant $C_l$ that depends only on the parameters of \eqref{eq:SKT-eps}, $d, T, l, \tilde\Omega$ and the norm of $u_{A, \init}, u_{B, \init}, v_\init$ in $\Ccal^2(\overline\Omega)\cap \fsH^{l}(\Omega)$, but not on $\eps$, such that
  \begin{align*}
    \|\nabla^{\otimes l}u_C^\eps\|_{\fsL^\infty([0, T], \fsL^2({\tilde\Omega}))} + \|\nabla^{\otimes (l+1)}u_C^\eps\|_{\fsL^2({\tilde\Omega\times[0, T]})} \ioe C_l,\\
    \|\nabla^{\otimes (l - 1)}v^\eps\|_{\fsL^{\infty}(\tilde\Omega\times[0, T])} + \|\nabla^{\otimes l}v^\eps\|_{\fsL^{10}(\tilde\Omega\times[0, T])} \ioe C_l,\\
    \|\partial_t\nabla^{\otimes (l-1)}v^\eps\|_{\fsL^{\frac{10}{3}}(\tilde\Omega\times[0, T])} + \|\nabla^{\otimes (l+1)}v^\eps\|_{\fsL^\frac{10}{3}(\tilde\Omega\times[0, T])}\ioe C_l,\\
  \end{align*}
  and
  \begin{equation*}
    \frac{1}{\eps}\int_0^t\int_{\tilde\Omega} |\nabla^{\otimes l} Q|^2 \ioe C_l.
  \end{equation*}
\end{lemm}

\begin{proof}
  We will prove by induction on $l \soe 1$ that for any $\Omega_1 \Subset \Omega$, there exists some constant $C > 0$ depending on $\Omega_1, l, T, d$ and the parameters of \eqref{eq:SKT-eps} but independent of $\eps$, such that

  \begin{equation}\label{eq:HRu1}
    \|\nabla^{\otimes l}u_C^\eps\|_{\fsL^\infty([0, T], \fsL^2({\Omega_1}))} + \|\nabla^{\otimes (l+1)}u_C^\eps\|_{\fsL^2({\Omega_1\times[0, T]})} \ioe C,
  \end{equation}
  and
  \begin{align}\label{eq:HRv1}
    \|\nabla^{\otimes (l - 1)}v^\eps\|_{\fsL^{\infty}(\Omega_1\times[0, T])} + \|\nabla^{\otimes l}v^\eps\|_{\fsL^{10}(\Omega_1\times[0, T])}\ioe C,\\
     \|\partial_t\nabla^{\otimes (l-1)}v^\eps\|_{\fsL^{\frac{10}{3}}(\Omega_1\times[0, T])} + \|\nabla^{\otimes (l+1)}v^\eps\|_{\fsL^\frac{10}{3}(\Omega_1\times[0, T])}\ioe C.\nonumber
  \end{align}

  The case $l = 1$ has already been proven (see \cref{thm:DGNM-bound} for \eqref{eq:HRv1} and \cref{thm:regu1} for \eqref{eq:HRu1}).
  
  Let $l\soe 2$, we assume that the induction hypothesis holds for $l - 1$, hence for any $\Omega_1 \Subset \Omega$
    \begin{equation}\label{eq:HRu}
    \|\nabla^{\otimes (l-1)}u_C^\eps\|_{\fsL^\infty([0, T], \fsL^2({\Omega_1}))} + \|\nabla^{\otimes l}u_C^\eps\|_{\fsL^2({\Omega_1\times[0, T]})} \ioe C,
  \end{equation}
  and
  \begin{align}\label{eq:HRv}
    \|\nabla^{\otimes (l - 2)}v^\eps\|_{\fsL^{\infty}(\Omega_1\times[0, T])} + \|\nabla^{\otimes (l-1)}v^\eps\|_{\fsL^{10}(\Omega_1\times[0, T])}\ioe C,\\
    \|\partial_t\nabla^{\otimes (l-2)}v^\eps\|_{\fsL^{\frac{10}{3}}(\Omega_1\times[0, T])} + \|\nabla^{\otimes l}v^\eps\|_{\fsL^\frac{10}{3}(\Omega_1\times[0, T])}\ioe C.\nonumber
  \end{align}

  In the rest of the proof, $\alpha, \beta, \gamma \in \N^d$ will always denote multi-indices. We denote $|\alpha| := \sum_{i = 1}^d \a_i$ and we write $\a \ioe \beta$ if for $1\ioe i\ioe d$, one has $\a_i \ioe \beta_i$. We denote $<$ the associated strict order (thus $\a < \beta$ if $\alpha \ioe \beta$ and $\a \neq \beta$). Finally, for $\a\ioe\beta$, we also denote $\binom{\a}{\beta} := \frac{|\a|!}{\prod_{i =1}^d |\beta_i|!}$ the binomial coefficient.\\
  Let $\Omega_2\Subset\Omega$ and let $k<l$, using \eqref{eq:HRu} as well as a Sobolev embedding we obtain $\nabla^{\otimes k} u_C^\eps \in\fsL^\infty([0, T], \fsL^2(\Omega_2))\cap\fsL^2([0, T], \fsL^6(\Omega_2))$ in dimension $d\ioe 3$. Thus, by an interpolation argument, we deduce that for any $\Omega_2\Subset\Omega$
  \begin{equation}\label{eq:HR3}
    \|\nabla^{\otimes k}u_C^\eps\|_{\fsL^{\frac{10}{3}}(\Omega_2\times [0, T])}\ioe C.
  \end{equation}\medskip

  \textit{Regularity of $v^\eps$.} For readability, when it doesn't create any ambiguity, we will drop the super-script $\eps$ in the rest of the proof. Let any $\Omega_1 \Subset \Omega_2 \Subset \Omega_3\Subset\Omega$ and let $\alpha$ be such that $|\alpha| := l - 1$, and $v_\alpha := \partial^\alpha v$. We recall the notation $u^\eps := u_A^\eps + u_B^\eps$. We apply $\partial^\alpha$ to \eqref{eq:SKT-eps} to obtain
  \begin{align*}
    \partial_t v_\alpha - d_v\Delta v_\alpha - r_v v_\alpha
     & = - d_{21}\partial^\alpha(uv) - d_{22}\partial^\alpha(v^2)\\
     & = (- d_{21}u - 2d_{22}v) v_\alpha - d_{21}\partial^\a u v\\
     & -  d_{21}\sum_{\beta + \gamma = \alpha \text{ and }\beta, \gamma > 0}\binom{\a}{\beta}\partial^\gamma u \partial^\beta v - d_{22}\sum_{\beta + \gamma = \alpha \text{ and }\beta, \gamma > 0}\binom{\a}{\beta}\partial^\gamma v \partial^\beta v \\
     & =: (- d_{21}u- 2d_{22}v) v_\alpha - d_{21}\partial^\a u v + R_1 + R_2.
  \end{align*}

  By \eqref{eq:initial-estimates-uLq} and \eqref{eq:HRv} (applied to $\Omega_3$), $R_2$ and $(- d_{21}u- d_{22}v) \in \fsL^{\infty-}(\Omega_3\times[0, T]) $. Using that in the sum defining $R_1$, we have $|\beta| \ioe |\alpha| - 1 = l - 2$ as well as \eqref{eq:HRv} and \eqref{eq:HR3}, we obtain $\|R_1\|_{\fsL^\frac{10}{3}(\Omega_3\times [0, T])}\ioe C$. Similarly, we also have $v\partial^\a u \in\fsL^\frac{10}{3}(\Omega_3\times [0, T]) $.
  Thus, noticing that $\nabla^{\otimes (l- 1)} v_\init\in\fsL^\infty(\Omega)$, we obtain by (interior) parabolic regularity (see for instance \cite[Chapter III, Theorem 8.1]{ladyzenskajasolonnikov1968}), in dimension $d\ioe 3$ for any $\Omega_2\Subset\Omega$,
  \begin{equation}\label{eq:boundvlm1}
    \|\nabla^{\otimes (l- 1)} v\|_{\fsL^{\infty}(\Omega_2\times [0, T])} \ioe C.
  \end{equation}

  Let again $\Omega_1\Subset\Omega_2\Subset \Omega_3\Subset\Omega$. We consider $\chi$ a smooth cutoff function between $\Omega_2$ and $\Omega_3$ (hence $\chi \equiv 1$ on $\Omega_2$ and $\chi \equiv 0$ on $\Omega_3^c$).
  
  By the previous estimates (applied on $\Omega_3$), we observe that $r_v v_\alpha + (- d_{21}u - 2d_{22}v) v_\alpha$ is in $\fsL^{\infty-}(\Omega_3\times [0, T])$. Thus, we have
  \begin{equation*}
    \partial_t (\chi v_\a) - d_v\Delta (\chi v_\a) = R - d_v\Delta\chi v_\a - 2d_v\nabla\chi\cdot\nabla v_\a,
  \end{equation*}
  for some $R\in\fsL^\frac{10}{3}(\Omega_3\times[0, T])$. By the induction hypothesis, we also have $d_v\Delta\chi v_\a - 2d_v \nabla\chi\cdot\nabla v_\a \in\fsL^\frac{10}{3}(\Omega_3\times[0, T])$.
  
  By maximal regularity (noticing that $\partial^\a v_\init\in\fsW^{2, \infty}$), we obtain that for all $\Omega_2\Subset\Omega$, we have
  \begin{equation}\label{eq:boundvlp1}
    \|\partial_t \nabla^{\otimes (l - 1)} v\|_{\fsL^\frac{10}{3}(\Omega_2\times [0, T])} + \|\nabla^{\otimes (l+1)} v\|_{\fsL^\frac{10}{3}(\Omega_2\times [0, T])} \ioe C.
  \end{equation}

  Let again $\Omega_1\Subset\Omega_2\Subset \Omega_3\Subset\Omega$ and $\chi$ a smooth cutoff function between $\Omega_2$ and $\Omega_3$.

  We now take $|\alpha| = l$, similar computations as previously (using \eqref{eq:HRu}) show that
  \begin{align*}
    \partial_t (\chi v_\alpha) - d_v\Delta (\chi v_\alpha)
     & = r_v\chi v_\alpha- d_{21}\chi\partial^\alpha(uv) - d_{22}\chi \partial^\alpha(v^2) - d_v\Delta\chi v_\a - 2d_v\nabla\chi\cdot\nabla v_\a\\
     & = r_v \chi v_\alpha -  d_{21}\chi \sum_{\beta + \gamma = \alpha}\binom{\a}{\beta}\partial^\gamma u \partial^\beta v - d_{22}\chi \sum_{\beta + \gamma = \alpha}\binom{\a}{\beta}\partial^\gamma v \partial^\beta v\\
     &- d_v\Delta\chi v_\a - 2d_v\nabla\chi\cdot\nabla v_\a.
  \end{align*}
  We first consider $\partial^\gamma u \partial^\beta v$ for $|\beta|\ioe l - 1$, then by \eqref{eq:boundvlm1} and \eqref{eq:HRu} (applied on $\Omega_3$), we obtain $\partial^\gamma u \partial^\beta v \in\fsL^2(\Omega_3\times[0, T])$. When we have $\beta = \a$, we can use \eqref{eq:initial-estimates-uLq} and \eqref{eq:HRv} (applied on $\Omega_3$) to deduce that $u\partial^\a v\in\fsL^{\frac{10}{3}-}(\Omega_3\times[0, T])\subset \fsL^{2}(\Omega_3\times[0, T])$. Similarly, one has $\partial^\gamma v \partial^\beta v\in\fsL^{2}(\Omega_3\times[0, T])$. Finally, we can use \eqref{eq:boundvlp1} to show that $\nabla v_\a \in \fsL^{2}(\Omega_2\times[0, T])$. Hence, we have shown
  \begin{equation*}
    \partial_t (\chi v_\a) - d_v\Delta (\chi v_\a) \in\fsL^2(\Omega_3\times[0, T]).
  \end{equation*}

  Hence, by \cite[Chapter III, Theorem 9.1]{ladyzenskajasolonnikov1968}, we have for all $\Omega_2\Subset\Omega$ in dimension $d\ioe 3$ (we notice that $\nabla^{\otimes l} v_\init\in\fsL^\infty(\Omega)$)
  \begin{equation}\label{eq:boundvl}
    \|\nabla^{\otimes l} v\|_{\fsL^{10}(\Omega_2\times [0, T])}\ioe C.
  \end{equation}

  The estimates \eqref{eq:boundvlm1}, \eqref{eq:boundvlp1}, \eqref{eq:boundvl} are also valid if we replace $v$ by $h(v)$ (or $k(v)$), since we assumed that $h, k$ are affine on $[0, \sup_{\eps > 0} \|v^\eps\|_{\fsL^\infty(\Omega_T)}]$. We notice that we have proven the recursion hypothesis \eqref{eq:HRv1} concerning $v$.\medskip
  
  Let again $\Omega_1\Subset\Omega_2\Subset\Omega$ and $\chi$ a smooth cutoff function between $\Omega_1$ and $\Omega_2$ (hence $\chi \equiv 1$ on $\Omega_1$ and $\chi \equiv 0$ on $\Omega_2^c$). We will now give one conditional regularity estimate. We consider some $|\a| = l$, using \eqref{eq:boundvlm1} and \eqref{eq:HR3}, we see that $R_1 + R_2\in\fsL^{\frac{10}{3}}(\Omega_2\times[0, T])$. Then, similarly as in the proof of \eqref{eq:initial-bound2plus}, using an interpolation argument and a Sobolev inequality in dimension $d\ioe3$, we have that for some $0 <\theta <1 $ and any parameter $\delta > 0$:
  \begin{align}\label{eq:32}
    \|\chi\partial^\alpha u\|_{\fsL^{2+}(\Omega_2\times[0, t])}
    &\ioe \|\chi\partial^\alpha u\|_{\fsL^{\infty}([0, t], \fsL^2(\Omega_2))}^\theta\|\chi\partial^\alpha u\|_{\fsL^{2}([0, t], \fsL^6(\Omega_2))}^{1- \theta}\nonumber\\
    &\ioe \delta \|\chi\partial^\alpha u\|_{\fsL^{\infty}([0, t], \fsL^2(\Omega_2))} + C_\delta \|\nabla(\chi\partial^\alpha u)\|_{\fsL^2(\Omega_2\times[0, t])}.
  \end{align}

  Thus, if we consider the equation followed by $\chi v_\a$, we get
  \begin{equation*}
    \partial_t (\chi v_\a) - d_v \Delta(\chi v_a) = \chi \tilde R - d_v v_\a\Delta \chi - 2d_v \nabla v_\a \cdot \nabla \chi,
  \end{equation*}
  where $\tilde R := (- d_{21}u- 2d_{22}v) v_\alpha - d_{21}\partial^\a u v + R_1 + R_2$. We also have $- d_v v_\a\Delta \chi - 2d_v \nabla v_\a \cdot \nabla \chi \in\fsL^{2+}(\Omega_2\times[0, T])$ by \eqref{eq:boundvlp1}. By \eqref{eq:initial-estimates-uLq} and \eqref{eq:boundvl}, we also obtain $\chi (- d_{21}u- 2d_{22}v) v_\alpha\in\fsL^{2+}(\Omega_2\times[0, T])$.

  Thus, by maximal regularity (noticing that $\partial^\a v_\init\in\fsW^{2, \infty}$) and  \eqref{eq:32}, we obtain
  \begin{align*}
    \|\partial_t(\chi\partial^\alpha v)\|_{\fsL^{2+}(\Omega_2\times[0, t])}
     & \lesssim \|\chi \tilde R\|_{\fsL^{2+}(\Omega_2\times[0, t])} + C\\
     & \ioe \delta \|\chi\partial^\alpha u\|_{\fsL^{\infty}([0, t], \fsL^2(\Omega_2))} + C_\delta \|\nabla(\chi\partial^\alpha u)\|_{\fsL^{2}([0, t], \fsL^2(\Omega_2))} + C.
  \end{align*}

  We conclude (by time Sobolev injection) that
  \begin{align}\label{eq:boundvlinf2+}
    &\|\chi\partial^\alpha v\|_{\fsL^{\infty}([0, t], \fsL^{2+}(\Omega_2))}\\
    &\ioe\delta \|\chi\partial^\alpha u\|_{\fsL^{\infty}([0, t], \fsL^2(\Omega_2))} + C_\delta \|\nabla(\chi\partial^\alpha u)\|_{\fsL^{2}([0, t], \fsL^2(\Omega_2))} + C.\nonumber
  \end{align}\\

  \textit{Regularity of $u^\eps$.} In the following, when we claim that we \guillemotleft have a similar statement for $u_B$\guillemotright, it means that if we replace $u_A$ by $u_B$ in the statement and switch the role of $\tilde{k}$ and $\tilde{h}$, then the obtained statement is also true.\\

  Let $\Omega_1\Subset\Omega_2\Subset\Omega$ and $\chi$ a smooth cutoff function between $\Omega_1$ and $\Omega_2$ (meaning that $\chi \equiv 1$ on $\Omega_1$ and $\chi \equiv 0$ on $\Omega_2^c$).
Let $\alpha \in \N^d$ with $|\a| = l$, we consider the energy:
  \begin{equation*}
    \Ecal_{A, \eps, \chi, \alpha}(t) := \int_\Omega \chi^2 \tilde{h}^\eps\left|\partial^\alpha u^\eps_A - \frac{1}{S}\sum_{\gamma + \beta = \alpha \text{ and } |\gamma| > 0} \binom{\a}{\beta}\partial^\gamma \tilde{k}^\eps \partial^\beta u^\eps\right|^2.
  \end{equation*}

  For readability in the rest of the proof, we drop the indices $\eps$ and $\chi$. We denote $X_A := (\partial^\alpha u_A - \frac{1}{S}\sum_{\gamma + \beta = \alpha\text{ and } |\gamma| > 0} \binom{\a}{\beta}\partial^\gamma \tilde{k} \partial^\beta u)$. We similarly define $\Ecal_{B, \eps, \chi, \alpha}(t)$ and $X_B:= (\partial^\alpha u_B - \frac{1}{S}\sum_{\gamma + \beta = \alpha\text{ and } |\gamma| > 0} \binom{\a}{\beta}\partial^\gamma \tilde{h} \partial^\beta u)$.

  Then
  \begin{align*}
    \frac{\dd}{\dd t} \Ecal_{A, \alpha}(t)
     & = \int_\Omega \chi^2\partial_t\tilde{h} X_A^2 + 2\int_\Omega \chi^2 \tilde{h} X_A \partial_tX_A\\
     & = \int_\Omega \chi^2\partial_t\tilde{h} X_A^2 + 2\int_\Omega \chi^2 \tilde{h} X_A \left[\partial^\alpha (d_A\Delta u_A + f_u u_A + \frac{1}{\eps}Q)\right]\\
     & - \sum_{\gamma + \beta = \alpha\text{ and } |\gamma| > 0}  \frac{2}{S}\binom{\a}{\beta}\int_\Omega \chi^2 \tilde{h} X_A \bigg[\partial^\gamma \partial_t\tilde{k} \partial^\beta u\\
     &+\partial^\gamma \tilde{k} \partial^\beta (d_A\Delta u_A + (d_A + d_B)\Delta u_B + f_u(u_A + u_B))\bigg].
  \end{align*}

  We first consider the term in which $\frac{1}{\eps}$ appears. We recall that $\tilde{k} + \tilde{h} = S$, thus for $\gamma > 0$, we have $\partial^\gamma \tilde{k} + \partial^\gamma \tilde{h} = 0$.
  Hence
  \begin{align*}
    \partial^\alpha Q
     & = \sum_{\gamma + \beta = \alpha} \left[\binom{\a}{\beta}\partial^\gamma \tilde{k} \partial^\beta u_B - \binom{\a}{\beta}\partial^\gamma \tilde{h} \partial^\beta u_A \right]\\
     & = \tilde{k}\partial^\alpha u_B - \tilde{h}\partial^\alpha u_A + \sum_{\gamma + \beta = \alpha \text{ and }\gamma > 0} \left[\binom{\a}{\beta}\partial^\gamma \tilde{k} \partial^\beta u_A + \binom{\a}{\beta}\partial^\gamma \tilde{k} \partial^\beta u_B \right]\\
     & = \tilde{k}\partial^\alpha u_B - \tilde{h}\partial^\alpha u_A + \frac{\tilde h + \tilde k }{S}\sum_{\gamma + \beta = \alpha \text{ and } \gamma > 0} \binom{\a}{\beta}\partial^\gamma \tilde{k} \partial^\beta u                    \\
     & = \tilde{k}\partial^\alpha u_B - \frac{\tilde k }{S}\sum_{\gamma + \beta = \alpha\text{ and } \gamma > 0} \binom{\a}{\beta}\partial^\gamma \tilde{h} \partial^\beta u
     - \tilde{h}\partial^\alpha u_A + \frac{\tilde h}{S}\sum_{\gamma + \beta = \alpha \text{ and } \gamma > 0}\binom{\a}{\beta} \partial^\gamma \tilde{k} \partial^\beta u\\
     & = \tilde{k}X_B - \tilde{h}X_A.
  \end{align*}

  Hence, we obtain
  \begin{align*}
    \frac{1}{\eps}\int_\Omega \chi^2 \tilde{h} X_A \partial^\alpha Q - \frac{1}{\eps}\int_\Omega \chi^2 \tilde{k} X_B \partial^\alpha Q
     & = - \frac{1}{\eps} \int_\Omega \chi^2 |\partial^\alpha Q|^2.
  \end{align*}

  For the next term, we perform an integration by part, that holds thanks to the cutoff function $\chi$ (which vanishes outside $\Omega_2$)
  \begin{align*}
    \int_\Omega \chi^2 \tilde{h} \partial^\alpha u_A \partial^\alpha \Delta u_A
     & =  - \int_\Omega \chi\tilde{h} \nabla (\chi \partial^\alpha u_A)\cdot \nabla \partial^\a u_A - \int_\Omega \nabla(\chi \tilde{h})\cdot \nabla \partial^\alpha u_A(\chi\partial^\alpha u_A)                                                            \\
     & = - \int_\Omega \tilde{h} |\nabla (\chi \partial^\alpha u_A)|^2 + \int_\Omega \tilde{h}\nabla \chi\cdot\nabla(\chi\partial^\a u_A) \partial^\alpha u_A                                                                                                \\
     & - \int_\Omega \nabla (\chi\tilde{h})\cdot \nabla(\chi \partial^\alpha u_A)\partial^\alpha u_A + \int_\Omega \nabla (\chi\tilde{h})\cdot \nabla\chi|\partial^\alpha u_A|^2.
  \end{align*}

  We now sum $\frac{\dd }{\dd t}\Ecal_{A, \alpha}$ with the similar expression for $\Ecal_{B, \a}$ and obtain
  \begin{align*}
    \frac{\dd }{\dd t}\Ecal_{A, \alpha}(t) + \frac{\dd }{\dd t}\Ecal_{B, \alpha}
     & = -\int_\Omega \left( \tilde{h} |\nabla (\chi \partial^\alpha u_A)|^2 + \tilde{k} |\nabla (\chi \partial^\alpha u_B)|^2\right)- \frac{2}{\eps} \int_\Omega \chi^2 |\partial^\alpha Q|^2 \\
     & + I_A + I_B,
  \end{align*}

  where
  \begin{align*}
    I_A
     & :=d_A\int_\Omega \tilde{h}\nabla \chi\cdot\nabla(\chi\partial^\a u_A) \partial^\alpha u_A\\
     & - d_A\int_\Omega \nabla (\chi\tilde{h})\cdot \nabla(\chi \partial^\alpha u_A)\partial^\alpha u_A + d_A\int_\Omega \nabla (\chi\tilde{h})\cdot \nabla\chi|\partial^\alpha u_A|^2 \\
     & + \int_\Omega \chi^2\partial_t\tilde{h} X_A^2 + 2\int_\Omega \chi^2 \tilde{h} X_A \partial^\alpha (f_u u_A)\\
     & +2\int_\Omega \chi^2 \tilde{h} \left[- \frac{1}{S}\sum_{\gamma + \beta = \alpha \text{ and } |\gamma| > 0} \binom{\a}{\beta}\partial^\gamma \tilde{k} \partial^\beta u\right] \partial^\alpha (d_A \Delta u_A) \\
     & - \sum_{\gamma + \beta = \alpha \text{ and } |\gamma| > 0}  \frac{2}{S}\binom{\a}{\beta}\int_\Omega \chi^2 \tilde{h} X_A \left[\partial^\gamma \partial_t \tilde{k} \partial^\beta u +  \partial^\gamma \tilde{k} \partial^\beta (d_A\Delta u_A + (d_A + d_B)\Delta u_B + f_u(u_A + u_B))\right] \\
     & =: I_1 + I_2 + I_3 + I_4 + I_5 + I_6 + I_7.
  \end{align*}
  We can similarly define $I_B$.\medskip

  We will now bound $\int_0^tI_A$. In the remaining of the proof, $\eta, \delta > 0$ are small parameters to be chosen later. We start with some general bounds.\medskip

  We start by interpolating $\fsL^{\frac{10}{3}-}(\Omega\times[0, T])$ between $\fsL^{\infty-}([0, T], \fsL^2(\Omega))$ and $\fsL^{2}([0, T], \fsL^6(\Omega))$, to deduce that for some $0 <\theta_1 < 1$, we have
  \begin{align*}
    \|\chi\partial^\a u\|_{\fsL^{\frac{10}{3}-}(\Omega_T)}
    &\ioe \|\chi\partial^\a u\|_{\fsL^{2}([0, T], \fsL^6(\Omega))}^{(1 - \theta_1)}\|\chi\partial^\a u\|_{\fsL^{\infty-}([0, T], \fsL^2(\Omega))}^{\theta_1}\\
    &\ioe  C\|\nabla(\chi\partial^\a u)\|_{\fsL^{2}([0, T], \fsL^2(\Omega))}^{(1 - \theta_1)}\|\chi\partial^\a u\|_{\fsL^{\infty-}([0, T], \fsL^2(\Omega))}^{\theta_1}.
  \end{align*}

  Then, we interpolate  $\fsL^{\infty-}([0, T], \fsL^2(\Omega))$ between $\fsL^{\infty}([0, T], \fsL^2(\Omega))$ and $\fsL^{2}([0, T], \fsL^2(\Omega))$, to get that for some $0 <\theta_2 < 1$, we have
  \begin{align*}
    \|\chi\partial^\a u\|_{\fsL^{\frac{10}{3}-}(\Omega_T)}
    &\ioe  C\|\nabla(\chi\partial^\a u)\|_{\fsL^{2}([0, T], \fsL^2(\Omega))}^{(1 - \theta_1)}\|\chi\partial^\a u\|_{\fsL^{\infty}([0, T], \fsL^2(\Omega))}^{\theta_1\theta_2}\|\chi\partial^\a u\|_{\fsL^{2}([0, T], \fsL^2(\Omega))}^{\theta_1(1 - \theta_2)}.
  \end{align*}

  Using the induction hypothesis (applied on $\Omega_2$) and Young's inequality, we obtain that for some constant $C_{\eta, \delta}$
  \begin{equation}\label{eq:bound2plus}
    \|\chi\partial^\alpha u\|^{2}_{\fsL^{\frac{10}{3}-}(\Omega_T)}
    \ioe \eta\|\nabla(\chi\partial^\alpha u)\|_{\fsL^{2}(\Omega_T)}^2 + \delta \|\chi\partial^\alpha u \|_{\fsL^{\infty}([0, T], \fsL^{2}(\Omega))}^2 + C_{\eta, \delta}.
  \end{equation}
  The same inequality holds for $u_C$.\medskip

  As a consequence, we will be able to bound $\chi X_A$ in $\fsL^{\frac{5}{2}}(\Omega_T)$ (better bounds hold, but the bound in $\fsL^{\frac{5}{2}}(\Omega_T)$ suffices for our purpose). Using \eqref{eq:HR3} and \eqref{eq:boundvl} to bound $\partial^\gamma\tilde k\partial^\beta u$ in $\fsL^{\frac{5}{2}}(\Omega_T)$ for $\gamma > 0$ and noticing that $\frac52<\frac{10}{3}$, we obtain thanks to \eqref{eq:bound2plus} that for any $\eta, \delta > 0$ we have
  \begin{equation}\label{eq:boundXA2+}
    \|\chi X_A\|^2_{\fsL^{{\frac{5}{2}}}(\Omega_2\times[0, T])} \ioe \eta\|\nabla(\chi\partial^\alpha u_A)\|_{\fsL^{2}(\Omega_T)}^2 + \delta \|\chi\partial^\alpha u_A \|_{\fsL^{\infty}([0, T], \fsL^{2}(\Omega))}^2 + C_{\eta, \delta}.
  \end{equation}

  Similarly, by the induction hypothesis (applied on $\Omega_2$) we also immediately have
  \begin{equation}\label{eq:boundXA2}
    \|X_A\|_{\fsL^{2}(\Omega_2\times[0, T])} \ioe C.
  \end{equation}

  We now go back to bounding $I_A$. The first three terms can be bounded with similar arguments. Using the induction hypothesis \eqref{eq:HRu} and Young's inequality, one obtains
  \begin{align*}
    \int_0^t I_1
     & = d_A\int_0^t\int_\Omega \tilde{h}\nabla \chi\cdot\nabla(\chi\partial^\a u_A) \partial^\alpha u_A \\
     & \ioe d_A\|\tilde{h}\|_{\fsL^{\infty}(\Omega_T)}\|\nabla\chi\|_{\fsL^{\infty}(\Omega_T)}\|\partial^\a u_A\|_{\fsL^{2}(\Omega_2\times[0, T])}\|\nabla(\chi\partial^\a u_A)\|_{\fsL^{2}(\Omega_2\times[0, T])} \\
     & \ioe \eta\|\nabla(\chi\partial^\a u_A)\|_{\fsL^{2}(\Omega_2\times[0, T])}^2 + C_\eta.
  \end{align*}

  The second term can be bounded similarly, using this time that $\tilde h, \chi \in \fsW^{1, \infty}(\Omega_T)$ by \eqref{eq:initial-estimates-vinf}
  \begin{align*}
    \int_0^t I_2
     & = - d_A \int_\Omega \nabla (\chi\tilde{h})\cdot \nabla(\chi \partial^\alpha u_A)\partial^\alpha u_A \\
     & \ioe \eta\|\nabla(\chi\partial^\a u_A)\|_{\fsL^{2}(\Omega_2\times[0, T])}^2 + C_\eta.
  \end{align*}

  In the same way, we obtain a bound for the next term:
  \begin{align*}
    \int_0^tI_3
     & = d_A\int_0^t\int_\Omega \nabla (\chi\tilde{h})\cdot \nabla\chi|\partial^\alpha u_A|^2 \\
     & \ioe C \|\partial^\a u_A\|_{\fsL^{2}(\Omega_2\times[0, T])}^2                       \\
     & \ioe C.
  \end{align*}

  We can now tackle the fourth term, using \eqref{eq:initial-estimates-v} and \eqref{eq:boundXA2+}
  \begin{align*}
    \int_0^t I_4
     & =\int_\Omega \chi^2\partial_t\tilde{h} X_A^2                                                                                                                          \\
     & \ioe \|\partial_t\tilde{h}\|_{\fsL^{\infty-}(\Omega_T)}\|\chi^2 X_A^2\|_{\fsL^{1+}(\Omega_2\times[0, T])}                                                             \\
     & \ioe \eta\|\nabla(\chi\partial^\alpha u_A)\|_{\fsL^{2}(\Omega_t)}^2 + \delta \|\chi\partial^\alpha u_A \|_{\fsL^{\infty}([0, T], \fsL^{2}(\Omega))}^2 + C_{\eta, \delta}.
  \end{align*}

  We now consider $\int_0^t I_5 = 2\int_0^t\int_\Omega \chi^2 \tilde{h} X_A \partial^\alpha (f_u u_A)$. We state independently a bound on $\partial^\alpha (f_u u_A) = r_u \partial^\a u_A - \sum_{\beta + \gamma = \a} \left(\binom{\a}{\beta}d_{11}\partial^\beta u\partial^\gamma u_A + \binom{\a}{\beta}d_{12}\partial^\beta v\partial^\gamma u_A\right)$. If $\beta = \alpha$ or $\gamma = \a$, we can use \eqref{eq:initial-estimates-uLq} and \eqref{eq:HRu} to deduce that $u_A\partial^\a u, u\partial^\a u_A \in \fsL^{2-}(\Omega_2\times[0, T])$ and the same holds for $v\partial^\a u_A$ and $u_A\partial^\a v$.\medskip

  Now for $\beta < \alpha$, we interpolate and use a Sobolev injection (in dimension $d\ioe 4$) to obtain
  \begin{align}\label{eq:boundu4}
    \|\chi\partial^\beta u\|_{\fsL^4(\Omega_T)}
     & \ioe \|\chi\partial^\beta u\|^\frac12_{\fsL^\infty([0, T], \fsL^4(\Omega_2))}\|\chi\partial^\beta u\|^\frac12_{\fsL^2([0, T], \fsL^4(\Omega_2))}\nonumber\\
     & \lesssim \|\nabla(\chi \partial^\beta u)\|^\frac12_{\fsL^\infty([0, T], \fsL^2(\Omega_2))}\|\nabla (\chi\partial^\beta u)\|^\frac12_{\fsL^2([0, T], \fsL^2(\Omega_2))}\nonumber \\
     & \lesssim \left(\|\chi \nabla^{\otimes l} u\|^\frac12_{\fsL^\infty([0, T], \fsL^2(\Omega_2))}+1\right)\left(\|\chi \nabla^{\otimes l} u\|^\frac12_{\fsL^2([0, T], \fsL^2(\Omega_2))} + 1\right).
  \end{align}
  To obtain the last line, we used Poincaré inequality, the smoothness of $\chi$ and the recursion hypothesis \eqref{eq:HRu} to bound the remaining terms $\nabla^{\otimes k_1}\chi \nabla^{\otimes k_2}u$ with $k_1 + k_2 = l$, $k_1 > 0$. The same result holds for $\chi\partial^\gamma u_A$ for $\gamma < \a$. Finally using again the induction hypothesis \eqref{eq:HRu}, we have shown that for $0 < \gamma, \beta < \a$
  \begin{equation}\label{eq:boundfu}
    \|\chi^2 \partial^\beta u \partial^\gamma u_A\|_{\fsL^{2}(\Omega_2 \times[0, T])} \ioe C (\|\chi \nabla^{\otimes l} u\|_{\fsL^\infty([0, T], \fsL^2(\Omega_2))} + \|\chi \nabla^{\otimes l} u_A\|_{\fsL^\infty([0, T], \fsL^2(\Omega_2))} + 1).
  \end{equation}

  The terms involving some $\partial^\beta v$ can be easily bounded using the regularity of $v$, more precisely \eqref{eq:boundvlm1} and \eqref{eq:boundvl}.

  We can now use the previous estimates as well as \eqref{eq:boundXA2+} and \eqref{eq:boundXA2} to deduce that
  \begin{align}\label{eq:I5}
    \int_0^t I_5
     & \lesssim \sum_{\beta + \gamma =\alpha \text{ and } \beta, \gamma < \a} \|\chi^2 (\partial^\beta u \partial^\gamma u_A + \partial^\beta v \partial^\gamma u_A)\|_{\fsL^{2}(\Omega_2 \times[0, T])}\|X_A\|_{\fsL^{2}(\Omega_2 \times[0, T])} \\
     & + \|\chi (u\partial^\a u_A + u_A\partial^\a u + v\partial^\a u_A + u_A\partial^\a v)\|_{\fsL^{2-}(\Omega_2 \times[0, T])}\|\chi X_A\|_{\fsL^{2+}(\Omega_2 \times[0, T])} + 1\nonumber\\
     & \lesssim \sum_{\beta + \gamma =\alpha \text{ and } \beta, \gamma < \a} \delta\|\chi^2 \partial^\beta u \partial^\gamma u_A\|^2_{\fsL^{2}(\Omega_2 \times[0, T])} + C_\delta\\
     & + (\|u\|_{\fsL^{\infty-}(\Omega_T)}+\|u_A\|_{\fsL^{\infty-}(\Omega_T)}+\|v\|_{\fsL^{\infty}(\Omega_T)})\\
     &\times(\|\chi \partial^\a u\|_{\fsL^{2}(\Omega_2 \times[0, T])} + \|\chi \partial^\a u_A\|_{\fsL^{2}(\Omega_2 \times[0, T])})\|\chi X_A\|_{\fsL^{2+}(\Omega_2 \times[0, T])}\nonumber\\
     & \ioe \eta\|\nabla(\chi\partial^\alpha u_A)\|_{\fsL^{2}(\Omega_T)}^2 + \delta (\|\chi\nabla^{\otimes l} u \|_{\fsL^{\infty}([0, T], \fsL^{2}(\Omega))}^2+\|\chi\nabla^{\otimes l} u_A \|_{\fsL^{\infty}([0, T], \fsL^{2}(\Omega))}^2) + C_{\eta, \delta}.\nonumber
  \end{align}

  We now consider
  \begin{align*}
    \int_0^t I_6
     & = - \frac{2d_A}{S}\sum_{\gamma + \beta = \alpha \text{ and } |\gamma| > 0}\binom{\a}{\beta} \int_0^t \int_\Omega \chi \tilde{h} \partial^\gamma \tilde{k} \partial^\beta u \chi\partial^\alpha \Delta u_A  \\
     & = \frac{2d_A}{S}\sum_{\gamma + \beta =\alpha \text{ and } |\gamma| > 0} \binom{\a}{\beta}\int_0^t \int_\Omega \left[\nabla(\chi \tilde{h} \partial^\gamma \tilde{k} \partial^\beta u ) +  \tilde{h} \partial^\gamma \tilde{k} \partial^\beta u\nabla \chi\right]\cdot \chi\nabla \partial^\alpha u_A                                     \\
     & = \frac{2d_A}{S}\sum_{\gamma + \beta =\alpha \text{ and } |\gamma| > 0} \binom{\a}{\beta}\int_0^t \int_\Omega \left[\nabla(\chi \tilde{h} \partial^\gamma \tilde{k} \partial^\beta u ) +  \tilde{h} \partial^\gamma \tilde{k} \partial^\beta u\nabla \chi\right]\cdot (\nabla (\chi\partial^\alpha u_A) - \partial^\alpha u_A\nabla\chi) \\
     & \ioe \eta (\|\nabla (\chi\partial^\alpha u_A)\|_{\fsL^2(\Omega_T)} + \|\partial^\alpha u_A\nabla\chi\|_{\fsL^2(\Omega_T)})^2\\
     &+ \sum_{\gamma + \beta =\alpha \text{ and } |\gamma| > 0} C_\eta
    \|\nabla(\chi \tilde{h} \partial^\gamma \tilde{k} \partial^\beta u ) +  \tilde{h} \partial^\gamma \tilde{k} \partial^\beta u\nabla \chi\|^2_{\fsL^2(\Omega_T)}.
  \end{align*}

  Thanks to the induction hypothesis, we know that $\|\partial^\alpha u_A\nabla\chi\|_{\fsL^2(\Omega_T)} \ioe C$. We can also bound (we consider two different cases, depending on whether $\beta = 0$ or not)
  \begin{align*}
     & \sum_{\gamma + \beta =\alpha \text{ and } |\gamma| > 0}\|\tilde{h} \partial^\gamma \tilde{k} \partial^\beta u\nabla \chi\|^2_{\fsL^2(\Omega_T)}\\
     & \ioe C \|u\|_{\fsL^{4}(\Omega_T)}^2\|\partial^\a v\|_{\fsL^{4}(\Omega_2\times[0, T])}^2+ C\|\nabla^{\otimes (l - 1)}u\|_{\fsL^2(\Omega_2\times[0, T])}^2\|\nabla^{\otimes (l - 1)}v\|_{\fsL^\infty(\Omega_2\times[0, T])}^2\\
     & \ioe C.
  \end{align*}
  To obtain the last line, we used \eqref{eq:boundvlm1} and \eqref{eq:HRu}.\medskip

  We will now bound $S_{\gamma, \beta} := \nabla(\chi \tilde{h} \partial^\gamma \tilde{k} \partial^\beta u)$ (for $\gamma + \beta = \a$, $|\gamma|>0$).
  Let us start with $\gamma = \a$. Using \eqref{eq:boundvlp1} and \eqref{eq:initial-estimates-uLq} we have that $\nabla(\partial^\alpha \tilde{k}) u\in\fsL^2(\Omega_2\times [0, T])$. Similarly, $\partial^\alpha \tilde{k} \nabla u\in\fsL^2(\Omega_2\times [0, T])$ (we use \eqref{eq:boundvl} and \eqref{eq:HR3}). Thus $S_{\a, 0}\in\fsL^2(\Omega_2\times[0, T])$.
  For the other terms $S_{\gamma, \beta}$, one notices that $\partial^\gamma \tilde k \nabla\partial^\beta u \in \fsL^{2}(\Omega_2\times[0, T])$ by \eqref{eq:boundvlm1} and \eqref{eq:HRu}. Similarly $\nabla\partial^\gamma \tilde k \partial^\beta u \in \fsL^{2}(\Omega_2\times[0, T])$ by \eqref{eq:boundvl} and \eqref{eq:HR3}. We have thus shown that $S_{\gamma, \beta}\in\fsL^2(\Omega_2\times[0, T])$.

  Hence we deduce that
  \begin{align*}
    \int_0^t I_6
    \ioe \eta \|\nabla (\chi\partial^\alpha u_A)\|^2_{\fsL^2(\Omega_T)} + C_\eta.
  \end{align*}

  We can now bound the last term
  \begin{align*}
    \int_0^t I_7
    &= -\sum_{\gamma + \beta =\alpha \text{ and } |\gamma| > 0} \binom{\a}{\beta} \frac{2}{S}\int_0^t\int_\Omega \chi^2 \tilde{h} X_A \big[\partial^\gamma \partial_t \tilde{k} \partial^\beta u\\
    &+  \partial^\gamma \tilde{k} \partial^\beta (d_A\Delta u_A + (d_A + d_B)\Delta u_B + f_u(u_A + u_B))\big].
  \end{align*}

  Similarly as in  \eqref{eq:HR3}, we see that for $\theta =\frac{3}{4}$, we have
  \begin{align*}
    \|\chi \nabla^{\otimes (l - 1)}u\|_{\fsL^{8}(\Omega_T)}^2
    &\ioe \|\chi\nabla^{\otimes (l-1)}u\|^{2\theta}_{\fsL^\infty([0, T], \fsL^6({\Omega_2}))} \|\chi\nabla^{\otimes (l-1)}u\|_{\fsL^2([0, T], \fsL^{\infty}(\Omega_2))}^{2(1 - \theta)}\\
    &\lesssim \|\chi\nabla^{\otimes l}u\|^2_{\fsL^\infty([0, T], \fsL^2({\Omega_2}))} + \|\nabla(\chi\nabla^{\otimes l}u)\|_{\fsL^2({\Omega_2\times[0, T]})}^2 + 1,
  \end{align*}
  where we obtain the last line thanks to a Sobolev embedding (and we use the smoothness of $\chi$ and \eqref{eq:HRu} to bound the remaining terms $\nabla^{\otimes k_1}\chi\nabla^{\otimes k_2}u$).

  For  $0 < \gamma < \alpha$, using \eqref{eq:boundXA2}, \eqref{eq:boundvlp1} and the previous estimate (recalling $|\beta|\ioe l - 1$) we obtain that
  \begin{align*}
    \int_0^t\int_\Omega \chi^2 \tilde{h} X_A \partial^\gamma \partial_t \tilde{k} \partial^\beta u
    &\ioe C\|\chi X_A\|_{\fsL^{2}(\Omega_T)}\|\chi \partial^\beta u\|_{\fsL^{8}({\Omega_T})}\|\partial^\gamma \partial_t \tilde{k}\|_{\fsL^{\frac{10}{3}}({\Omega_2\times[0, T]})}\\
    &\ioe \delta\|\chi \partial^\beta u\|_{\fsL^{8}(\Omega_T)}^2 + C_\delta\\
    &\ioe \delta \left(\|\chi\nabla^{\otimes l}u\|^2_{\fsL^\infty([0, T], \fsL^2({\Omega}))} + \|\nabla(\chi\nabla^{\otimes l}u)\|_{\fsL^2({\Omega_T})}^2\right) + C_{\delta}.
  \end{align*}
  
  We have to bound $\int_0^t\int_\Omega \chi^2 \tilde{h} u X_A \partial^\a \partial_t \tilde{k} = - \int_0^t\int_\Omega \partial_i (\chi^2 \tilde{h} u X_A) \partial^{\a'} \partial_t \tilde{k}$, where $i$ is any index such that $\alpha_i > 0$ and $\a' + e_i = \a$ (where $e_i$ is the multi-index such that the $i$-th index is $1$ and all the others are $0$).
  But
  \begin{equation*}
    \partial_i (\chi^2 \tilde{h} u X_A) = \partial_i (\chi X_A)\chi u\tilde h + \chi X_A \chi u \partial_i \tilde h + \chi X_A \tilde h \partial_i (\chi u).
  \end{equation*}
  We notice that $\frac{3}{10} + \frac{3}{10} +\frac{2}{5} = 1 $, thus thanks to \eqref{eq:HR3}, \eqref{eq:boundvlp1}, \eqref{eq:boundXA2+}, we have
  \begin{align*}
    &- \int_0^t\int_\Omega \left[\chi X_A \chi u \partial_i \tilde h + \chi X_A \tilde h \partial_i (\chi u)\right]\partial^{\a'} \partial_t \tilde{k}\\
    &\ioe C\|\chi X_A\|_{\fsL^{{\frac{5}{2}}}(\Omega_2\times[0, T])}\|\partial^{\a'} \partial_t \tilde{k}\|_{\fsL^{{\frac{10}{3}}}(\Omega_2\times[0, T])}(\| u\|_{\fsL^{{\frac{10}{3}}}(\Omega_2\times[0, T])}+ \|\partial_i u\|_{\fsL^{{\frac{10}{3}}}(\Omega_2\times[0, T])})\\
    &\ioe \eta\|\nabla(\chi\partial^\alpha u_A)\|_{\fsL^{2}(\Omega_T)}^2 + \delta \|\chi\partial^\alpha u_A \|_{\fsL^{\infty}([0, T], \fsL^{2}(\Omega))}^2 + C_{\eta, \delta}.
  \end{align*}

  Then, we can bound
  \begin{align*}
    &-\int_0^t\int_\Omega \partial_i (\chi X_A)\chi u\tilde h\partial^{\a'}\partial_t\tilde k\\
     & \ioe \eta \|\partial_i (\chi X_A)\|^2_{\fsL^{2}(\Omega_T)} + C_\eta\|u\tilde h\|^2_{\fsL^{5}(\Omega_T)}\|\chi \partial^{\a'}\partial_t\tilde k\|_{\fsL^{\frac{10}{3}}(\Omega_T)}^2\\
     & \ioe \eta \left(\|\nabla (\chi \partial^\alpha u_A)\|^2_{\fsL^2(\Omega_T)} + C\sum_{\gamma' + \beta' =\alpha \text{ and } \gamma' > 0} \|\partial_i (\chi(\partial^{\gamma'} \tilde k)(\partial^{\beta'} u))\|^2_{\fsL^2(\Omega_T)}\right)+ C_\eta.
  \end{align*}

  In the case $\gamma' < \alpha$, $\partial_i (\partial^{\gamma'} \tilde k)$ is bounded in $\fsL^{10}(\Omega_2\times[0, T])$ by \eqref{eq:boundvl} and $\partial^{\beta'} u$ is bounded in $\fsL^\frac{10}{3}(\Omega_2\times[0, T])$ by \eqref{eq:HR3}, furthermore $(\partial^{\gamma'} \tilde k)$ is bounded in $\fsL^\infty(\Omega_2\times[0, T])$ by \eqref{eq:boundvlm1} and $\|\partial_i \partial^{\beta'} u\|_{\fsL^2(\Omega_2\times[0, T])} \ioe C$ by \eqref{eq:HRu} (note that $1 + |\beta'| \ioe l$). Thus, we have $\|\partial_i (\chi(\partial^{\gamma'} \tilde k)(\partial^{\beta'} u))\|_{\fsL^2(\Omega_T)}\ioe C$.\medskip

  If $\gamma' = \alpha$, we have $\partial_i\partial^\a\tilde k \in\fsL^{\frac{10}{3}}(\Omega_2\times[0, T])$ by \eqref{eq:boundvlp1} and $u\in\fsL^{\infty-}(\Omega_T)$ by \eqref{eq:initial-estimates-uLq}. Similarly, $\partial^\a \tilde k\in\fsL^{10}(\Omega_2\times[0, T])$ by \eqref{eq:boundvl}  and $\partial_i u \in\fsL^{\frac{10}{3}}(\Omega_2\times[0, T])$ by \eqref{eq:HR3}. Hence, we obtain $\|\partial_i (\chi(\partial^{\alpha} \tilde k) u)\|_{\fsL^2(\Omega_T)}\ioe C$.\medskip

  Thus, we have shown
  \begin{align*}
    -\int_0^t\int_\Omega \partial_i (\chi X_A)\chi u\tilde h\partial^{\a'}\partial_t\tilde k
     & \ioe \eta \|\nabla (\chi \partial^\alpha u_A)\|^2_{\fsL^2(\Omega_T)} + C_\eta.
  \end{align*}

  Hence, we obtain
  \begin{align*}
    \int_0^t\int_\Omega \chi^2 \tilde{h} u X_A \partial^\a \partial_t \tilde{k}
     & \ioe \eta \|\nabla (\chi \partial^\alpha u_A)\|^2_{\fsL^2(\Omega_T)} +  \delta \|\chi\partial^\alpha u_A \|_{\fsL^{\infty}([0, T], \fsL^{2}(\Omega))}^2 + C_{\eta, \delta}.
  \end{align*}

  We then consider $\int_0^t\int_\Omega \chi^2 \tilde{h} X_A \partial^\gamma \tilde{k} \partial^\beta (d_A\Delta u_A + (d_A + d_B)\Delta u_B)$. For the case when $\gamma < \alpha$, remembering that $\chi\tilde{h}\partial^\gamma \tilde{k}$ is bounded in $\fsL^\infty(\Omega_T)$ by \eqref{eq:boundvlm1}, we can deduce thanks to \eqref{eq:boundXA2} that
  \begin{align*}
    &\int_0^t\int_\Omega \chi^2 \tilde{h} X_A \partial^\gamma \tilde{k} \partial^\beta (d_A\Delta u_A + (d_A + d_B)\Delta u_B)\\
     & \ioe \eta (\|\chi\partial^\beta \Delta u_A\|^2_{\fsL^2(\Omega_T)} + \|\chi\partial^\beta \Delta u_B\|^2_{\fsL^2(\Omega_T)}) + C_\eta \|\chi X_A\|^2_{\fsL^2(\Omega_T)} \\
     & \ioe \eta (\|\chi\nabla^{\otimes (l+1)} u_A\|^2_{\fsL^2(\Omega_T)} + \|\chi\nabla^{\otimes (l+1)} u_B\|^2_{\fsL^2(\Omega_T)}) + C_\eta.
  \end{align*}

  We recall that for any $|\a'| = l$, thanks to \eqref{eq:HRu},
  \begin{align*}
    \|\chi\nabla\partial^{\a'} u_C\|^2_{\fsL^2(\Omega_2\times[0, T])}
     & \ioe \|\nabla(\chi\partial^{\a'} u_C)\|^2_{\fsL^2(\Omega_2\times[0, T])} + \|\nabla\chi \partial^{\a'} u_C\|^2_{\fsL^2(\Omega_2\times[0, T])} \\
     & \ioe \|\nabla(\chi\partial^{\a'} u_C)\|^2_{\fsL^2(\Omega_2\times[0, T])} + C.
  \end{align*}

  Thus, we have shown that when $0<\gamma<\a$,
  \begin{align*}
     & \int_0^t\int_\Omega \chi^2 \tilde{h} X_A \partial^\gamma \tilde{k} \partial^\beta (d_A\Delta u_A + (d_A + d_B)\Delta u_B)                                       \\
     & \ioe \eta (\|\nabla(\chi\nabla^{\otimes l} u_A)\|^2_{\fsL^2(\Omega_2\times[0, T])} + \|\nabla(\chi\nabla^{\otimes l} u_B)\|^2_{\fsL^2(\Omega_2\times[0, T])}) + C_\eta.
  \end{align*}

  We now tackle the case when $\gamma = \a$. We recall that $l \soe 2$, thus (after using a Poincaré inequality and the smoothness of $\chi$ as well as \eqref{eq:HRu} to bound the remaining terms $\nabla^{\otimes k_1}\chi \nabla^{\otimes k_2}u$ with $k_1 + k_2 = l$, $k_1 > 0$) similar computations as in \eqref{eq:HR3} show that
  \begin{equation*}
    \|\chi\Delta u_C\|^2_{\fsL^\frac{10}{3}(\Omega_T)} \ioe \|\chi\partial^\a u_C\|_{\fsL^\infty([0, T], \fsL^2(\Omega))}^2 + \|\nabla(\chi\partial^\a u_C)\|_{\fsL^2([0, T], \fsL^2(\Omega))}^2,
  \end{equation*}
  using \eqref{eq:boundXA2}, \eqref{eq:boundvl} and Young's inequality this is enough to conclude as previously that
  \begin{align*}
    &\int_0^t\int_\Omega \chi^2 \tilde{h} X_A \partial^\a \tilde{k} (d_A\Delta u_A + (d_A + d_B)\Delta u_B)\\
    &\ioe \delta\left(\|\chi\Delta u_A\|^2_{\fsL^\frac{10}{3}(\Omega_T)} + \|\chi\Delta u_B\|^2_{\fsL^\frac{10}{3}(\Omega_T)}\right) + C_\delta\\
    &\ioe \delta \bigg(\|\nabla(\chi\nabla^{\otimes l} u_A)\|^2_{\fsL^2(\Omega_2\times[0, T])} + \|\nabla(\chi\nabla^{\otimes l} u_B)\|^2_{\fsL^2(\Omega_2\times[0, T])}\\
    & + \|\chi\partial^\a u_A\|_{\fsL^\infty([0, T], \fsL^2(\Omega))}^2 + \|\chi\partial^\a u_B\|_{\fsL^\infty([0, T], \fsL^2(\Omega))}^2\bigg) + C_\delta.
  \end{align*}

  Finally,  we want to bound $\int_0^t\int_\Omega \chi^2 \tilde{h} X_A \partial^\gamma \tilde{k} \partial^\beta \left[f_u(u_A + u_B)\right]$. We first start with the case when $0<\gamma<\a$. We can use the same computations as in \eqref{eq:boundu4} to first deduce that for any ${\beta'} \in\N^d$
  \begin{equation*}
    \|\chi\partial^{\beta'} u\|_{\fsL^4(\Omega_T)} \lesssim \left(\|\chi\nabla^{\otimes (|{\beta'}| + 1)} u \|^{\frac12}_{\fsL^{\infty}([0, t], \fsL^{2}(\Omega))}+1\right)\left(\|\chi\nabla^{\otimes (|{\beta'}| + 1)} u \|^{\frac12}_{\fsL^{2}(\Omega_t)} + 1\right).
  \end{equation*}

  Hence, similarly as in the computations of \eqref{eq:boundfu}, we can first deduce thanks to the recursion hypothesis \eqref{eq:HRu} and \eqref{eq:boundvlm1} that for any $\beta_1 + \beta_2 = \beta$, we have (recalling that $0 < \gamma$, we get $|\beta|< l$)
   \begin{equation*}
    \|\chi^2\partial^{\beta_1} v\partial^{\beta_2}u\|_{\fsL^2(\Omega_T)} + \|\chi^2\partial^{\beta_1} u\partial^{\beta_2}u\|_{\fsL^2(\Omega_T)} \lesssim \|\chi\nabla^{\otimes l} u \|_{\fsL^{\infty}([0, T], \fsL^{2}(\Omega))} + 1.
  \end{equation*}

  Hence, thanks to Leibnitz formula, we obtain
  \begin{equation*}
    \|\chi^2\partial^{\beta} (f_uu)\|_{\fsL^2(\Omega_T)} \lesssim \|\chi\nabla^{\otimes l} u \|_{\fsL^{\infty}([0, T], \fsL^{2}(\Omega))} + 1.
  \end{equation*}

  Using \eqref{eq:boundvlm1} and \eqref{eq:boundXA2}, we deduce that
  \begin{align*}
    \int_0^t\int_\Omega \chi^2 \tilde{h} X_A \partial^\gamma \tilde{k} \partial^\beta f_u(u_A + u_B)
    & \ioe \delta \|\chi\nabla^{\otimes l} u \|_{\fsL^{\infty}([0, T], \fsL^{2}(\Omega))}^2 + C_{\delta}.\nonumber
  \end{align*}

  If $\gamma = \alpha$, we can use \eqref{eq:boundvl}, \eqref{eq:boundXA2} and \eqref{eq:initial-estimates-uLq} to deduce that
  \begin{align*}
    &\int_0^t\int_\Omega \chi^2 \tilde{h} X_A \partial^\a \tilde{k} f_u(u_A + u_B)\\
    &\lesssim \|X_A\|_{\fsL^{2}(\Omega_2\times[0, T])} \|\partial^\a \tilde{k}\|_{\fsL^{10}(\Omega_2\times[0, T])}\|f_u(u_A + u_B)\|_{\fsL^{\frac{5}{2}}(\Omega_2\times[0, T])}\\
    &\ioe C.
  \end{align*}
  Finally, we have shown that (using that $u_A + u_B = u$)
  \begin{align*}
    \int_0^t I_7
    &\ioe \delta (\|\chi\nabla^{\otimes l} u_A \|_{\fsL^{\infty}([0, T], \fsL^{2}(\Omega))}^2 +\|\chi\nabla^{\otimes l} u_B \|_{\fsL^{\infty}([0, T], \fsL^{2}(\Omega))}^2)\\
    &+ (\eta + \delta) \left( \|\nabla(\chi\nabla^{\otimes l} u_A)\|_{\fsL^{2}(\Omega_T)}^2 + \|\nabla(\chi\nabla^{\otimes l} u_B)\|_{\fsL^{2}(\Omega_T)}^2 \right)+ C_{\delta, \eta}.
  \end{align*}

  Let $\Ecal_l:=  \sum_{|\alpha| = l} \Ecal_{A, \alpha} +\Ecal_{B, \a}$. We collect all our estimates to obtain after integration on $[0, t]$:
  \begin{align*}
     & \Ecal_l(t) + \sum_{|\alpha| = l} \int_0^t\int_\Omega \left(|\nabla (\chi \partial^\alpha u_A)|^2 + |\nabla (\chi \partial^\alpha u_B)|^2\right) + \sum_{|\alpha| = l}\frac{2}{\eps} \int_0^t\int_\Omega \chi^2 |\partial^\alpha Q|^2 \\
     & \lesssim \Ecal(0) + (\eta+\delta) \sum_{|\alpha| = l}\int_0^T\int_\Omega \left(|\nabla (\chi \partial^\alpha u_A)|^2 + |\nabla (\chi \partial^\alpha u_B)|^2 \right)\nonumber \\
     & + \delta \sum_{|\alpha| = l}\left(\|\chi\partial^\alpha u_A \|_{\fsL^{\infty}([0, T], \fsL^{2}(\Omega))}^2 + \|\chi\partial^\alpha u_B \|_{\fsL^{\infty}([0, t], \fsL^{2}(\Omega))}^2\right)+ C_{\eta, \delta}.\nonumber
  \end{align*}

  We now take the supremum over $t\in [0, T]$, to deduce

  \begin{align}\label{fge}
    & \sup_{0 < t < T}\Ecal_l(t) + \sum_{|\alpha| = l} \int_0^T\int_\Omega \left(|\nabla (\chi \partial^\alpha u_A)|^2 + |\nabla (\chi \partial^\alpha u_B)|^2\right) + \sum_{|\alpha| = l}\frac{2}{\eps} \int_0^T\int_\Omega \chi^2 |\partial^\alpha Q|^2 \\
    & \lesssim \Ecal(0) + (\eta+\delta)\sum_{|\alpha| = l}\int_0^T\int_\Omega \left(|\nabla (\chi \partial^\alpha u_A)|^2 + |\nabla (\chi \partial^\alpha u_B)|^2 \right)\nonumber \\
    & + \delta \sum_{|\alpha| = l}\left(\|\chi\partial^\alpha u_A \|_{\fsL^{\infty}([0, T], \fsL^{2}(\Omega))}^2 + \|\chi\partial^\alpha u_B \|_{\fsL^{\infty}([0, T], \fsL^{2}(\Omega))}^2\right)+ C_{\eta, \delta}.\nonumber
  \end{align}

  To close the loop, we notice that for $0<t<T$, (using \eqref{eq:boundvlm1} and \eqref{eq:HRu} for $ 0 < \gamma < \a$)
  \begin{align*}
    \|\chi\partial^\alpha u_A (t)\|_{\fsL^{2}(\Omega)}^2
     & \lesssim \Ecal_{A, \alpha}(t) + \sum_{\gamma + \beta =\alpha \text{ and } \gamma > 0} \|\chi\partial^\gamma \tilde k\partial^\beta u \|_{\fsL^{\infty}([0, T], \fsL^{2}(\Omega_2))}^2 \nonumber\\
     & \lesssim \Ecal_{A, \alpha}(t) + 1 + \|\chi u\partial^\a v\|_{\fsL^{\infty}([0, T], \fsL^{2}(\Omega))}^2\nonumber \\
     & \lesssim \Ecal_{A, \alpha}(t) + 1 + \|\chi\partial^\a v \|_{\fsL^\infty([0, T], \fsL^{2+}(\Omega_2))}^2\|u\|_{\fsL^\infty([0, T], \fsL^{\infty-}(\Omega))}^2.
  \end{align*}

  We can now take the supremum over $0 < t < T$ to obtain thanks to \eqref{eq:initial-estimates-uLq}
  \begin{align}\label{eq:intermediate}
    \|\chi\partial^\alpha u_A\|_{\fsL^\infty([0, T], \fsL^{2}(\Omega))}^2
     & \ioe C\left(\sup_{0 < t < T} \Ecal_{A, \alpha}(t) + 1 + \|\chi\partial^\a v \|_{\fsL^\infty([0, T], \fsL^{2+}(\Omega_2))}^2\right).
  \end{align}

  By \eqref{eq:boundvlinf2+}, we have for any $\delta' > 0$
  \begin{align*}
    \|\chi\partial^\a v \|_{\fsL^\infty([0, T], \fsL^{2+}(\Omega_2))}^2
     & \ioe \delta' \|\chi\partial^\alpha u\|_{\fsL^{\infty}([0, T], \fsL^2(\Omega_2))}^2 + C_{\delta'} (1 + \|\nabla(\chi\partial^\alpha u)\|_{\fsL^2(\Omega_2\times[0, T])}^2).
  \end{align*}

  We can sum \eqref{eq:intermediate} with the equivalent expression for $u_B$ and choose $\delta' > 0$ small enough to obtain
  \begin{align}\label{b7}
    \|\chi\partial^\alpha u\|_{\fsL^\infty([0, T], \fsL^{2}(\Omega))}^2
    & \lesssim \|\chi\partial^\alpha u_A\|_{\fsL^\infty([0, T], \fsL^{2}(\Omega))}^2 +  \|\chi\partial^\alpha u_B\|_{\fsL^\infty([0, T], \fsL^{2}(\Omega))}^2\nonumber\\
    & \ioe C(\sup_{0 < t < T}\Ecal_{A, \alpha}(t) + \sup_{0 < t < T}\Ecal_{B, \alpha}(t) + \|\nabla(\chi\partial^\alpha u)\|_{\fsL^2(\Omega_2\times[0, T])}^2 + 1).
  \end{align}

  We can now inject \eqref{b7} in \eqref{fge} to obtain
  \begin{align*}
     & \sup_{0 < t < T }\Ecal_l(t) + \sum_{|\alpha| = l} \left(\int_0^T\int_\Omega |\nabla (\chi \partial^\alpha u_A)|^2 + |\nabla (\chi \partial^\alpha u_B)|^2\right) + \frac{1}{\eps} \int_0^T\int_\Omega \chi^2 |\nabla^{\otimes l} Q|^2 \\
     & \lesssim \Ecal_l(0) + \sum_{|\alpha| = l}\eta\int_0^T\int_\Omega |\nabla (\chi \partial^\alpha u_A)|^2 + |\nabla (\chi \partial^\alpha u_B)|^2 \nonumber           \\
     & + \delta \sum_{|\alpha| = l}C(\sup_{0 < t < T}\Ecal_{l}(t) + \|\nabla(\chi\partial^\alpha u)\|_{\fsL^2(\Omega_2 \times [0, T])}^2) + C_{\eta, \delta}.\nonumber
  \end{align*}
  We choose $\eta, \delta > 0$ small enough to get
  \begin{align*}
    \sup_{0 < t < T}\Ecal_l(t) + \sum_{|\alpha| = l} \left(\int_0^T\int_\Omega |\nabla (\chi \partial^\alpha u_A)|^2 + |\nabla (\chi \partial^\alpha u_B)|^2\right) + \frac{1}{\eps} \int_0^T\int_\Omega \chi^2 |\nabla^{\otimes l} Q|^2\ioe C + \Ecal_l(0).
  \end{align*}
  Using again \eqref{b7}, we derive \eqref{eq:HRu1}, which enables to finish the induction and concludes the proof of \cref{thm:regu2}.
\end{proof}

\section{Rate of convergence : proof of \texorpdfstring{\cref{thm:conv}}{Theorem 1}}\label{sec:convergence}
In this section we prove \cref{thm:conv}. Similarly as previously we cut the proof in two parts, one where we consider convergence in the whole domain in $\fsL^\infty\fsL^2$ and $\fsL^2\fsH^1$ and with a rate $O(\eps^\frac 12)$ and one where we restrict ourself to the interior but we enhance the spaces in which the convergence holds as well as the dependency with respect to $\eps_{\init, l}$.

We recall \eqref{eq:def} and we also define
\begin{equation}\label{def:maj}
  \left\{
  \begin{aligned}
     & U_B^\eps := u_B^\eps - u_B, \quad U_A^\eps := u_A^\eps - u_A,\\
     & H^\eps := h(v^\eps) - h(v), \quad K^\eps := k(v^\eps) - k(v),\\
     & F_u^\eps := f_u(u^\eps, v^\eps)u^\eps - f_u(u, v)u, \quad F_v^\eps := f_v(u^\eps, v^\eps)v^\eps - f_v(u, v)v.
  \end{aligned}
  \right.
\end{equation}
We also recall the definitions $u_A := \frac{k(v)}{S}u$ and $u_B := \frac{h(v)}{S}u$.

\subsection{Regularity of the limit problem}
We first start by showing some estimates on the limit problem \eqref{eq:SKT} in the spirit of \cite{desvillettes2024a}.

\begin{lemm}\label{thm:improved_regu}
  Let $d \ioe 4$ and $\Omega \subset \R^d$ be a smooth bounded domain, and let $d_u,d_v, \sigma >0$, $r_u,r_v\ge 0$, $d_{ij} > 0$ for $i,j=1,2$. Let $\a > 0$, and suppose that $u_\init, v_\init \ge 0$ are initial data which lie in $\Ccal^{2, \a}(\overline{\Omega})$ and are compatible with the Neumann boundary conditions.\medskip 

  Then, the global solutions $(u, v)$ to \eqref{eq:SKT}--\eqref{eq:SKTbc} constructed in \cref{thm:globalSKT} belong to $\Ccal^{2, \beta}(\overline{\Omega}\times[0, T])$ for some $\beta > 0$.\medskip
  
  In particular we have $u, \nabla u \in \fsL^\infty(\Omega_T)$.  
\end{lemm}

\begin{proof}
  We first compute
  \begin{align*}
    \partial_t \left[(d_u + \sigma v)u\right] - (d_u + \sigma v)\Delta((d_u + \sigma v) u ) = (d_u + \sigma v)f_u u + \sigma u \partial_t v \in \fsL^{\infty -}(\Omega_T),
  \end{align*}
  the last estimate coming from \cref{thm:globalSKT}.\medskip
  
  Hence, by \cref{thm:DGNM} applied on $w := (d_u + \sigma v)u$, we deduce that $w\in\Ccal^{0, \beta}(\overline{\Omega}\times[0, T])$ for some $\beta > 0$.\medskip

  Then, we use $\partial_t v - d_v\Delta v = f_v v\in\fsL^{\infty -}(\Omega_T)$ to deduce that $v\in\Ccal^{0, \beta}(\overline{\Omega}\times[0, T])$. Let $x, x'\in\Omega, t, t'\in [0, T]$, for any function $g(t, x)$, we denote $g:= g(t, x)$ and $g' := g(t', x')$ hence
  \begin{align*}
    \frac{|u - u'|}{|x - x'|^\beta + |t - t'|^\frac{\beta}{2}}
    &= \frac{|w(d_u + \sigma v') - w'(d_u + \sigma v)|}{(d_u + \sigma v)(d_u + \sigma v')}\frac{1}{|x - x'|^\beta + |t - t'|^\frac{\beta}{2}}\\
    &\ioe \frac{|(w - w')(d_u + \sigma v')| + w'\sigma |v - v'|}{d_u^2(|x - x'|^\beta + |t - t'|^\frac{\beta}{2})}\\
    &\ioe C.
  \end{align*}
  Hence, we deduce that $u \in\Ccal^{0, \beta}(\overline{\Omega}\times[0, T])$.\medskip

  If we go back to $v$, we see that $\partial_t v - d_v\Delta v \in\Ccal^{0, \beta}(\overline{\Omega}\times[0, T])$ and thus using $v_\init \in \Ccal^{2, \alpha}(\Omega_T)$, we deduce by Schauder estimates (up to changing $\beta$) that $v \in \Ccal^{2, \beta}(\overline{\Omega}\times[0, T])$, in particular that $\partial_t v \in \Ccal^{0, \beta}(\overline{\Omega}\times[0, T])$.\medskip

  We deduce from this that $\partial_t \left[(d_u + \sigma v)u\right] - (d_u + \sigma v)\Delta((d_u + \sigma v) u ) \in \Ccal^{0, \beta}(\overline{\Omega}\times[0, T])$, with  $v \in \Ccal^{2, \beta}(\overline{\Omega}\times[0, T])$. Thus, using Schauder estimates (for instance \cite[Theorem 4.31]{lieberman1996}), we obtain that $w = (d_u + \sigma v) u\in\Ccal^{2, \beta}(\overline{\Omega}\times[0, T])$. Let $1\ioe i \ioe d$, we have
  \begin{equation*}
    \partial_i w = (d_u + \sigma v) \partial_i u + \sigma u\partial_i v.
  \end{equation*}
  Thus $(d_u + \sigma v) \partial_i u\in \Ccal^{0, \beta}(\overline{\Omega}\times[0, T])$. Using that $d_u + \sigma v\soe d_u$, we see that $\partial_i u\in \Ccal^{0, \beta}(\overline{\Omega}\times[0, T])$. Then, let $1\ioe i,j  \ioe d$, we have
  \begin{equation*}
    \partial_{ij} w = (d_u + \sigma v)\partial_{ij} u + \sigma (\partial_i u\partial_j v + \partial_j u\partial_i v) + \sigma u\partial_{ij} v.
  \end{equation*}
  Thus, $(d_u + \sigma v)\partial_{ij} u\in \Ccal^{0, \beta}(\overline{\Omega}\times[0, T])$, and consequently $\partial_{ij} u\in \Ccal^{0, \beta}(\overline{\Omega}\times[0, T])$. Using \eqref{eq:SKT}, we see that $\partial_t u \in\Ccal^{0, \beta}(\overline\Omega)$, so $u\in\Ccal^{2, \beta}(\overline{\Omega}\times[0, T])$.
\end{proof}

\subsection{Part I: convergence in \texorpdfstring{$\fsL^\infty\fsL^2$ and $\fsL^2\fsH^1$}{LinftyL2 and L2H1}}
In this section we prove the 
\begin{lemm}\label{thm:conv1}
  Let $T>0$, $d\ioe 4$, $\Omega$ be a smooth  bounded domain of $\R^d$,
  $d_1,d_2,\sigma >0$, $r_u,r_v\ge 0$, $d_{ij} > 0$ for $1\ioe i, j \ioe 2$ and $d_A, d_B, h, k$ satisfying \eqref{eq:hk-relation}. Let $u_{A, \init}, u_{B, \init}, v_\init \soe 0$ belong to $\Ccal^{2, \alpha}(\overline\Omega)$ for some $\a > 0$ and compatible with the homogeneous Neumann boundary conditions. \medskip
  
 Let $u^\eps_A, u^\eps_B, v^\eps$ be the unique global strong solution (in $\Ccal^2(\R_+\times\overline\Omega)$) to \eqref{eq:SKT-eps} (with initial data $u_{A, \init}, u_{B, \init}, v_\init$), and $u, v$ be a solution to \eqref{eq:SKT} given by \cref{thm:conv} (with initial data $u_\init := u_{A, \init} + u_{B, \init}, v_\init$).\medskip

  Then there exists a constant $C_0$ that depends only on the parameters of \eqref{eq:SKT-eps}, $d, T, \Omega$ and the norm of $u_{A, \init}, u_{B, \init}, v_\init$ in $\Ccal^{2, \alpha}(\overline\Omega)$, but not on $\eps$, such that
  \begin{equation*}
    \|U^\eps\|_{\fsL^\infty([0, T], \fsL^2(\Omega))} + \|U^\eps\|_{\fsL^2([0, T], \fsH^1(\Omega))} \ioe C_0\eps^\frac12 ,
  \end{equation*}
  and 
  \begin{equation*}
    \|V^\eps\|_{\fsL^\infty([0, T], \fsL^2(\Omega))} + \|V^\eps\|_{\fsL^2([0, T], \fsH^1(\Omega))} \ioe C_0\eps^\frac12.
  \end{equation*}
\end{lemm}

\begin{proof}
One can check that $u^\eps$ respects the following equation :
\begin{equation*}
  \left\{\begin{aligned}
     & \partial_t u^\epsilon - d_A \,\Delta_x u^\epsilon = d_B\Delta_x u_B^\eps + f_u(u^\eps, v^\eps)\,u^\epsilon, \\
     & \nabla_x u^\epsilon \cdot n = 0  \quad \text{on } [0, T]\times\partial\Omega.
  \end{aligned}
  \right.
\end{equation*}

Furthermore, using \eqref{eq:hk-relation} and the definition of $u_B$, we deduce that
\begin{equation*}
  \partial_t u = d_A \Delta u + d_B \Delta u_B + f_u(u, v)u.
\end{equation*}
Hence, one has the following equations (using the notations introduced in \eqref{def:maj})
  \begin{equation}\label{eq:U}
    \left\{
    \begin{aligned}
       & \partial_t U^\eps = \Delta(d_A U^\eps + d_BU_B^\eps) + F_u^\eps,        \\
       & \nabla_x U^\epsilon \cdot n = 0 \quad \text{on } [0, T]\times\partial\Omega,
    \end{aligned}
    \right.
  \end{equation}

  and
  \begin{equation}\label{eq:V}
    \left\{
    \begin{aligned}
       & \partial_t V^\eps = d_v\Delta V^\eps + F_v^\eps,\\
       & \nabla_x V^\epsilon \cdot n = 0 \quad \text{on } [0, T]\times\partial\Omega.
    \end{aligned}
    \right.
  \end{equation}

  Let $0 < t < T$, we first multiply \eqref{eq:U} by $U^\eps$ and integrate over $\Omega_t$ to obtain
  \begin{equation*}
    \int_\Omega U^\eps(t)^2 + d_A \int_0^t\int_\Omega |\nabla U^\eps|^2 \ioe \int_\Omega U^\eps(0)^2 - d_B \int_0^t\int_\Omega \nabla U_B^\eps\cdot \nabla U^\eps + \int_0^t \int_\Omega F_u^\eps U^\eps.
  \end{equation*}

  We recall that \cref{thm:improved_regu} shows that $u$ is bounded in $\fsL^\infty(\Omega_T)$, thus an immediate computation shows that
  \begin{align*}
    \int_0^t \int_\Omega F_u^\eps U^\eps
     & = r_u \int_0^t \int_\Omega |U^\eps|^2 - d_{11}\int_0^t \int_\Omega |U^\eps|^2(u + u^\eps)\\
     &- d_{12}\int_0^t \int_\Omega |U^\eps|^2v^\eps -d_{12} \int_0^t \int_\Omega U^\eps V^\eps u \\
     & \ioe  r_u \int_0^t \int_\Omega |U^\eps|^2-d_{12} \int_0^t \int_\Omega U^\eps V^\eps u \\
     & \lesssim (1 + \|u\|_{\fsL^\infty(\Omega_t)}) \left( \int_0^t \int_\Omega |U^\eps|^2 + \int_0^t \int_\Omega |V^\eps|^2\right).
  \end{align*}

  Then, one has
  \begin{equation*}
    U_B^\eps = \frac{Q^\eps + H^\eps u + h(v^\eps) U^\eps}{S}.
  \end{equation*}
  Indeed one can compute
  \begin{align*}
     & Q^\eps + H^\eps u + h(v^\eps) U^\eps                                                        \\
     & = k(v^\eps)u_B^\eps - h(v^\eps)u_A^\eps + h(v^\eps)u - h(v)u + h(v^\eps)u^\eps - h(v^\eps)u \\
     & = k(v^\eps)u_B^\eps - h(v^\eps)u_A^\eps  - h(v)u + h(v^\eps)(u_A^\eps + u_B^\eps) \\
     & = (h(v^\eps) + k(v^\eps))u_B^\eps  - h(v)u \\
     & = S(u^\eps_B - u_B).
  \end{align*}

  Hence, one obtains
  \begin{align*}
     & - d_B \int_0^t\int_\Omega \nabla U_B^\eps\cdot \nabla U^\eps\\
     & = - \frac{d_B}{S} \left(\int_0^t\int_\Omega \nabla Q^\eps\cdot\nabla U^\eps + u\nabla H^\eps \cdot\nabla U^\eps + H^\eps \nabla u \cdot \nabla U^\eps + h(v^\eps) |\nabla U^\eps|^2 + U^\eps \nabla U^\eps\cdot \nabla (h(v^\eps))\right) \\
     & =: I_1 + I_2 + I_3 + I_4 + I_5.
  \end{align*}

  We can bound all those terms using Young's inequality. Let $\eta > 0$ be a small parameter to be fixed later, then
  \begin{align*}
    |I_1|
     & \ioe \eta \int_0^t\int_\Omega |\nabla U^\eps|^2 + C_\eta\int_0^t\int_\Omega |\nabla Q^\eps|^2.
  \end{align*}

  For the next term, we recall \cref{thm:improved_regu}, in particular we have that $\|u\|_{\fsL^\infty(\Omega_T)} \ioe C$. Hence, we obtain
  \begin{align*}
    |I_2|
     & \ioe \eta  \int_0^t\int_\Omega |\nabla U^\eps|^2 + C_\eta\int_0^t\int_\Omega |\nabla H^\eps|^2.
  \end{align*}

  But
  $\nabla H^\eps = \nabla(h(v^\eps) - h(v)) =  h'(v^\eps)\nabla v^\eps  - h'(v)\nabla v = h'(v^\eps)\nabla V^\eps  - (h'(v^\eps) - h'(v^\eps))\nabla v$. But $\nabla v$ is bounded, and $h'$ is bounded and Lipschitz, thus one concludes that
  \begin{align*}
    |I_2| & \ioe \eta  \int_0^t\int_\Omega |\nabla U^\eps|^2 + C_\eta \int_0^t\int_\Omega |\nabla V^\eps|^2 + |V^\eps|^2.
  \end{align*}

  We recall that by \cref{thm:improved_regu}, we have $\nabla u \in\fsL^\infty(\Omega_T)$, thus for the same reasons as previously, one obtains
  \begin{align*}
    |I_3|
     & \ioe \eta \int_0^t\int_\Omega |\nabla U^\eps|^2 + C_\eta\int_0^t\int_\Omega |V^\eps|^2.
  \end{align*}

  Then
  $I_4 = - \frac{d_B}{S}\int_0^t\int_\Omega h(v^\eps)|\nabla U^\eps|^2 \ioe 0$ has a good sign.

  Finally, using that $\nabla (h(v^\eps)) = h'(v^\eps)\nabla v^\eps$ is bounded in $\fsL^\infty(\Omega_T)$ thanks to \cref{thm:DGNM-bound}, one obtains that
  \begin{align*}
    |I_5|
     & \ioe \eta \int_0^t\int_\Omega |\nabla U^\eps|^2 + C_\eta\int_0^t\int_\Omega |U^\eps|^2.
  \end{align*}

  If we collect all our estimates and after choosing $\eta$ small enough, we have shown that for some constant $C_1 > 0$:

  \begin{align}\label{eq:boundU1}
    &\int_\Omega U^\eps(t)^2 + \int_0^t\int_\Omega |\nabla U^\eps|^2\\
    &\ioe C_1\left(\int_\Omega U^\eps(0)^2 +   \int_0^t \int_\Omega |U^\eps|^2 + \int_0^t\int_\Omega |\nabla Q^\eps|^2 + \int_0^t\int_\Omega |\nabla V^\eps|^2 + \int_0^t\int_\Omega |V^\eps|^2\right).\nonumber
  \end{align}

  We now multiply \eqref{eq:V} by $V^\eps$ to obtain similarly that for some constant $C > 0$,
  \begin{equation}\label{eq:boundV1}
    \int_0^t V^\eps(t)^2 + \int_0^t\int_\Omega |\nabla V^\eps|^2 \ioe C \left(\int_0^t V^\eps(0)^2 + \int_0^t\int_\Omega |U^\eps|^2 + |V^\eps|^2\right).
  \end{equation}

  We now sum \eqref{eq:boundU1} and \eqref{eq:boundV1} (multiplied by $2C_1$), to obtain
  \begin{align}\label{eq:estimateUV}
     & \int_0^t U^\eps(t)^2 + \int_0^t\int_\Omega |\nabla U^\eps|^2 + \int_0^t V^\eps(t)^2 + \int_0^t\int_\Omega |\nabla V^\eps|^2                                                \\
     & \ioe C \left(\int_0^t U^\eps(0)^2 + \int_0^t\int_\Omega |U^\eps|^2 + \int_0^t V^\eps(0)^2 + \int_0^t\int_\Omega |V^\eps|^2 + \int_0^t\int_\Omega |\nabla Q^\eps|^2\right).
  \end{align}

  Using that $U^\eps(0, \cdot) = V^\eps(0, \cdot) = 0$ and estimate \eqref{eq:estimateQ} of \cref{thm:SKTeps-regu}, we obtain
  \begin{align*}
    \int_\Omega U^\eps(t)^2 +  \int_\Omega V^\eps(t)^2
     & \ioe C \left(\int_0^t\int_\Omega |U^\eps|^2+ \int_0^t\int_\Omega |V^\eps|^2 + \eps\right).
  \end{align*}

  Gronwall's inequality, together with \eqref{eq:estimateUV} enables to conclude that for $0 < t <T$, we have
  \begin{equation*}
     \int_\Omega U^\eps(t)^2 + \int_0^t\int_\Omega |\nabla U^\eps|^2 +  \int_\Omega V^\eps(t)^2 + \int_0^t\int_\Omega |\nabla V^\eps|^2 \ioe \eps C.
  \end{equation*}
\end{proof}

\subsection{Part 2: convergence in a subdomain in \texorpdfstring{$\fsL^\infty\fsH^l$ for all $l$}{LinftyHl for all l}}
We now prove the
\begin{lemm}\label{thm:conv2}
  Let $T>0$, $d\ioe 3$, $\Omega$ be a smooth  bounded domain of $\R^d$,
  $d_1,d_2,\sigma >0$, $r_u,r_v\ge 0$, $d_{ij} > 0$ for $1\ioe i, j \ioe 2$ and $d_A, d_B, h, k$ satisfying \eqref{eq:hk-relation}. We assume that $h, k$ are affine on $[0, \sup_{\eps > 0} \|v^\eps\|_{\fsL^\infty(\Omega_T)}]$. Let $l>0$ and $u_{A, \init}, u_{B, \init}, v_\init \soe 0$ be smooth functions on $\overline{\Omega}$ which are compatible with the homogeneous Neumann boundary conditions.\medskip
  
  Let $u^\eps_A, u^\eps_B, v^\eps$ be the unique global strong solution (in $\Ccal^2(\R_+\times\overline\Omega)$) to \eqref{eq:SKT-eps} (with initial data $u_{A, \init}, u_{B, \init}, v_\init$), and $u, v$ be a solution to \eqref{eq:SKT} given by \cref{thm:conv} (with initial data $u_\init := u_{A, \init} + u_{B, \init}, v_\init$).\medskip

  Let $\tilde\Omega \Subset \Omega$. We recall the definition of $\eps_{\init, l}^2 := \sup_{k\ioe l+1} \int_\Omega |\nabla^{\otimes k} Q^\eps(0)|^2$. Then there exists a constant $C_l$ that depends only on the parameters of \eqref{eq:SKT-eps}, $d, T, l, \tilde\Omega$ and the norm of the initial data $u_{A, \init}, u_{B, \init}, v_\init$  in a Sobolev space $\fsH^k(\Omega)$ for $k$ large enough, but not on $\eps$, such that
  \begin{equation*}
    \|U^\eps\|_{\fsL^\infty([0, T], \fsH^l(\tilde{\Omega}))} + \|V^\eps\|_{\fsL^\infty([0, T], \fsH^l(\tilde{\Omega}))}
    \ioe C_l(\eps_{\init, l}\eps^\frac12 + \eps).
  \end{equation*}
\end{lemm}

We will prove this result by induction on $l$. However, for the initial case, we cannot use \cref{thm:conv1}, because we need an improved convergence rate $\eps_{\init, l}\eps^\frac12 + \eps$. That's why we start with an intermediate lemma, that study convergence in the space $\fsL^2(\Omega_T)$, with a rate $\eps_{\init, l}\eps^\frac12 + \eps$. As a by-product this lemma shows the last estimate in the whole space in \cref{thm:conv}.

\begin{lemm}\label{thm:convlemma}
  Let $T>0$, $d\ioe 4$, $\Omega$ be a smooth  bounded domain of $\R^d$,
  $d_1,d_2,\sigma >0$, $r_u,r_v\ge 0$, $d_{ij} > 0$ for $1\ioe i, j \ioe 2$ and $d_A, d_B, h, k$ satisfying \eqref{eq:hk-relation}. Let $u_{A, \init}, u_{B, \init}, v_\init \soe 0$ belong to $\Ccal^{2, \alpha}(\overline\Omega)$ for some $\a > 0$ and compatible with the homogeneous Neumann boundary conditions. \medskip
  
 Let $u^\eps_A, u^\eps_B, v^\eps$ be the unique global strong solution (in $\Ccal^2(\R_+\times\overline\Omega)$) to \eqref{eq:SKT-eps} (with initial data $u_{A, \init}, u_{B, \init}, v_\init$), and $u, v$ be a solution to \eqref{eq:SKT} given by \cref{thm:conv} (with initial data $u_\init := u_{A, \init} + u_{B, \init}, v_\init$). We recall the definition of $\eps^2_{\init, -1} := \int_\Omega |Q(0)|^2$. \medskip

  Then there exists a constant $C_0$ that depends only on the parameters of \eqref{eq:SKT-eps}, $d, T, \Omega$ and the norm of $u_{A, \init}, u_{B, \init}, v_\init$ in $\Ccal^{2, \alpha}(\overline\Omega)$, but not on $\eps$, such that
  \begin{equation*}
    \|U^\eps\|_{\fsL^2(\Omega_T)} + \|V^\eps\|_{\fsL^2(\Omega_T)} \ioe C_0(\eps_{\init, -1}\eps^\frac12 + \eps).
  \end{equation*}
\end{lemm}

The proof of this lemma follows the lines of \cite{morgansoresina2026}.

\begin{proof}
  We start by improving the results of \cref{thm:SKTeps-regu}, concerning $Q$. We compute
  \begin{align*}
  \partial_t Q^\eps
     & = \partial_t \tilde k^\eps u_B^\eps + \tilde k^\eps \partial_t u_B^\eps - \partial_t \tilde h^\eps u_A^\eps - \tilde h^\eps \partial_t u_A^\eps\\
    & = \partial_t \tilde k^\eps u_B^\eps +  (d_A + d_B) \tilde k^\eps \Delta  u_B^\eps +\tilde k^\eps(f_u^\eps u_B^\eps - \frac{1}{\eps}Q^\eps)\\
    &- \partial_t \tilde h^\eps u_A^\eps - d_A\tilde h^\eps \Delta u_A^\eps - \tilde h^\eps (f_u^\eps u_A^\eps + \frac{1}{\eps}Q^\eps)\\
     & =: - \frac{S}{\eps}Q^\eps + R^\eps,
  \end{align*}
  where $R^\eps := \partial_t \tilde k^\eps u_B^\eps +  (d_A + d_B) \tilde k^\eps \Delta  u_B^\eps +\tilde k^\eps f_u^\eps u_B^\eps - \partial_t \tilde h^\eps u_A^\eps - d_A\tilde h^\eps \Delta u_A^\eps - \tilde h^\eps f_u^\eps u_A^\eps $. We deduce thanks to \eqref{eq:initial-estimates-uLq}-\eqref{eq:initial-estimates-vinf} and \cref{thm:regu1} (to bound the terms with $\Delta u_C^\eps$) that $R^\eps$ is bounded in $\fsL^{2}(\Omega_T)$. \medskip

  We multiply this equation by $Q^\eps$ and integrate over $\Omega_t$, to deduce that
  \begin{align*}
    \frac12\int_\Omega |Q^\eps(t)|^2 + \frac{S}{\eps}\int_0^t\int_\Omega |Q^\eps|^2
    &=\frac12 \int_\Omega |Q^\eps(0)|^2 + \int_0^t\int_\Omega Q^\eps R^\eps\\
    &\ioe \frac12\eps^2_{\init, -1} + \frac{S}{2\eps} \int_0^t\int_\Omega |Q^\eps|^2 + \eps C \int_0^t\int_\Omega |R^\eps|^2.
  \end{align*}

  We deduce that
  \begin{align}\label{eq:boundQ}
    \frac{S}{\eps}\int_0^t\int_\Omega |Q^\eps|^2 \ioe  \eps^2_{\init, -1} + \eps C.
  \end{align}

  We can now bound $U^\eps, V^\eps$. Let $\eta > 0$ be a small parameter to be fixed later. We consider the operator $L := \Delta - I$, and we define the operator $L^{-1}$, that maps any $f\in\fsL^2(\Omega)$ to the unique solution to
  \begin{equation*}
    \begin{cases}
      L(L^{-1} f) = f\quad\text{in}\quad \Omega,\\
      \nabla (L^{-1} f)\cdot n = 0 \quad \text{on}\quad \partial\Omega.
    \end{cases}    
  \end{equation*}
By classical elliptic theory, this operator is self-adjoint and well-defined from $\fsL^2(\Omega)$ to $\fsH^{2}(\Omega)$. \medskip
  
  Let $Y^\eps := L^{-1} U^\eps$, $Z^\eps := L^{-1} V^\eps$, by \eqref{eq:U} and \eqref{eq:V}, we deduce
  \begin{equation}\label{eq:Y}
    \left\{
    \begin{aligned}
       & \partial_t Y^\eps = (I + L^{-1})(d_A U^\eps + d_BU_{B}^\eps) + L^{-1}F_{u, }^\eps,        \\
       & \nabla_x Y^\epsilon \cdot n = 0 \quad \text{on } [0, T]\times\partial\Omega,
    \end{aligned}
    \right.
  \end{equation}
    and
  \begin{equation}\label{eq:Z}
    \left\{
    \begin{aligned}
       & \partial_t Z^\eps = d_v(I + L^{-1}) V^\eps + L^{-1}F_{v}^\eps,\\
       & \nabla_x Z^\epsilon \cdot n = 0 \quad \text{on } [0, T]\times\partial\Omega.
    \end{aligned}
    \right.
  \end{equation}

  We can now multiply \eqref{eq:Y} by $- U^\eps = - L Y^\eps$ and integrate over $[0, t]$ to deduce that
  \begin{align*}
    &\frac12\int_\Omega |Y^\eps(t)|^2 +\frac12 \int_\Omega |\nabla Y^\eps(t)|^2
    + d_A\int_0^t \int_\Omega |U^\eps|^2\\
    &\ioe \frac12\int_\Omega |Y^\eps(0)|^2 + \frac12\int_\Omega |\nabla Y^\eps(0)|^2 - d_A \int_0^t \int_\Omega U^\eps Y^\eps - \int_0^t \int_\Omega F_u^\eps Y^\eps - d_B\int_0^t \int_\Omega (U_B^\eps U^\eps + U_{B}^\eps Y^\eps).
  \end{align*}

  First of all, we can bound thanks to Young's inequality
  \begin{equation*}
    - \int_0^t \int_\Omega U^\eps Y^\eps \ioe \eta \int_0^t \int_\Omega |U^\eps|^2 + C_\eta\int_0^t \int_\Omega |Y^\eps|^2.
  \end{equation*}

  We then recall that
  \begin{align*}
    F_u^\eps
    &= r_u U^\eps - d_{11}U^\eps (u + u^\eps) - d_{12}U^\eps v^\eps - d_{12}V^\eps u.
  \end{align*}

  Hence, using that $u, v^\eps$ are bounded in $\fsL^\infty(\Omega_T)$ (thanks to \cref{thm:improved_regu} and \cref{thm:DGNM-bound}), we obtain thanks to Young's inequality that
  \begin{align*}
    &- \int_0^t \int_\Omega F_u^\eps Y^\eps\\
    &\ioe \eta  \int_0^t \int_\Omega (|U^\eps|^2 + |V^\eps|^2) + C_\eta \int_0^t \int_\Omega |Y^\eps|^2 + d_{11}\int_0^t \int_\Omega u^\eps U^\eps Y^\eps.
  \end{align*}

  But by Hölder's inequality, we first have
  \begin{equation*}
    \int_0^t \int_\Omega u^\eps U^\eps Y^\eps
    \ioe \|u^\eps\|_{\fsL^\infty([0, T], \fsL^q(\Omega))}\|U^\eps\|_{\fsL^2([0, T], \fsL^2(\Omega))}\|Y^\eps\|_{\fsL^2([0, T], \fsL^{2^*}(\Omega))},
  \end{equation*}
  where $2^* := \frac{2d}{d-2}$ is the Sobolev exponent for $d=3$ (and any $p>2$ large enough in dimension $d=1, 2$) and $2 < q < + \infty$ is such that $\frac{1}{q} + \frac{1}{2} + \frac{1}{2^*} = 1$. By \cref{thm:SKTeps-regu}, $u^\eps$ is bounded in $\fsL^\infty([0, T], \fsL^q(\Omega))$, thus we obtain thanks to Young and Sobolev inequalities
  \begin{align*}
    \int_0^t \int_\Omega u^\eps U^\eps Y^\eps
    &\lesssim \eta\|U^\eps\|_{\fsL^2([0, T], \fsL^2(\Omega))}^2 + C_\eta\left(\| Y^\eps\|^2_{\fsL^2([0, T], \fsL^2(\Omega))} + \|\nabla Y^\eps\|^2_{\fsL^2([0, T], \fsL^2(\Omega))}\right).
  \end{align*}

  Hence, we have shown that
  \begin{align*}
    - \int_0^t \int_\Omega F_u^\eps Y^\eps
    &\ioe \eta  \int_0^t \int_\Omega (|U^\eps|^2 + |V^\eps|^2) + C_\eta \int_0^t \int_\Omega \left(| Y^\eps|^2 + |\nabla Y^\eps|^2\right).
  \end{align*}

  Then, we recall that
  \begin{equation*}
    U_B^\eps = \frac{Q^\eps + H^\eps u + h(v^\eps) U^\eps}{S},
  \end{equation*}
  thus, using the bound $u\in\fsL^\infty(\Omega_T)$ from \cref{thm:improved_regu}, we get
  \begin{align*}
    - \int_0^t \int_\Omega U_{B}^\eps U^\eps
    &= - \frac{1}{S}\int_0^t \int_\Omega (Q^\eps + H^\eps u + h(v^\eps) U^\eps) U^\eps\\
    &\ioe \eta \int_0^t \int_\Omega |U^\eps|^2 + C_\eta \int_0^t \int_\Omega \left(|Q^\eps|^2 + |H^\eps|^2\right)- \frac{1}{S}\int_0^t \int_\Omega h(v^\eps) |U^\eps|^2.
  \end{align*}

  We note that $- \int_0^t \int_\Omega h(v^\eps) |U^\eps|^2 \ioe 0$ and since $h$ is Lipschitz, we also have $\int_0^t \int_\Omega |H^\eps|^2\lesssim \int_0^t \int_\Omega |V^\eps|^2$, thus we obtain that

  \begin{align*}
    - \int_0^t \int_\Omega U_{B}^\eps U^\eps
    &\ioe \eta \int_0^t \int_\Omega |U^\eps|^2 + C_\eta \int_0^t \int_\Omega \left(|Q^\eps|^2 + |V^\eps|^2\right).
  \end{align*}

  We can similarly bound
  \begin{align*}
    - \int_0^t \int_\Omega U_{B}^\eps Y^\eps
    &\ioe \eta \int_0^t \int_\Omega |U^\eps|^2 + C \int_0^t \int_\Omega \left(|Q^\eps|^2 + |V^\eps|^2\right) + C_\eta \int_0^t \int_\Omega |Y^\eps|^2.
  \end{align*}

  Collecting all our estimates and using one last time Young inequality, we have shown that up to choosing $\eta$ small enough

  \begin{align}\label{eq:finalY}
    &\frac12\int_\Omega |Y^\eps(t)|^2 + \frac12\int_\Omega |\nabla Y^\eps(t)|^2
    + \frac{d_A}{2}\int_0^t \int_\Omega |U^\eps|^2\\
    &\ioe \frac12\int_\Omega |Y^\eps(0)|^2 + \frac12\int_\Omega |\nabla Y^\eps(0)|^2\nonumber\\
    &+ C_1\int_0^t \int_\Omega(|Y^\eps|^2 + |\nabla Y^\eps|^2 + |Q^\eps|^2 + |V^\eps|^2).\nonumber
  \end{align}

  We can now multiply \eqref{eq:Z} by $-V^\eps = - L Z^\eps$ and integrate over $[0, t]$ to obtain
  \begin{align*}
    &\frac12\int_\Omega |Z^\eps(t)|^2 + \frac12\int_\Omega  |\nabla Z^\eps(t)|^2
    + d_v\int_0^t \int_\Omega |V^\eps|^2\\
    &\ioe \frac12\int_\Omega |Z^\eps(0)|^2 + \frac12\int_\Omega  |\nabla Z^\eps(0)|^2 - d_v \int_0^t \int_\Omega V^\eps Z^\eps - \int_0^t \int_\Omega F_v^\eps Z^\eps.
  \end{align*}

  Thanks to similar computations as previously, one can bound
  \begin{align*}
    - \int_0^t \int_\Omega F_{v}^\eps Z^\eps
    &\ioe \eta \int_0^t \int_\Omega (|U^\eps|^2 + |V^\eps|^2) + C_\eta \int_0^t \int_\Omega \left(|Z^\eps|^2 + |\nabla Z^\eps|^2\right).
  \end{align*}

  Finally, we obtain
  \begin{align}\label{eq:finalZ}
    &\frac12\int_\Omega |Z^\eps(t)|^2 + \frac12\int_\Omega |\nabla Z^\eps(t)|^2
    + d_v\int_0^t \int_\Omega |V^\eps|^2\\
    &\ioe\frac12 \int_\Omega |Z^\eps(0)|^2 + \frac12\int_\Omega |\nabla Z^\eps(0)|^2
    + \eta \int_0^t \int_\Omega \left(|V^\eps|^2 + |U^\eps|^2\right)\nonumber\\
    &+ C_\eta\int_0^t \int_\Omega(|Z^\eps|^2 + |\nabla Z^\eps|^2).\nonumber
  \end{align}

  Then, combining \eqref{eq:finalY} and \eqref{eq:finalZ} (after choosing $\eta$ small enough in \eqref{eq:finalZ}), we get
  \begin{align*}
    &\int_\Omega \left(|Z^\eps(t)|^2 + |\nabla Z^\eps(t)|^2 + |Y^\eps(t)|^2 + |\nabla Y^\eps(t)|^2\right)
    + \int_0^t \int_\Omega \left(|V^\eps|^2 + |U^\eps|^2\right)\\
    &\lesssim \int_\Omega \left(|Z^\eps(0)|^2 + |\nabla Z^\eps(0)|^2 + |Y^\eps(0)|^2 + \nabla |Y^\eps(0)|^2\right)\\
    &+ \int_0^t \int_\Omega\left(|Z^\eps|^2 + |\nabla Z^\eps|^2 + |Y^\eps|^2 + |\nabla Y^\eps|^2\right)\\
    &+ \int_0^t \int_\Omega|Q^\eps|^2.
  \end{align*}

  Since $U^\eps(0, \cdot) = V^\eps(0, \cdot) = 0 $, we deduce that $Y^\eps(0, \cdot) = Z^\eps(0, \cdot) = 0$, and using \eqref{eq:boundQ}, we deduce that
    \begin{align*}
    &\int_\Omega \left(|Z^\eps(t)|^2 + |\nabla Z^\eps(t)|^2 + |Y^\eps(t)|^2 + |\nabla Y^\eps(t)|^2\right)
    + \int_0^t \int_\Omega \left(|V^\eps|^2 + |U^\eps|^2\right)\\
    &\lesssim  \int_0^t \int_\Omega\left(|Z^\eps|^2 + |\nabla Z^\eps|^2 + |Y^\eps|^2 + |\nabla Y^\eps|^2\right)\\
    &+ \eps_{\init, -1}^2\eps + \eps^2.
  \end{align*}

  Using Gronwall's inequality, we derive the sought conclusion.

\end{proof}

\begin{proof}[Proof of \cref{thm:conv2}]
  The proof contains two parts. We first improve (with respect to $\eps_{\init, l}$) the results of \cref{thm:SKTeps-regu} on $\int_0^t\int_\Omega |\chi\nabla^{\otimes l} Q|^2$. We then use this refined estimate similarly as in the proof of \cref{thm:conv1}.\\

  We first claim that the solutions of \eqref{eq:SKT} constructed in \cref{thm:globalSKT} enjoy enough space smoothness. It is possible to provide an independent proof without using a compactness argument, but for our purpose, using the uniform boundedness of $u_C^\eps$ and $v^\eps$ in $\fsL^\infty([0, T], \fsH^{k}(\tilde{\Omega}))$ for any $\tilde\Omega\Subset\Omega$ and all $k\in\N$ proved in \cref{thm:SKTeps-regu} and the convergence result proved in \cref{thm:conv1}, we can immediately deduce that $u,v$ are bounded in  $\fsL^\infty([0, T], \fsH^{k}(\tilde{\Omega}))$ for all $k\in\N$.\medskip

  Let $0 < t < T$ and let $\tilde{\Omega} \Subset \hat{\Omega} \Subset \Omega$. Let $\chi$ be a cutoff function between $\tilde{\Omega}$ and $\hat\Omega$.

  We recall that
  \begin{align}\label{eq:dtQ}
    \partial_t Q^\eps
     & = \partial_t \tilde k^\eps u_B^\eps + \tilde k^\eps \partial_t u_B^\eps - \partial_t \tilde h^\eps u_A^\eps - \tilde h^\eps \partial_t u_A^\eps\\
     & = \partial_t v^\eps k'(v^\eps) u_B^\eps + \tilde k^\eps ((d_A + d_B)\Delta u_B^\eps + f_u^\eps u_B^\eps- \frac{1}{\eps}Q^\eps)\nonumber\\
     &- \partial_t v^\eps h'(v^\eps) u_A^\eps - \tilde h^\eps (d_A\Delta u_A^\eps + f_u^\eps u_A^\eps + \frac{1}{\eps}Q^\eps)\nonumber\\
     & =: - \frac{S}{\eps}Q^\eps + R^\eps.\nonumber
  \end{align}
  By \cref{thm:SKTeps-regu}, $R^\eps$ is bounded (uniformly in $\eps$) in $\fsL^{2}([0, T], \fsH^{k}(\hat\Omega))$ for any $k\in\N$ (for the terms with $\partial_t v^\eps$ we use the equation \eqref{eq:SKT-eps}).

  Let $\beta \in\N^d$, we can compute
  \begin{align*}
    \frac{\dd}{\dd t}\int_\Omega \chi^2 |\partial^\beta Q^\eps|^2
     & = 2\int_\Omega \chi^2 \partial^\beta Q^\eps \partial^\beta(- \frac{S}{\eps}Q^\eps + R^\eps).
  \end{align*}

  Hence, by integrating the previous estimate (and using Young's inequality)
  \begin{align*}
    &\int_\Omega \chi^2 |\partial^\beta Q^\eps|^2(t) + 2\int_0^t\int_\Omega \chi^2 \frac S\eps|\partial^\beta Q^\eps|^2\\
     & \ioe \int_\Omega \chi^2 |\partial^\beta Q^\eps|^2(0) + 2\int_0^t\int_\Omega \chi^2 \partial^\beta Q^\eps\partial^\beta R^\eps\\
     & \ioe \int_\Omega \chi^2 |\partial^\beta Q^\eps|^2(0) +\frac{S}{2\eps}\int_0^t\int_\Omega |\chi^2 \partial^\beta Q^\eps|^2 + C\eps \int_0^t\int_\Omega |\chi\partial^\beta R^\eps|^2.
  \end{align*}

  Since $\int_0^t\int_\Omega |\chi\partial^\beta R^\eps|^2\ioe C$, we deduce that
  \begin{equation*}
    \frac{1}{\eps}\int_0^t\int_\Omega \chi^2 |\partial^\beta Q^\eps|^2
    \ioe C(\eps_{\init, |\beta| - 1}^2 + \eps).
  \end{equation*}

  In particular, we deduce that for all $\tilde\Omega\Subset\Omega$
  \begin{equation}\label{eq:Qconv}
    \int_0^t\int_{\tilde\Omega}|\partial^\beta Q^\eps|^2
    \ioe C(\eps_{\init, |\beta| - 1}^2\eps + \eps^2).
  \end{equation}

  We can now mimic the proof of \cref{thm:conv1}. More precisely, we will prove by induction that for any $l\soe 0$ and $\Omega_1\Subset \Omega$, there exists some constant $C$ that depends only on $T, d, l, \Omega_1$, the parameters of \eqref{eq:SKT-eps} and the norm of the initial data $u_{A, \init}, u_{B, \init}, v_\init$ in some $\fsH^k(\Omega)$ for $k$ large enough such that
  \begin{equation}\label{eq:HRconv1}
    \int_0^T\int_\Omega |\nabla^{\otimes (l+1)} U^\eps|^2 + \int_0^T\int_\Omega |\nabla^{\otimes (l+1)} V^\eps|^2 \ioe (\eps_{\init, l}^2\eps + \eps^2) C.
  \end{equation}

  In the proof of this result, we will also deduce that for $0 < t <T$
  \begin{equation}\label{eq:HRconv1infinity}
    \int_{\Omega_1} |\nabla^{\otimes l} U^\eps(t)|^2  + \int_{\Omega_1} |\nabla^{\otimes l}V^\eps(t)|^2 \ioe (\eps_{\init, l}^2\eps + \eps^2) C.
  \end{equation}

  The initial case $l=-1$ is provided by \cref{thm:convlemma}. For $l\soe 0$, we assume that the induction hypothesis holds for some $l - 1\soe 0$ and we prove it for $l$, hence we assume that
  \begin{equation}\label{eq:HRconv}
    \int_0^T\int_{\Omega_1} |\nabla^{\otimes l} U^\eps|^2 + \int_0^T\int_{\Omega_1} |\nabla^{\otimes l} V^\eps|^2 \ioe (\eps_{\init, l}^2\eps + \eps^2) C.
  \end{equation}

  Since we assumed that $h, k$ are affine on $[0, \sup_{\eps > 0} \|v^\eps\|_{\fsL^\infty(\Omega_T)}]$, we won't make any difference between $H^\eps, K^\eps$ and $V^\eps$ in the estimates. \medskip

  Using \eqref{eq:U} and \eqref{eq:V}, we deduce that
  \begin{equation}\label{eq:a}
    \left\{
    \begin{aligned}
       & \partial_t \partial^\a U^\eps = \partial^\a\Delta(d_A U^\eps + d_BU_B^\eps) + \partial^\a F_u^\eps, \\
       & \partial_t  \partial^\a V^\eps =  d_v \partial^\a \Delta V^\eps +  \partial^\a F_v^\eps.
    \end{aligned}
    \right.
  \end{equation}

  Let $0 <t <T$, $\Omega_1\Subset\Omega_2\Subset\Omega$ and let $\chi$ be a smooth cutoff function between $\Omega_1$ and $\Omega_2$. Let $\a$ be such that $|\a| = l$, we multiply \eqref{eq:a} by $\chi^2 \partial^\a U$ and use the fact that
  \begin{align*}
    - \int_\Omega \chi^2 \partial^\a U^\eps \partial^\a \Delta U^\eps
     & = \int_\Omega \left(\nabla(\chi\partial^\a U^\eps) + (\nabla\chi) \partial^\a U^\eps\right)\left(\nabla (\chi \partial^\a U^\eps) - (\nabla\chi)\partial^\a U^\eps\right) \\
     & = \int_\Omega |\nabla(\chi\partial^\a U^\eps)|^2\\
    &- \int_\Omega \nabla(\chi \partial^\a U^\eps)(\nabla\chi)\partial^\a U^\eps + \int_\Omega (\nabla\chi)\partial^\a U^\eps\left(\nabla (\chi \partial^\a U^\eps) - (\nabla\chi)\partial^\a U^\eps\right),
  \end{align*}

  to obtain
  \begin{align*}
     &\frac12 \int_\Omega  |\chi\partial^\a U^\eps|^2(t) + d_A \int_0^T\int_\Omega |\nabla(\chi\partial^\a U^\eps)|^2 \\
     &\ioe\frac12\int_\Omega|\chi\partial^\a U^\eps|^2(0) + d_A\int_0^T\int_\Omega \nabla(\chi \partial^\a U^\eps)(\nabla\chi)\partial^\a U^\eps\\
     &- d_A\int_0^T\int_\Omega (\nabla\chi)\partial^\a U^\eps\left(\nabla (\chi \partial^\a U^\eps) - (\nabla\chi)\partial^\a U^\eps\right) \\
     & + d_B \int_0^T\int_\Omega \chi^2\partial^\a U^\eps \Delta\partial^\a U_B^\eps + \int_0^T \int_\Omega \partial^\a F_u^\eps \chi^2 \partial^\a U^\eps \\
     & =: \frac12\int_\Omega|\chi\partial^\a U^\eps|^2(0) + I_1 + I_2 + I_3 + I_4.
  \end{align*}

  We will now bound all the terms appearing in this sum. In the rest of the proof, $\eta$ is a small parameter to be chosen later.
  We start with $I_1$:
  \begin{align*}
     I_1
     & = d_A\int_0^T\int_\Omega |\nabla(\chi \partial^\a U^\eps)(\nabla\chi)\partial^\a U^\eps|\\
     & \ioe \eta \int_0^T\int_\Omega |\nabla(\chi \partial^\a U^\eps)|^2 + C_\eta \int_0^T\int_\Omega |\nabla\chi|^2|\partial^\a U^\eps|^2 \\
     & \ioe \eta\int_0^T\int_\Omega |\nabla(\chi \partial^\a U^\eps)|^2 + C_\eta(\eps_{\init, l}^2\eps + \eps^2),
  \end{align*}
  the last line comes from the induction hypothesis (applied on $\Omega_2$).

  We use the same kind of computations to bound $I_2$
  \begin{align*}
    I_2
     & \ioe \eta\int_0^T\int_\Omega |\nabla(\chi \partial^\a U^\eps)|^2 + C_\eta(\eps_{\init, l}^2\eps + \eps^2).
  \end{align*}

  We now consider $I_3 = d_B \int_0^T\int_\Omega \chi^2\partial^\a U^\eps \Delta\partial^\a U_B^\eps$. As previously we start by rewriting
  \begin{align*}
    \int_0^t \int_\Omega \chi^2\partial^\a U^\eps \Delta\partial^\a U_B^\eps
     & = - \int_0^t\int_\Omega \nabla(\chi\partial^\a U^\eps)\nabla(\chi\partial^\a U^\eps_B)
    -\int_0^t\int_\Omega (\nabla\chi)\partial^\a U^\eps\nabla(\chi\partial^\a U^\eps_B)       \\
     & +\int_0^t\int_\Omega \nabla(\chi\partial^\a U^\eps)(\nabla\chi)\partial^\a U^\eps_B
    +\int_0^t\int_\Omega (\nabla\chi)\partial^\a U^\eps(\nabla\chi)\partial^\a U^\eps_B       \\
     & =: I_{31} + I_{32} + I_{33} + I_{34}.
  \end{align*}

  Hence, by the induction hypothesis \eqref{eq:HRconv} (on $\Omega_2$), we first obtain
  \begin{align*}
    I_{32}
    &\ioe \eta \|\nabla(\chi\partial^\a U^\eps_B)\|_{\fsL^2(\Omega\times[0, T])}^2 + C_\eta \|(\nabla\chi)\partial^\a U^\eps\|_{\fsL^2(\Omega\times[0, T])}^2\\
    &\ioe \eta \|\nabla(\chi\partial^\a U^\eps_B)\|_{\fsL^2(\Omega\times[0, T])}^2 + C_\eta (\eps_{\init, l}^2\eps + \eps^2).
  \end{align*}
  Similarly,
  \begin{align*}
    I_{33} \ioe \eta \|\nabla(\chi\partial^\a U^\eps)\|_{\fsL^2(\Omega_2\times[0, T])}^2 +C_\eta \|\partial^\a U^\eps_B\|_{\fsL^2(\Omega_2\times[0, T])}^2.
  \end{align*}
  We also obtain
  \begin{align*}
    I_{34} \lesssim  \|\partial^\a U^\eps_B\|_{\fsL^2(\Omega_2\times[0, T])}^2 +\eps_{\init, l}^2\eps + \eps^2.
  \end{align*}

  We recall that
  \begin{equation*}
    U_B^\eps = \frac{Q^\eps + H^\eps u + h(v^\eps) U^\eps}{S}.
  \end{equation*}
  We notice that
  \begin{equation}\label{eq:Ku}
    \partial^\alpha(H^\eps u)= \sum_{\beta + \gamma = \alpha}\binom{\a}{\beta}\partial^\beta H^\eps \partial^\gamma u.
  \end{equation}
  Since $\partial^\gamma u \in\fsL^\infty([0, T], \fsH^2(\Omega_2))$, we deduce by a Sobolev injection that $\partial^\gamma u\in\fsL^\infty(\Omega_2\times[0, T])$. Hence,
  $$\|\partial^\beta H^\eps \partial^\gamma u \|_{\fsL^2(\Omega_2\times[0, T])}\ioe C\|\partial^\beta H^\eps \|_{\fsL^2(\Omega_2\times[0, T])}\ioe C(\eps_{\init, l}\eps^{\frac 12} + \eps),$$
  by the induction hypothesis. \medskip

  Using the space-smoothness of $v^\eps$, one can similarly show that
  $$\|\partial^\a (U^\eps \tilde{h}^\eps) \|_{\fsL^2(\Omega_2\times[0, T])} \ioe C\sum_{\beta \ioe \a}\|\partial^\beta U^\eps \|_{\fsL^2(\Omega_2\times[0, T])}\ioe C(\eps_{\init, l}\eps^{\frac 12} + \eps).$$
  Thus, we can use \eqref{eq:Qconv} to deduce that
  \begin{equation}\label{eq:UBl}
    \|\partial^\a U^\eps_B\|_{\fsL^2(\Omega_2\times[0, T])}^2 \ioe C(\eps_{\init, l}^2\eps + \eps^2).
  \end{equation}
  
  We can now bound $\nabla(\chi\partial^\a U^\eps_B)$. By \eqref{eq:Qconv}, we have
  $$\|\nabla(\chi\partial^\a Q^\eps)\|_{\fsL^2(\Omega_2\times[0, T])}^2 \ioe C(\eps_{\init, l}^2\eps + \eps^2).$$ Using the previous computations we can similarly bound
  \begin{equation*}
    \|\nabla(\chi\partial^\a (H^\eps u ))\|_{\fsL^2(\Omega_2\times[0, T])}^2 \ioe C(\eps_{\init, l}^2\eps + \eps^2 + \|\chi H^\eps \nabla\partial^\a u\|_{\fsL^2(\Omega_2\times[0, T])}^2 + \|u\nabla(\chi\partial^\a H^\eps)\|_{\fsL^2(\Omega_2\times[0, T])}^2).
  \end{equation*}
  
  As previously, a Sobolev injection yields that $\nabla\partial^\a u \in \fsL^\infty(\Omega_2\times[0, T])$, thus we deduce that
  \begin{align}\label{eq:boundHu}
    \|\nabla(\chi\partial^\a (H^\eps u ))\|_{\fsL^2(\Omega_2\times[0, T])}^2
    &\ioe C(\eps_{\init, l}^2\eps + \eps^2 +  \|\nabla(\chi\partial^\a H^\eps)\|_{\fsL^2(\Omega_2\times[0, T])}^2 \|u\|_{\fsL^\infty(\Omega_2\times[0, T])}^2)\nonumber\\
    &\ioe C(\eps_{\init, l}^2\eps + \eps^2 +  \|\nabla(\chi\partial^\a V^\eps)\|_{\fsL^2(\Omega_2\times[0, T])}^2 ).
  \end{align}
    
  Similarly, we have by a Sobolev injection that $\nabla^{\otimes k} \tilde h^\eps \in\fsL^\infty(\Omega_2\times[0, T])$ for any $k>0$, thus by the recursion hypothesis and Leibnitz formula, we obtain
  $$ \|\nabla(\chi\partial^\a (U^\eps \tilde{h}^\eps ))\|_{\fsL^2(\Omega_2\times[0, T])}^2 \ioe C(\eps_{\init, l}^2\eps + \eps^2 +  \|\nabla(\chi\partial^\a U^\eps)\|_{\fsL^2(\Omega_2\times[0, T])}^2).$$

  Finally, we have proven that
  \begin{equation}\label{eq:UBlp1}
    \|\nabla(\chi\partial^\a U^\eps_B)\|_{\fsL^2(\Omega_2\times[0, T])}^2 \lesssim (\eps_{\init, l}^2\eps + \eps^2) + \|\nabla(\chi\partial^\a U^\eps)\|_{\fsL^2(\Omega_2\times[0, T])}^2 + \|\nabla(\chi\partial^\a V^\eps)\|_{\fsL^2(\Omega_2\times[0, T])}^2.
  \end{equation}

  We can now deal with the last term
  \begin{align*}
    SI_{31}
     & = - \int_0^t\int_\Omega \nabla(\chi\partial^\a U^\eps)\nabla(\chi\partial^\a (Q^\eps + H^\eps u + h(v^\eps) U^\eps)) \\
     & = - \int_0^t\int_\Omega h(v^\eps)|\nabla(\chi\partial^\a U^\eps)|^2 - \int_0^t\int_\Omega \nabla(\chi\partial^\a U^\eps)\nabla(\chi\partial^\a (Q^\eps + H^\eps u))\\
     & - \sum_{\beta + \gamma =\alpha \text{ and } \beta > 0}\binom{\a}{\beta}\int_0^t\int_\Omega \nabla(\chi\partial^\a U^\eps)\nabla(\chi\partial^\beta \tilde{h}^\eps\partial^\gamma U^\eps)- \int_0^t\int_\Omega \nabla(\chi\partial^\a U^\eps)(\nabla\tilde{h}^\eps)\chi\partial^\a U^\eps.
  \end{align*}

  We first notice that $- \int_0^t\int_\Omega h(v^\eps)|\nabla(\chi\partial^\a U^\eps)|^2\ioe 0$, then one has thanks to \eqref{eq:Qconv} and \eqref{eq:boundHu}
  \begin{align*}
    - \int_0^t\int_\Omega \nabla(\chi\partial^\a U^\eps)\nabla(\chi\partial^\a (Q^\eps + H^\eps u))
    &\ioe \eta\|\nabla(\chi\partial^\a U^\eps)\|_{\fsL^2(\Omega_2\times[0, T])}^2 \\
     &  + C_\eta\left(\eps_{\init, l}^2\eps + \eps^2 + \|\nabla(\chi\partial^\a V^\eps)\|_{\fsL^2(\Omega_2\times[0, T])}^2\right).
  \end{align*}

  We notice that $\tilde{h}^\eps \in\fsL^\infty([0, T], \fsH^{k+2}(\Omega_2))\hookrightarrow\fsL^\infty([0, T], \fsW^{k, \infty}(\Omega_2))$ for all $k\soe 0$ in dimension $d\ioe 3$. Hence, thanks to the induction hypothesis \eqref{eq:HRconv}, we deduce that for any $\beta, \gamma$ such that $\beta + \gamma = \a$ and $\beta >0$ we have
  \begin{align*}
    -\int_0^t\int_\Omega \nabla(\chi\partial^\a U^\eps)\nabla(\chi\partial^\beta \tilde{h}^\eps\partial^\gamma U^\eps) \ioe \eta\|\nabla(\chi\partial^\a U^\eps)\|_{\fsL^2(\Omega_2\times[0, T])}^2 + C_\eta\left(\eps_{\init, l}^2\eps + \eps^2\right).
  \end{align*}

  Similarly, we obtain
  \begin{align*}
    -\int_0^t\int_\Omega \nabla(\chi\partial^\a U^\eps)(\nabla\tilde{h}^\eps)\chi\partial^\a U^\eps\ioe \eta\|\nabla(\chi\partial^\a U^\eps)\|_{\fsL^2(\Omega_2\times[0, T])}^2 + C_\eta\left(\eps_{\init, l}^2\eps + \eps^2\right).
  \end{align*}

  Finally, using \eqref{eq:UBl} and \eqref{eq:UBlp1} in the bounds obtained for $I_{31}, I_{32}, I_{33}, I_{34}$, we have shown that
  \begin{equation*}
    I_3\ioe \eta\|\nabla(\chi\partial^\a U^\eps)\|_{\fsL^2(\Omega_2\times[0, T])}^2 +  C_\eta(\eps_{\init, l}^2\eps + \eps^2 + \|\nabla(\chi\partial^\a V^\eps)\|_{\fsL^2(\Omega_2\times[0, T])}^2).
  \end{equation*}

  For $I_4$, we first notice that
  \begin{equation*}
    F^\eps_u = r_u U^\eps- d_{11} U^\eps (u + u^\eps) - d_{12}uV^\eps -d_{12} U^\eps v^\eps.
  \end{equation*}

  We use Leibnitz formula to get $\partial^\a (U^\eps (u + u^\eps)) = \sum_{\gamma + \beta = \a}\binom{\a}{\beta}\partial^\gamma U^\eps\partial^\beta(u + u^\eps)$. By the result of \cref{thm:SKTeps-regu}, we have $u, u^\eps, v^\eps \in\fsL^\infty([0, T], \fsH^{l+2}(\Omega_2))$ and thus by a Sobolev injection (in dimension $d\ioe 3$), we have a uniform bound $\partial^\beta u, \partial^\beta u^\eps, \partial^\beta v^\eps \in\fsL^\infty(\Omega_2\times[0, T])$.

  We can similarly bound the other terms to deduce that
  \begin{align}\label{eq:boundF}
    \|\chi\partial^\a F^\eps_u\|_{\fsL^2(\Omega_2\times[0, T])}^2
     & \lesssim \sum_{\beta \ioe \alpha }\left(\|\chi\partial^\beta U^\eps\|_{\fsL^2(\Omega_2\times[0, T])}^2 + \|\chi\partial^\beta V^\eps\|_{\fsL^2(\Omega_2\times[0, T])}^2\right) \nonumber\\
     & \lesssim C(\eps_{\init, l}^2\eps + \eps^2).
  \end{align}
  Using again the induction hypothesis, we deduce that
  \begin{equation*}
    I_4 \ioe C(\eps_{\init, l}^2\eps + \eps^2).
  \end{equation*}

  If we gather all our estimates and choose $\eta$ small enough, we see that
  \begin{align}\label{eq:finalUa}
    &\int_\Omega  |\chi\partial^\a U^\eps|^2(t) + \int_0^T\int_\Omega |\nabla(\chi\partial^\a U^\eps)|^2\\
    &\lesssim \int_\Omega|\chi\partial^\a U^\eps|^2(0) + \eps_{\init, l}^2\eps + \eps^2 + \|\nabla(\chi\partial^\a V^\eps)\|_{\fsL^2(\Omega_2\times[0, T])}^2.\nonumber
  \end{align}

  We now multiply the second part of \eqref{eq:a} by $\chi^2 \partial^\a V^\eps$ to obtain that
  \begin{equation*}
    \int_\Omega |\chi\partial^\a V^\eps|(t)^2 = \int_\Omega |\chi\partial^\a V^\eps|(0)^2 + d_v \int_0^t\int_\Omega \chi^2 \partial^\a V^\eps \partial^\a \Delta V^\eps +  \int_0^t\int_\Omega\chi^2\partial^\a V^\eps\partial^\a F_v^\eps.
  \end{equation*}

  Similar computations as in \eqref{eq:boundF} show that $\|\chi\partial^\a F^\eps_v\|_{\fsL^2(\Omega_2\times[0, T])}^2 \lesssim C(\eps_{\init, l}^2\eps + \eps^2)$ and thus we get
  \begin{equation*}
    \int_0^T\int_\Omega\chi^2\partial^\a V^\eps\partial^\a F_v^\eps \ioe C(\eps_{\init, l}^2\eps + \eps^2).
  \end{equation*}

  Then, one notices again that
  \begin{align*}
    \int_\Omega \chi^2 \partial^\a V^\eps \partial^\a \Delta V^\eps
    & = - \int_\Omega |\nabla(\chi\partial^\a V^\eps)|^2\\
    & +\int_\Omega \nabla(\chi \partial^\a V^\eps)(\nabla\chi)\partial^\a V^\eps - \int_\Omega (\nabla\chi)\partial^\a V^\eps\left(\nabla (\chi \partial^\a V^\eps) - (\nabla\chi)\partial^\a V^\eps\right).
  \end{align*}

  Following similar computations as previously, one obtains
  \begin{align*}
    &\int_0^T\int_\Omega \nabla(\chi \partial^\a V^\eps)(\nabla\chi)\partial^\a V^\eps - \int_0^T\int_\Omega (\nabla\chi)\partial^\a V^\eps\left(\nabla (\chi \partial^\a V^\eps) + (\nabla\chi)\partial^\a V^\eps\right)\\
    &\ioe \eta\|\nabla(\chi \partial^\a V^\eps)\|^2_{\fsL^2(\Omega_t)} + C_\eta (\eps_\init^2\eps + \eps^2).
  \end{align*}

  Hence by choosing $\eta$ small enough, we obtain for some constant $C>0$ that
  \begin{equation}\label{eq:finalV}
    \int_\Omega |\chi\partial^\a V^\eps(t)|^2 + \int_0^t\int_\Omega |\nabla(\chi \partial^\a V^\eps)|^2
    \ioe C\left(\int_\Omega |\chi\partial^\a V^\eps|(0)^2 + \eps_{\init, l}^2\eps + \eps^2\right).
  \end{equation}

  As previously, using \eqref{eq:finalUa} and \eqref{eq:finalV} and recalling that $V^\eps(0, \cdot) = 0$ and $U^\eps(0, \cdot) = 0$, we obtain \eqref{eq:HRconv1} and \eqref{eq:HRconv1infinity}, which enables to finish the induction and concludes the proof of \cref{thm:conv2}.
\end{proof}

\section{Initial layer}\label{sec:initial}
In this section, we give a proof of \cref{thm:initiallayer}.

\begin{proof}[Proof of \cref{thm:initiallayer}]
  We first start with the initial layer in the whole domain $\Omega$. We recall the computations in the beginning of the proof of \cref{thm:convlemma}
\begin{align*}
    \partial_t Q^\eps =: - \frac{S}{\eps}Q^\eps + R^\eps,
  \end{align*}
  where $R^\eps := \partial_t \tilde k^\eps u_B^\eps +  (d_A + d_B) \tilde k^\eps \Delta  u_B^\eps +\tilde k^\eps f_u^\eps u_B^\eps - \partial_t \tilde h^\eps u_A^\eps - d_A\tilde h^\eps \Delta u_A^\eps - \tilde h^\eps f_u^\eps u_A^\eps $ is bounded in $\fsL^2([0, T], \fsL^{2}(\Omega))$ (cf. \eqref{eq:initial-estimates-uLq}-\eqref{eq:initial-estimates-vinf} and \cref{thm:regu1}).\medskip

  We now define for $0 < t < T$, $\tilde{Q}^\eps(t, \cdot) := Q^\eps(t, \cdot) - e^{-\frac{St}{\eps}}Q^\eps(0, \cdot)$. An immediate computation shows that
  \begin{equation}\label{eq:boundtQ}
    \partial_t \tilde{Q}^\eps + \frac{S}{\eps}\tilde{Q}^\eps = R^\eps,
  \end{equation}
  with $\tilde{Q}^\eps(0, \cdot) = 0$.

  Let $0< t < T$, we multiply the previous expression by $\tilde{Q}^\eps$ and integrate over $\Omega\times[0, t]$ to deduce (after the use of Young's inequality)
  \begin{align*}
    \frac12\int_\Omega |\tilde{Q}^\eps(t)|^2 + \frac{S}{\eps}\int_0^t\int_\Omega |\tilde{Q}^\eps|^2
    &= \int_0^t\int_\Omega \tilde{Q}^\eps R^\eps\\
    &\ioe \frac{S}{2\eps}\int_0^t\int_\Omega |\tilde{Q}^\eps|^2 + C\eps \int_0^t\int_\Omega |R^\eps|^2 \\
    &\ioe \frac{S}{2\eps}\int_0^t\int_\Omega |\tilde{Q}^\eps|^2 + C\eps.
  \end{align*}

  Hence, we obtain
  \begin{align*}
    \int_\Omega |\tilde{Q}^\eps(t)|^2 \ioe C\eps.
  \end{align*}

  We then have:
  \begin{equation*}
    \|Q^\eps(t)\|^2_{\fsL^2(\Omega)} \ioe 2\|\tilde{Q}^\eps(t)\|^2_{\fsL^2(\Omega)} + 2\|e^{-\frac{St}{\eps}}Q_\init\|^2_{\fsL^2(\Omega)} \ioe C(\eps + e^{-\frac{2St}{\eps}}).
  \end{equation*}
  With $t_\eps := \frac{\eps|\ln\eps|}{2S}$ we obtain the desired estimate
  \begin{equation*}
    \|Q^\eps(t_\eps)\|^2_{\fsL^2(\Omega)} \ioe C\eps.
  \end{equation*}

  We now consider the initial layer in a subdomain $\tilde\Omega\Subset\Omega$, and we make the assumptions of the second part of \cref{thm:initiallayer}. Let $\tilde\Omega\Subset\hat\Omega\Subset\Omega$ and let $\chi$ be a smooth cutoff function between $\tilde\Omega$ and $\hat\Omega$. Thanks to \cref{thm:SKTeps-regu} (applied with $l=2$), we deduce that $u_C^\eps \in\fsL^\infty([0, T], \fsH^{2}(\hat\Omega))$. Thus, we have $(d_A + d_B) \tilde k^\eps \Delta  u_B^\eps +\tilde k^\eps f_u^\eps u_B^\eps- d_A\tilde h^\eps \Delta u_A^\eps - \tilde h^\eps f_u^\eps u_A^\eps \in\fsL^\infty([0, T], \fsL^{2}(\hat\Omega))$.\medskip

  We now perform similar computations as those in \eqref{eq:boundvl}. Thanks to \eqref{eq:SKT}, we deduce that
  \begin{equation*}
    \partial_t(\Delta v^\eps) - d_v\Delta (\Delta v^\eps) = \Delta (f_v(u^\eps, v^\eps)v^\eps).
  \end{equation*}

  Hence, we obtain
  \begin{align*}
    \partial_t(\chi\Delta v^\eps) - d_v\Delta (\chi\Delta v^\eps)
    &= \chi\Delta (f_v(u^\eps, v^\eps)v^\eps) - d_v \Delta v^\eps\Delta \chi - 2d_v\nabla\Delta v^\eps\cdot \nabla \chi\\
    &=: T^\eps.
  \end{align*}

  Using again \cref{thm:SKTeps-regu}, we deduce that $T^\eps \in\fsL^2(\Omega_T)$. Indeed, concerning the term $\Delta(u^\eps v^\eps) = u^\eps\Delta v^\eps + v^\eps\Delta u^\eps + 2 \nabla u^\eps\cdot\nabla v^\eps$, by \cref{thm:SKTeps-regu} (we use the first part of the Proposition), we have $\Delta u^\eps, \nabla u^\eps \in \fsL^{2}(\Omega_T)$ and by \eqref{eq:initial-estimates-vinf}, we have $v^\eps, \nabla v^\eps \in\fsL^\infty(\Omega_T)$. Then, one has $u^\eps, \Delta v^\eps \in \fsL^{q}(\Omega_T)$ for all $q<\infty$ by \eqref{eq:initial-estimates-uLq} and \eqref{eq:initial-estimates-nablau}, which suffices to conclude that $\Delta(u^\eps v^\eps) \in\fsL^2(\Omega_T)$. The other terms in $\Delta (f_v(u^\eps, v^\eps)v^\eps)$ can be treated in the same way. Similarly, by \cref{thm:SKTeps-regu} (we apply it with $l=2$), we have $\Delta v^\eps, \nabla\Delta v^\eps\in\fsL^{\frac{10}{3}}(\hat\Omega\times[0, T])$. We now multiply the previous equation by $\chi\Delta v^\eps$ and integrate over $\Omega\times[0, t]$ to deduce that (after standard manipulations)
  \begin{equation*}
    \frac{1}{2}\int_\Omega |\chi\Delta v^\eps|^2 + \frac{d_v}{2}\int_0^t\int_\Omega |\nabla(\chi\Delta v^\eps)|^2 \lesssim \int_0^t\int_\Omega (T^\eps)^2.
  \end{equation*}
  
  Thus, we deduce that $\chi \Delta v^\eps \in\fsL^\infty([0, T], \fsL^{2}(\Omega))$. We can now consider $\partial_t \tilde k^\eps u_B^\eps = u_B^\eps k'(v^\eps) (d_v \Delta v^\eps + f_v v^\eps)$. Since $u_B^\eps\in\fsL^\infty([0, T], \fsH^{2}(\hat\Omega))$, a Sobolev injection in dimension $d\ioe 3$, yields that $u_B^\eps \in\fsL^\infty(\hat\Omega\times[0, T])$, thus $\partial_t \tilde k^\eps u_B^\eps \in\fsL^\infty([0, T], \fsL^{2}(\hat\Omega))$. Thus we deduce that $\chi R^\eps$ is bounded in $\fsL^\infty([0, T], \fsL^{2}(\Omega))$. \medskip
  
  We multiply again \eqref{eq:boundtQ} by $\chi^2\tilde{Q}^\eps$ and integrate over $\Omega\times[0, t]$. Using that $\tilde{Q}^\eps(0, \cdot) = 0$, we obtain 
  \begin{align*}
    \frac12\int_\Omega |\chi \tilde{Q}^\eps(t)|^2 + \frac{S}{\eps}\int_0^t\int_\Omega |\chi \tilde{Q}^\eps|^2
    &\ioe \int_0^t\int_\Omega \chi^2 \tilde{Q}^\eps R^\eps\\
    &\ioe \frac{S}{2\eps}\int_0^t\int_\Omega |\chi \tilde{Q}^\eps|^2 + C\eps \int_0^t\int_\Omega |\chi R^\eps|^2 \\
    &\ioe \frac{S}{2\eps}\int_0^t\int_\Omega |\chi \tilde{Q}^\eps|^2 + C\eps t\|\chi R^\eps\|^2_{\fsL^\infty([0, T], \fsL^2(\Omega))}.
  \end{align*}

  Hence, we deduce that
  \begin{equation*}
    \|\chi \tilde{Q}^\eps(t)\|^2_{\fsL^2(\Omega)} \ioe C\eps t.
  \end{equation*}

  We then have:
  \begin{equation*}
    \|\chi Q^\eps(t)\|^2_{\fsL^2(\Omega)} \ioe 2\|\chi \tilde{Q}^\eps(t)\|^2_{\fsL^2(\Omega)} + 2\|\chi e^{-\frac{St}{\eps}}Q_\init\|^2_{\fsL^2(\Omega)} \ioe C(\eps t + e^{-\frac{2St}{\eps}}).
  \end{equation*}
  With $t'_\eps := \frac{\eps|\ln\eps^2|}{2S}$, we obtain the desired estimate
  \begin{equation*}
    \|Q^\eps(t'_\eps)\|^2_{\fsL^2(\tilde \Omega)} \ioe C\eps^2 \lvert\ln\eps\rvert.
  \end{equation*}
\end{proof}

\subsection*{Acknowledgments}
Part of this article has been produced during a master internship, the author is deeply thankful to his advisors Laurent Desvillettes and Helge Dietert for many fruitful discussions and for having proofread earlier versions of this article.

\printbibliography
\end{document}

%% file: BaseTex.tex
\usepackage{ stmaryrd }

\newcommand*{\dd}{\mathop{}\!\mathrm{d}}

\newcommand\R{\mathbb{R}}

\newcommand\N{\mathbb{N}}

\newcommand\soe{\geqslant}
\newcommand\ioe{\leqslant}

\newcommand{\fsL}{\textnormal{L}} 
\newcommand{\fsH}{\textnormal{H}} 
\newcommand{\fsW}{\textnormal{W}} 
\newcommand{\init}{\mathrm{in}} 

\newcommand\Ccal{\ensuremath{\mathcal C}}

\newcommand\Ecal{\ensuremath{\mathcal E}}

\renewcommand\a{\alpha}

\renewcommand\phi{\varphi}
\newcommand\eps{\varepsilon}

\newtheorem{theo}{Theorem}
\newtheorem{coro}{Corollary}
\newtheorem{lemm}{Lemma}

\newtheorem{prop}{Proposition}

\newtheorem{rema}{Remark}

\newtheorem*{theo*}{Theorem}
\newtheorem*{coro*}{Corollary}
\newtheorem*{lemm*}{Lemma}
\newtheorem*{conj*}{Conjecture}
\newtheorem*{prop*}{Proposition}
\newtheorem*{defi*}{Definition}
\newtheorem*{rema*}{Remark}